\documentclass[12pt]{article}
\usepackage{float}
\usepackage{amsmath, amssymb, amsthm, latexsym,color, natbib,MnSymbol}
\usepackage{algorithmicx}
\usepackage{fancyhdr}
\usepackage{algpseudocode}
\usepackage{floatpag}

\usepackage{graphicx,subfigure,subcaption}
\usepackage{enumerate}
\usepackage{placeins}
\usepackage{natbib, makecell, booktabs}
\usepackage{color}
\usepackage{amsbsy}
\usepackage{amsmath}
\usepackage{amssymb}
\usepackage{comment}
\usepackage{amsfonts}
\usepackage{multirow,latexsym}
\usepackage{bm}
\usepackage{adjustbox}
\usepackage{ragged2e}
\usepackage[linesnumbered,ruled,vlined]{algorithm2e}

\allowdisplaybreaks[4]
\newtheorem{theorem}{Theorem}
\newtheorem{remark}{Remark}
\newtheorem{lemma}{Lemma} 
\newtheorem{corollary}{Corollary}
\newtheorem{proposition}{Proposition}

\renewcommand{\hat}{\widehat}
\def\singlespace{\def\baselinestretch{1}\@normalsize}

\def\wh{\widehat}
\def\wt{\widetilde}

\newcommand{\argmin}{{\rm argmin}}
\newcommand{\argmax}{{\rm argmax}}

\newcommand{\diag}{{\rm diag}}

\newcommand{\var}{{\rm Var}}

\def\E{{\rm E}} 

\newcommand{\bA}{{\mathbf A}}
\newcommand{\bB}{{\mathbf B}}
\newcommand{\bD}{{\mathbf D}}
\newcommand{\bF}{{\mathbf F}}

\newcommand{\bG}{{\mathbf G}}
\newcommand{\bH}{{\mathbf H}}
\newcommand{\bI}{{\mathbf I}}
\newcommand{\bK}{{\mathbf K}}
\newcommand{\bL}{{\mathbf L}}
\newcommand{\bM}{{\mathbf M}}

\newcommand{\bR}{{\mathbf R}}

\newcommand{\bU}{{\mathbf U}}
\newcommand{\bV}{{\mathbf V}}

\newcommand{\bX}{{\mathbf X}}
\newcommand{\bY}{{\mathbf Y}}
\newcommand{\bZ}{{\mathbf Z}}
\newcommand{\ba}{{\mathbf a}}
\newcommand{\bb}{{\mathbf b}}

\newcommand{\bx}{{\mathbf x}}
\newcommand{\by}{{\mathbf y}}

\newcommand{\bbeta}  {\boldsymbol{\beta}}
\newcommand{\bfeta}  {\boldsymbol{\eta}}

\newcommand{\btheta} {\boldsymbol{\theta}}

\newcommand{\bGamma} {\boldsymbol{\Gamma}}
\newcommand{\bLambda} {\boldsymbol{\Lambda}}
\newcommand{\bC}{{\mathbf C}}

\newcommand{\bnu}{\boldsymbol{\nu}}

\def\6bullets{\bullet\bullet\bullet\bullet\bullet\bullet}

\DeclareMathAlphabet\EuScriptBF{U}{eus}{b}{n}

\def \b1{{\bf 1}}
\def \f{{\bf f}}
\def \bY{{\bf Y}}

\begin{document}

\def\spacingset#1{\renewcommand{\baselinestretch}%
	{#1}\small\normalsize} \spacingset{1}
\title{\bf Dynamic Networks with Node Heterogeneity and Homophily}
\author{Binyan Jiang\thanks{Author names in alphabet order.} \hspace{.2cm}\\
	Department of Data Science and Artificial Intelligence \\
	The Hong Kong Polytechnic University\\
    \\
	Qiwei Yao \quad and \quad Xinyang Yu \\
	Department of Statistics\\
	London School of Economics and Political 
    Science}
\maketitle

\begin{abstract}
    The goal of this paper is to model node heterogeneity and link homophily for dynamic networks. The proposed framework brings new insights on how networks evolve over time. It also provides more sophisticated tools for the prediction of  future networks with statistical guarantees.  The new model accounts for the link homophily associated with both observed traits and latent traits. The joint modeling of node heterogeneity and both observed and latent homophily effects 
    also poses the significant challenge in statistical inference, resulted from the large number of confounding parameters in the model. To overcome this, we propose a novel normalized squared loss, paving the way for efficient and stable  estimation of parameters in a high-dimensional setting. We provide a rigorous theoretical analysis of the estimation method, and demonstrate its effectiveness through extensive simulations and the illustration with some real-world network data.

\end{abstract}

\section{Introduction}
The literature on network analysis is broad and spans mathematics, statistics, computer science, biology,  social sciences and other disciplines \citep{ek17}. In recent years, statistically grounded methods for network analysis have been developed.
See, for example, \cite{wiuf2006likelihood, Bickel2009, goldenberg2010survey, Graham20, Borsboom21}. However, the conceptual and methodological developments for dynamic network analysis are lagged behind, and many basic questions remain open \citep{fienberg2012brief, ek17, ANJiang}. In particular, while important stylized features such as node heterogeneity, homophily, and latent structure have been extensively studied for static networks, their joint treatment for dynamic network models is still missing. The goal of this study is to fill in this gap.

Node heterogeneity is one of the most common features of network data. In static network analysis, it is usually linked to degree heterogeneity, meaning that some nodes are systematically more active or more popular than others. This feature is often modeled through node-specific parameters, as in the classical $\beta$-model \citep{chatterjee2011} and its variations \citep{graham2017econometric, yan2019statistical, chen2021analysis, stein2021sparse}. For dynamic networks, however, node heterogeneity is richer than degree variation alone. Nodes may differ not only in their overall levels of connectivity, but also in their tendencies in forming new links and/or dissolving existing ones. This extra layer of heterogeneity, together with temporal dependence in the network process, creates major challenges for statistical inference and its associated theory \citep{jiang2025two, chang2024autoregressive}.

Link homophily refers to the tendency for nodes with similar traits (such as age, race, hobby, occupation, economic/social status and etc) to be connected \citep{MLC2001, graham2016}. Homophily  affects edge formation directly, and shapes the global structure of networks. In particular, strong homophily may lead to modular or community-like structures with dense within-group connections and relatively sparse between-group connections. 
Effective inference methods have been developed for community detection and related problems using both observed and latent traits for static networks; see, for example, \cite{yan2021covariate, hu2024network} and the references therein.  

In this paper, we propose a dynamic network model that jointly accommodates  node heterogeneity and link homophily. For the latter, we allow both observed and latent node traits. By doing so, the statistical inference for the model is inevitably challenging. First, the number of parameters in the model is large, and is typically in the order of the number of nodes. Since real networks are often large in relation to the length of the observed network sequence, this is a high-dimensional inference problem.
But unlike most high-dimensional problems, there exist no obvious low-dimensional structures such as sparsity in our model.
Secondly, due to the presence of the parameter representing different stylized feature coupled with dynamic transitions, the negative log-likelihood is highly nonconvex, somehow similar to but more complex than
 that for dynamic $\beta$-model of \cite{jiang2025two}. Such nonconvexity, especially in high-dimensional settings, can render the parameter estimation numerically unstable. Furthermore, it is hard to establish the uniform convergence for the
 estimators without further structural assumptions. To overcome those obstacles, we introduce a new normalized squared loss function which leads to a bi-convex optimization problem. We solve
 the optimization problem by a simple iterative algorithm which is computationally stable. 
 Note static formulations such as the $\beta$-model with covariates \citep{yan2019statistical} and latent space models \citep{ma2020universal} can handle degree heterogeneity and homophily in static networks.  But extending them to dynamic settings with additional node-level heterogeneity and temporal dependence is far from straightforward.

To position properly our proposal in relation to the existing literature, Table \ref{models} compares our model with several benchmarks, including the $\beta$-model with covariates of \cite{graham2017econometric, yan2019statistical},  latent space models of \cite{ma2020universal, li2023statistical},  autoregressive network model of \cite{ANJiang}, and its extension with node heterogeneity of \cite{jiang2025two}. As indicated in the table, our new model is designed to  capture node heterogeneity, observed homophily, latent homophily, and temporal network dynamics altogether, while also allowing for theoretical guarantees when the number of nodes diverges.

 \begin{table}[ht] 
\centering
\small
\caption{Capabilities of some network models}
\label{models}
\begin{tabular}{|c|c|c|c|c|c|}
\hline
\multirow{2}{*}{Model} & Node   & \multicolumn{2}{c|}{Homophily} & \multirow{2}{*}{Dynamic} & Uniform     \\ 
\cline{3-4}
& Heterogeneity & Observed & Latent & & Consistency \\
\hline
\cite{graham2017econometric, yan2019statistical} & $\checkmark$ & $\checkmark$ &  & & $\checkmark$  \\ \hline
\cite{ma2020universal} & $\checkmark$ & $\checkmark$ & $\checkmark$ &  &   \\ \hline
\cite{li2023statistical}& $\checkmark$ &  & $\checkmark$ &  & $\checkmark$   \\ \hline
\cite{ANJiang} &  &  &  & $\checkmark$ & $\checkmark$ \\ \hline
\cite{jiang2025two} & $\checkmark$ &  &  & $\checkmark$ & $\checkmark$ \\ \hline
Our model & $\checkmark$ & $\checkmark$ & $\checkmark$ & $\checkmark$ & $\checkmark$      \\ \hline
\end{tabular}
\end{table}
We summarize the main contributions of this paper below.
\begin{itemize}
    \item We propose a new dynamic network model that jointly accommodates  node heterogeneity and link homophily with both observed and latent traits. 
    
    \item We introduce a normalized squared loss, together with a suitable parameter transformation, for estimating the high-dimensional model parameters. This formulation leads to a biconvex optimization problem and supports a simple iterative estimation procedure. Compared with likelihood-based approaches, the proposed loss-based framework is computationally more stable and can be easily extended to settings with weighted or dependent edges. Moreover, the alternating scheme typically stabilizes after one or two local updates, so the computational burden is essentially that of solving a small number of convex problems and low-rank spectral thresholding steps. This efficiency enables the method to scale effectively to large networks.
    
    \item We establish theoretical guarantees for the proposed estimators obtained in the optimization algorithms. In particular, we show consistency of a simple two-step estimator under suitable initialization, provide guidance on the choice of initial values, and derive sufficient conditions for the consistency of the estimator obtained with further local iterations. We also develop a concise argument, building on the uniform deviation bounds of \cite{jiang2025two}, to establish uniform convergence for the high-dimensional node-specific parameters. This part of the analysis may also be useful in other network models with high-dimensional node-specific parameters.
\end{itemize}

The rest of the paper is organized as follows. Section \ref{sec: Homophily} introduces the proposed 
model. 
Sections \ref{sec: parameter estimation} and \ref{sec:theory} present the estimation procedure and the associated theoretical results. Section \ref{sec:numerical_study} reports simulation studies and real-data analysis of citation networks and co-authorship networks. Section \ref{sec: conclusion} concludes the paper. All technical proofs are collected in the Appendix.

We use $\bI_p$ to denote the $p\times p$ identity matrix, and ${\bf{1}}_p$ the $p\times 1$ vector with all elements equal to 1.
For vector $\ba=(a_1, \cdots, a_p)^\top$, 
let $\|\ba\|_q = ( \sum_{i=1}^p |a_i|^q )^{1/q}$ for any $q\ge 1$, and $\|\ba\|_\infty = \max_i |a_i|$. For 
$p\times p$ matrix $\bA=(A_{i,j})$, let $\|\bA\|_F$ be its Frobenius norm, 
$\|\bA\|_2$ be its spectral norm, 
$
\|\bA\|_1 = \max_{1 \le j \le p} \sum_{i=1}^p |A_{i,j}|, 
$ $
\|\bA\|_\infty = \max_{i,j} |A_{i,j}|$, $\|\bA\|_{2,\infty}$ denotes the largest $\ell_2$ norm of its rows, and
$\|\bA\|_* = \sum_{i=1}^{p} \sigma_{i}(\bA)$ where $\sigma_{i}(A)$ are the singular values of $\bA$. Denoted by $x \lesssim y$, or equivalently $y \gtrsim x$, if there exists a constant $c>0$ such that $|x| \le c|y|$. We write $x \asymp y$ when both $x \lesssim y$ and $y \lesssim x$ hold. 
We use $c, c_1, c_2, \cdots $ to denote positive generic constants which may be different at different places.

\section{Models}\label{sec: Homophily}
Consider a dynamic network with the fixed $p$ nodes labeled as $1, \cdots, p$. Let 
$\bX_t \equiv (X_{i,j}^t)$ be its 
adjacency matrix at time $t$, i.e. $\bX_t$ is
$p\times p$, and
 $X_{i,j}^t = 1$ or 0, indicating an edge between nodes $i$ and $j$ at time $t$ or no edge. 
For simplicity, we assume that the network is undirected (i.e., $X_{i,j}^t \equiv X_{j,i}^t$ for all $1\le i<j\le p$) without self-loops (i.e., $X_{i,i}^t \equiv 0$ for $i\in[p]$), though both the methods and the theory concerned can be extended to directed networks. We also assume that, for each fixed $t$, the edges $X_{i,j}^t$, $1 \le i < j \le p$, are independent. 

We consider two types of traits for homophily, i.e. observed $q\times 1$ vector $\bZ_{i,j}\equiv \bZ_{j,i}$ reflecting the similarity between nodes $i$ and $j$, and latent $d\times 1$ vector $\bU_i=(\bU_{i,1}^\top, \bU_{i,2}^\top)^\top$ representing unobservable trait of node $i$, where $\bU_{i,1}, \bU_{i,2}$ are, respectively, $d_1\times 1$ and $d_2 \times 1$, and $d_1 + d_2 =d$. For simplicity of presentation, we treat all $\bZ_{i,j}$ and $\bU_i$ as constants.  The analogue can be drawn from linear regression analysis in which all regressors are treated as constants. 
We assume that
the transition probabilities of the network admit the following forms: For $1\le i <j \le p$,
\begin{align}\label{T0}
	&P(X_{i,j}^t=1 \mid X_{i,j}^{t-1}=0) = 
    \theta_i \theta_j (\bZ_{i,j}^\top \bbeta
    + \bU_{i,1}^\top \bLambda_1 \bU_{j,1})
    \equiv
    A_{i,j},
    \quad \text{and} \\ \nonumber
	&P(X_{i,j}^t=0 \mid X_{i,j}^{t-1}=1) = 
    \eta_i \eta_j (\bZ_{i,j}^\top \bnu 
    + \bU_{i,2}^\tau \bLambda_2 \bU_{j,2})
    \equiv B_{i,j}.
\end{align}
In the model above, $\btheta = (\theta_1, \ldots, \theta_p)^\top$ and $\bfeta = (\eta_1, \ldots, \eta_p)^\top$ are two sets unknown non-negative parameters representing node heterogeneity. Formulation of node heterogeneity in this manner was adopted in the degree-corrected stochastic block models 
\citep{karrer2011stochastic, ji2022co}, and also some dynamic network models \citep{jiang2025two, chang2024autoregressive}.
The effect of the observed traits is represented in the form of linear regression with  unknown coefficient vectors
$\bbeta$ and $\bnu$. Although Model \eqref{T0} theoretically requires the parameters to reside in a constrained feasible set to ensure
$A_{i,j}, B_{i,j}\in (0,1)$, our estimation procedure utilizes unconstrained optimization as in the degree-corrected stochastic block models and the random dot product models \citep{athreya2018statistical}. We justify this 
choice through our uniform convergence results (see Theorem 3), which provides theoretical guarantees for the estimated probabilities to fall  within the interval
$[0,1]$ with probability approaching 1 as $n,p\rightarrow \infty$.   
 We impose the restriction $\|\bbeta\|_2 = \|\bnu\|_2 = 1$ to make $\btheta$ and $\bfeta$ identifiable. Further we model the latent homophily effect through two quadratic forms of latent traits $\bU_{i,1}$ and $\bU_{i,2}$, similar to the eigenmodel of \cite{hoff2008},
where $\bLambda_1$ and $\bLambda_2$ are two unknown diagonal matrices. We require 
$\|\bU_{i,1}\|= \|\bU_{i,2}\|=1$.
Then for $k=1$ or $2$, similar values of components $\bU_{i,k}$ and $\bU_{j,k}$ 
contribute positively or negatively to the transition probability between nodes $i$ and $j$, depending on the corresponding diagonal element of $\bLambda_k$ being positive or negative. Note that \cite{hoff2008} adopted the eigenmodel to define a random event of existing edge or not for a static network, and the estimation was carried out using computationally intensive methods such as MCMC. In contrast, based on the proposed normalized squared loss  all $\bU_{i,k}, d_k $ and $\bLambda_k$ $(k=1,2)$ in our model can be estimated efficiently by a penalized procedure. See Section \ref{latentG} below.

\section{Estimation}\label{sec: parameter estimation}

With available observations $\bX_1, \cdots, \bX_n$, we develop some new methods for estimating parameters $\btheta, \bfeta, \bbeta, \bnu, \bLambda_1, \bLambda_2$ as well as  latent traits $\bU_1, \cdots, \bU_p$ in model \eqref{T0}. The likelihood function is non-convex, and the maximum likelihood estimation is a high-dimensional nonlinear optimization problem with complicated constrains on the parameters (due to the constraint $A_{i,j}, B_{i,j} \in (0, 1)$). To ensure the estimation is statistically sound and numerically stable, we introduce a new normalized squared loss function. To highlight the key idea, we first consider in Section \ref{OT_homophily} a simpler case when the latent traits $\{ \bU_i, i\in [p]\}$ are absent. In this case, the proposed normalized squared loss is convex. 
The estimation for the full model \eqref{T0} will be developed in Section \ref{latentG}.
It turns out to be a biconvex optimization problem based on the normalized squared loss, and this optimization problem can be solved by a simple iterative algorithm.

\subsection{A model without latent homophily traits}\label{OT_homophily}

We first consider a reduced form of \eqref{T0} in absence of latent traits $\bU_{i,k}$:
\begin{align}
	\label{homophilyF2}
	A_{i,j}=P(X_{i,j}^t=1|X_{i,j}^{t-1}=0) = \theta_i \theta_j(\bZ_{i,j}^{\top}\bbeta+1), \\ \nonumber
	B_{i,j} =P(X_{i,j}^t=0|X_{i,j}^{t-1}=1) = \eta_i \eta_j(\bZ_{i,j}^{\top}\bnu+1).
\end{align}
Note that a different parametrization is adopted here: by inserting 1 on the RHS of equations, we do not require $\bbeta$ and $\bnu$ to be constrained to unit vectors. 

As the likelihood function is non-convex, a natural alternative is the method of moments estimation, i.e. to search for the values of $\btheta, \bfeta, \bbeta$ and $\bnu$ such that the RHS of the two equations in \eqref{homophilyF2} match some appropriate initial estimates $\wh A_{i,j}$ and $\wh B_{i,j}$ respectively. The obvious choice for the initial estimators are the relative frequencies:
\begin{equation}\label{MLE0}
	\hat{A}_{i,j}
	=
	\frac{\sum_{t=2}^n X_{i,j}^t(1-X_{i,j}^{t-1})}
	{\sum_{t=2}^n (1-X_{i,j}^{t-1})},
	\qquad
	\hat{B}_{i,j}
	=
	\frac{\sum_{t=2}^n (1-X_{i,j}^t)X_{i,j}^{t-1}}
	{\sum_{t=2}^n X_{i,j}^{t-1}}.
\end{equation}
See \cite{ANJiang}. We adopt the convention $0/0=1$ in the above expressions.  Below we only proceed to the estimation for $\btheta, \bbeta$ explicitly, as that for $\bfeta, \bnu$ follows the similar lines.

With properly specified initial estimator
$\wh A_{i,j}$, we may seek for the values
of $\btheta$ and $\bbeta$ to minimize
\begin{equation}\label{sqloss}
	\frac{1}{p}\sum_{1\le i< j\le p}
	\left\{
	\hat{A}_{i,j}-\theta_i\theta_j\left(\bZ_{i,j}^{\top}\bbeta+1\right)
	\right\}^2.
\end{equation}
However, this is not a convex function of  $(\btheta,\bbeta)$. In particular, when $p$ is large in relation to $n$, small perturbations in $\btheta$ may lead to substantial changes in the value of \eqref{sqloss}, making the minimization  difficult and potentially unstable.
Note that
the squared loss functions are often used to estimate latent homophily in static network models, such as the random dot product model \citep{athreya2018statistical, xie2023efficient} and the signal-plus-noise matrix model \citep{xie2022eigenvector,xie2024eigenvector}. However the direct use (\ref{sqloss}) is not appealing due to its non-convexity. 

To overcome the difficulties due to the non-convexity, we normalize the square loss \eqref{sqloss} as follows:
\begin{eqnarray}\label{loss_theta}
	L\left(\btheta,\bbeta \mid \hat{\bA}\right)
	&=&
	\frac{1}{p}\sum_{1\le i< j\le p}
	\left(
	\frac{\hat{A}_{i,j}}{\sqrt{\theta_i\theta_j}}
	-
\sqrt{\theta_i\theta_j}\left(\bZ_{i,j}^{\top}\bbeta+1\right)
	\right)^2 \nonumber\\
	&=&
	\frac{1}{p}\sum_{1\le i< j\le p}
	\frac{1}{\theta_i\theta_j}
	\left\{
	\hat{A}_{i,j}-\theta_i\theta_j\left(\bZ_{i,j}^{\top}\bbeta+1\right)
	\right\}^2,
\end{eqnarray}
where $\wh \bA= (\wh A_{i,j})$ is a $p\times p$ matrix with $\wh A_{i,i}\equiv 0$.
We argue that this normalization may be viewed as normalizing initial estimator $\wh A_{i,j}$ by its approximate standard error: 
Given $X_{i,j}^{t-1}=0$, $\pi \equiv A_{i,j} $ is probability $X_{i,j}^t =1$, i.e. the probability in a random trial with binary outcomes. Hence the relative frequency estimator for $\pi$ has the variance 
$
\pi(1-\pi) = A_{i,j}(1-A_{i,j}) \approx A_{i,j} $
when $A_{i,j} $ is small, which is true for most large and sparse networks. Therefore
we may take $ \sqrt{\theta_i\theta_j} \propto \sqrt{A_{i,j}}$ (see \eqref{homophilyF2}) as an approximate for the standard error of $\wh A_{i,j}$.  See also Proposition 7 in \cite{ANJiang} for the asymptotic variance of the initial estimators defined in \eqref{MLE0}.

This feature also makes the proposed loss more suitable for link prediction, since each term in the sum in \eqref{loss_theta} is rescaled according to its approximate variance. Note that  loss function $L(\btheta, \bbeta|\wh \bA)$ is different from classical weighted least squares, as the weights themselves are unknown parameters.
Now we introduce the parameter transformations:
\begin{equation*}
f_i=\frac{1}{2}\log(\theta_i), \qquad
\end{equation*}
and put $\f=(f_1,\ldots,f_p)^\top$. 
Then the loss function
\begin{equation}
    \label{loss_sample}
	l\left(\f,\bbeta \mid \hat{\bA}\right)
	\;\equiv\;L\left(\btheta,\bbeta \mid \hat{\bA}\right) \;=\;
	\frac{1}{p}\sum_{1\le i< j\le p}
	\left\{
	\hat{A}_{i,j}e^{-f_i-f_j}
	-
e^{f_i+f_j}\left(\bZ_{i,j}^{\top}\bbeta+1\right)
	\right\}^2
\end{equation}
is convex in $(\f, \bbeta)$; see Proposition \ref{Hessian_inf} below. We introduce a regularity condition first.

\begin{itemize}
	\item[(C1)] There exist positive constants $c_{1}$, $c_{2}$ and $c_{3}$ such that for all $\|\bx\|_{2}=1$,
	\begin{align*}
		c_{1}\leq \bx^{\top}\left\{ \frac{1}{p^2}\sum_{1\leq i < j\leq p}\left( \bZ_{i,j}\bZ_{i,j}^{\top}\right) \right\} \bx\leq c_{2},
		\quad \|\bZ_{i,j} \|_\infty\leq c_{3}.
	\end{align*}  
\end{itemize}

    \begin{proposition}\label{Hessian_inf}
    Let condition (C1) hold. For  any given $p\times p$ matrix $\bA$, the loss function $l\left(\f,\bbeta|\bA\right)$ is convex in $(\f, \bbeta)$. Moreover, there exists a constant $c> 0$   independent with $\bA$ such that for any $\ba\in \mathbb{R}^{p}$, $\bb\in \mathbb{R}^{q}$, 
	\begin{align*}
		(\ba^\top,\bb^{\top} )\nabla^{2}l\left(\f,\bbeta|\bA\right)
        (\ba^\top,\bb^\top)^\top
        \geq c  
       \min_{1\le i<j\le p}A_{i,j}^2 
        \left( \|\ba\|_{2}^2+p \| \bb\|_{2}^2\right).
	\end{align*}
\end{proposition}

\subsection{Estimation for model \eqref{T0}}
\label{latentG}

For the full model \eqref{T0}, we apply the same idea as in Section \ref{OT_homophily} to construct a normalized squared loss for the estimation. Unfortunately the normalized squared loss function is not entirely convex. Nevertheless its so-called bi-convexity facilitates a simple iterative algorithm which leads to numerically stable estimation. Furthermore, our estimation for latent traits are based on a direct matrix eigenanalysis, free from more computationally intensive methods such as MCMC, which are the standard practice in estimating latent homophily effects \citep{hoff2008}.
We only present the estimation for $\btheta, \bbeta, \bU_{i,1}, \bLambda_1$, and that for $\bfeta, \bnu, \bU_{i,2}, \bLambda_2$ can be proceed in the same manner.
Let $f_i = (\log \theta_i)/2$, and  $\bH_1 =(H_{i,j}^{(1)})$ be a $p\times p$ matrix with $H_{i,j}^{(1)} = \bU_{i,1}^\top \bLambda_1 \bU_{j,1}$. 
Put $\bZ_{i,i}={\bf 0}$ for $i=1, \ldots, p$. Following \eqref{T0}, the diagonal elements of $\bA$ are correspondingly defined as $A_{i,i}= \theta_i^2H_{i,i}^{(1)}$, while their estimates are set to $\hat{A}_{i,i}=0$. Similar to \eqref{loss_sample}, we define the normalized loss function as follows:

\begin{eqnarray} \label{loss_1}
l\left(\f,\bbeta,\bH_1| \wh \bA \right) 
&=& \frac{1}{2p}\sum_{1\leq i, j\leq p} \left\{ \hat{A}_{i,j}e^{-f_{i}-f_{j}}-e^{f_{i}+f_{j}} \left( \bZ_{i,j}^{\top} \bbeta  + H^{(1)}_{i,j}\right)\right\} ^2  \\
&=&\frac{1}{2p}\| \bD_f^{-1}\bA \bD_f^{-1}
- \bD_f \bZ_\beta \bD_f -\bD_f \bH_1 \bD_f\|_F^2, \nonumber 
\end{eqnarray}
where $\bD_f= \diag( e^{f_1}, \cdots, e^{f_p})$ and $\bZ_\beta=(\bZ_{i,j}^{\top} \bbeta)_{1\leq i,j\leq p}$.  

The proposition below shows that this loss function is of the so-called biconvexity which paves the way for an efficient and stable algorithm to evaluate its minimizer which will be taken as our estimator.

\begin{proposition}\label{Hessian_inf_H}
	Let condition (C1) hold.
    
    (i)  For any $p\times p$ matrices $\bA$ and $\bH=(H_{i,j})_{1\le i, j\le p}$ with $\|\bH \|_{\infty} <\infty$ and $H_{i,i}=0, i=1, \ldots, p$, the function 
  \begin{equation}
      \label{l_1}
 l_1(\f, \bbeta|\bA, \bH) 
    \equiv l(\f, \bbeta, \bH|\bA)
    \end{equation}
    is convex in $(\f, \bbeta)$, where
    $l(\cdot)$ is defined as in \eqref{loss_1}.
    More precisely,
    there exists a constant $C>0$  independent of $\bA$, such that for all $\ba\in \mathbb{R}^{p}$ and $\bb\in \mathbb{R}^{q}$,  
	\begin{align*}
		(\ba^\top,\bb^{\top}) \nabla^{2}l_1\left(\f,\bbeta|\bA,\bH\right)(\ba^\top,\bb^\top)^\top \geq C\min_{1\le i <j \le p}
        A_{i,j}^2  
        \left( \|\ba\|_{2}^2+p \| \bb\|_{2}^2\right).
	\end{align*}

    (ii) For any given  constant $\lambda>0$, $\f$, $\bbeta$, and symmetric matrix $\bA=(A_{i,j})$ with $A_{i,i}=0$, 
    \begin{align}
        \label{H}
    &{\arg\min}_{\bH} \big\{l(\f, \bbeta, \bH|\bA) + \lambda \|\bD_f \bH \bD_f\|_*\big\}\\ \nonumber
    = \; &\bD_f ^{-1} \bGamma \; \diag\big\{ 
   {\rm sgn}(e_1) (|e_1|-p\lambda)_+, \;\cdots,\; {\rm sgn}(e_{p})(|e_{p}|-p\lambda)_+
    \big\}\;\bGamma^\top \bD_f^{-1},
    \end{align}
where $x_+ = \max(0, x)$, 
and $\bGamma$ and $e_i$ are taken from the spectral decomposition of symmetric matrix  $\bD_f^{-1}\bA\bD_f^{-1} - \bD_f\bZ_\beta \bD_f = \bGamma\diag (e_1, \cdots, e_{p}) \bGamma^\top$ in which the eigenvalues are arranged in the order $|e_1| \ge \cdots \ge |e_{p}|.$
\end{proposition}

Assertion (ii) in the above proposition follows from the fact that, by substituting 
$\bL = \bD_f \bH \bD_f$, the minimization problem reduces to the following convex, nuclear norm penalized Frobenius projection:  
\begin{eqnarray*}
    l_2(\bL|\bA, \f, \bbeta) &\equiv&
l(\f, \bbeta, \bH|\bA)  + \lambda \|\bD_f \bH \bD_f\|_*  \\
&=&
\frac{1}{2p}\| \bD_f^{-1}\bA \bD_f^{-1}
- \bD_f \bZ_\beta \bD_f -\bL\|_F^2
+ \lambda\|\bL\|_*
\end{eqnarray*}
 Minimizing this convex function corresponds to evaluating the proximal operator of the nuclear norm, i.e., applying soft-thresholding to the singular values of $\bD_f^{-1}\bA\bD_f^{-1} - \bD_f\bZ_\beta \bD_f$.

Now we are ready to spell out the algorithm for computing the estimates for $\f, \bbeta$
and~$\bH_1$.

\begin{quote}
{\bf Algorithm for computing estimates $(\wh \f, \, \wh\bbeta, \, \wh\bH_1)$ }

{\bf Input:} initial estimate $\hat{\bA}=(\wh A_{i,j})$ is a $p\times p$ symmetric matrix with $\wh A_{i,i}=0$ , $q\times 1$ covariates $\bZ_{i,j}$ for $1\le i<j\le p$, a regularization parameter $\lambda>0$, a small constant $\pi >0$ which controls the convergence, and $\bH^{(0)}$ is an initial estimate for $\bH_1$ and is $p\times p$ and symmetric with diagonal elements set to be 0.

\begin{enumerate}
    \item By using a gradient descent method,
    find the minimizer
    \begin{equation}
        \label{step1}
    (\wt \f, \wt \bbeta)
    = \arg \min_{\f, \bbeta} \, l_1(\f, \bbeta|
    \wh \bA, \bH^{(0)}),
    \end{equation}
    where $l_1$ is defined as in \eqref{l_1}.

    \item Compute estimate $\wt \bH$ by 
    \eqref{H} with $(\f, \bbeta, \bA)$ replaced by $(\wt\f, \wt\bbeta, \wh\bA)$.

    \item Set $H^{(0)}= \wt \bH$, repeat Steps 1 and 2 above until the two successive values of $\wt \f$ are highly correlated with the sample correlation coefficient greater than $1-\pi$, and the two
    successive values of $\wt \bbeta$ are highly correlated with the sample correlation coefficient greater than $1-\pi$.
    \item Compute
    \[
    \hat{\bbeta}= \frac{\wt \bbeta}{\|\wt \bbeta\|_{2}} ,\quad \hat{\f}= \wt \f + \frac{\log(\|\wt \bbeta\|_{2})}{4},\quad \hat{\bH}_{1}=\frac{\wt \bH }{\|\wt \bbeta\|_{2}}.
    \]
\end{enumerate}
{\bf Output:}
	  $\hat{\f}$, $\hat{\bbeta}$ and $\hat{\bH}_{1}$.
\end{quote}

Some remarks are now in order.

\begin{remark}
    \label{algorithm-remark}
    (i) We start the algorithm with an initial estimate $\bH^{(0)}$ for $\bH_1$, as the choice for $\bH^{(0)}$ is relatively easy, see Section \ref{ini} below.
    Especially when  $\|\bH_1\|_{\mathrm{F}}$ is small in relation to $p$, we can simply set $\bH^{(0)}=\bf0$.

   (ii) Tuning parameter $\lambda>0$ controls the rank of $\wt \bH$, to prevent the overfitting (i.e. when $\lambda =0$). It is easy to see from \eqref{H} that if we set $\lambda \in [|e_{r+1}|/p, \; |e_r|/p)$, the resulting $\wt \bH$ has rank $r$, where $e_1, \cdots, e_{p}$ are the eigenvalues of
   $\bD^{-1}_{\wt\f} \wh \bA \bD^{-1}_{\wt \f} - \bD_{\wt \f} \bZ_{\wt \bbeta} \bD_{\wt \f}$ arranged in descending order according to their absolute values. 
   See also Theorem \ref{thm2} in Section \ref{sec:theory} below. 

   (iii) With a reasonably good initial estimates $\wh \bA$ (such as \eqref{MLE0}) and $\bH^{(0)}$ (see (i) above), one iteration in Step 1 leads a consistent estimator for $(\f, \bbeta)$, and further iteration in Step 2 leads a consistent estimator for $\bH_1$. This high efficiency renders the method viable for massive graphs. See Theorems \ref{thm1} and \ref{thm2} below.
\end{remark}

\section{Theoretical analysis}\label{sec:theory}

Further to Section \ref{sec: parameter estimation}, we only state the results for the
estimators of $\f, \bbeta$ and $\bH_1$ explicitly. 
 
\subsection{Consistency of the two-step estimator}\label{2step}  
 
 We introduce some regularity condition first.
Put
\[
e_{n} = \max_{1\le i<j\le p}|\hat{A}_{i,j}- A_{i,j}|,\quad  \sigma_{n}^2 = \max_{1\le i < j \le p} \var\left( \hat{A}_{i,j}\right),
\quad
D_{n}= \frac{\log(np)e_{n}}{p}+\sqrt{\frac{\log(np)}{p}}\sigma_{n} + \frac{\|\Delta_{n}\|_{1}}{p},        \]
where $
    \Delta_{n}$ as a $p\times p$ matrix with $|\E\left( \hat{A}_{i,j}- A_{i,j}\right)|$ as its $(i,j)$-th element for $i\ne j$, and the main diagonal elements equal to 0.
    
 \begin{itemize}
	\item[(C2)]  
    It holds that  $\|\f\|_{\infty } + \|\bH_1\|_{\infty } < c$, where $c>0$ is a constant independent of $p$.

\item[(C3)]
	Suppose $\hat{A}_{i,j}, 1 \le i < j \le p$, are independent. 
 Furthermore, $D_n \rightarrow 0$
	as { $np\rightarrow \infty$.}  
  
\end{itemize}

\begin{remark}
     
    Condition (C3)
     quantifies the accuracy of   the initial estimators $\hat{\bA}$. 
     \cite{ANJiang} shows that the estimates
     defined in \eqref{MLE0} satisfies 
    Condition (C3).
 
\end{remark}

	\begin{theorem}\label{thm1}
        Let Conditions (C1)-(C3) hold. For any given initial estimate $\bH^{(0)}$ with zero diagonal entries, let
       $(\wt \f, \wt \bbeta)$ be defined in \eqref{step1}. Then the two inequalities below hold with probability $1 - (np)^{-c}$ for all  $(n,p)$ such that $np$ is sufficiently large:
		\begin{eqnarray*}
			\frac{1}{\sqrt{p}}\| \left( \wt{\f},\sqrt{p}\wt{\bbeta}\right) -\left(\f, \sqrt{p}\bbeta\right) \|_{2}
		      & \leq & C D_n +  
            C\min\left\{\frac{\left\| \bH^{(0)}- \bH_1 \right\|_{1}}{p},\frac{\left\| \bH^{(0)}- \bH_1\right\|_{\mathrm{F}}}{p}\right\},\\ \nonumber
            \left\|\wt\f-{\f} \right\|_{\infty}&\leq &C\left( D_n + \frac{\left\| \bH^{(0)}- \bH_1 \right\|_{1}}{p}\right)
            ,
		\end{eqnarray*}
        where $C, c >0$ are some constants independent of $\bH^{(0)}$. Hence, both
        $\|\wt \bbeta - \bbeta\|_2$ and
        $\|\wt \f - \f\|_\infty$ converges to 0 in probability as long as $np\to \infty$ and also $\|\bH^{(0)}- \bH_1\|_1=o_P(p)$.
	\end{theorem}

   Theorem \ref{thm2} below establishes the consistency of estimator $\wt \bH$ for $\bH_1$ obtained in Step 2, i.e.  $\wt \bH$ is defined as the RHS of \eqref{H} with $(\f, \bbeta, \bA)$ replaced by $(\wt \f, \wt \bbeta, \wh \bA)$, where $(\wt \f, \wt \bbeta)$ is obtained in Step 1.
   We set $\lambda =\left\| \wt \bM -\bM \right\|_{2}\big/p$ in \eqref{H}, where
   $\wt \bM$ and $\bM$ are $p\times p$ matrices with, respectively,
   $\hat{A}_{i,j}e^{-\tilde f_{i}-\tilde f_{j}}-e^{\tilde f_{i}+ \tilde f_{j}}  \bZ_{i,j}^{\top} \wt \bbeta$
   and ${A}_{i,j}e^{- f_{i}- f_{j}}-e^{ f_{i}+  f_{j}}  \bZ_{i,j}^{\top} \bbeta$
   as the $(i,j)$-th element.

        \begin{theorem}\label{thm2} 
        Let Conditions (C1) - (C3) hold.
\begin{itemize}
        \item[(1)]
        Then the inequality below holds with probability $1-(np)^{-c}$ for all 
        $(n,p)$ such that $np$ is sufficiently large, where $c>0$ is a constant.
        \begin{eqnarray*}
                {\|\wt{\bH}-\bH_1\|_{\mathrm{F}}} 
                 &\lesssim& \left(\sqrt{p}\|\wt \f - \f\|_{2} +p\|\wt \bbeta - \bbeta\|_{2}\right)+   pD_{n}+O(1).
            \end{eqnarray*}

\item[(2)]  
           There exists an orthogonal matrix ${\bf O} \in \mathbb{R}^{d_{1}\times d_{1}}$  such that
        \begin{eqnarray*}
        		\|U_{\tilde{\bH}}  {\bf O}-U_{\bH_{1}}\|_{\mathrm{F}} 
            &\lesssim& \frac{\sqrt{p}\left\|\wt{\f}-\f \right\|_{2}  + p \left\|\wt{\bbeta}-\bbeta \right\|_{\infty}+ pD_{n}}{\lambda_{d_{1}}\left(\bH_{1}\right)},
        \end{eqnarray*} 
              where  $U_{\bH_{1}}$ and $U_{\tilde{\bH}}$ denote the matrices whose columns are the right singular vectors corresponding to the $d_1$ largest singular values of $\bH_{1}$ and $\tilde{\bH}$ respectively.
       
    	\end{itemize}  
 \end{theorem} 
        Theorem \ref{thm2} (1) indicates that $p^{-1}\|\bH_1-\tilde{\bH}\|_{\mathrm{F}} =o(1)$  as long as the initiate estimators  $\wt{\f}$ and $\wt{\bbeta}$ satisfy  $\|\wt \f - \f\|_{2}/\sqrt{p}  +  \|\wt \bbeta - \bbeta\|_{2} =o(1)$.  From Theorems \ref{thm1} we know that this can be satisfied by setting  an appropriate initial estimates $\bH^{(0)}$ such that 
       $\left\| \bH^{(0)}- \bH_1 \right\|_{ \mathrm{F}}=o(p)$. 
      Part (2) of Theorem \ref{thm2} further establishes 
      an upper bound for the estimation error of the embedding $U_{\bH_{1}}$. We remark that if $\bH_1$ is the probability formation matrix of a
        Erdos Renyi graph, or a stochastic block model with $O(p)$ nodes in each community, the 
         eigen gap $\lambda_{d_{1}}(\bH_{1})$ is exactly of order $p$, and we would have $\|U_{\tilde{\bH}}  {\bf O}-U_{\bH_{1}}\|_{\mathrm{F}}=O(\left\|\wt{\f}-\f \right\|_{2}/\sqrt{p} + \left\|\wt{\bbeta}-\bbeta \right\|_{\infty}+ D_{n})$. 
        Consequently,  we can show that the two-step estimators $(\hat{\f}, \hat{\bbeta}, \hat{\bH})$ are  consistent when $D_{n}= \frac{\log(np)e_{n}}{p}+\sqrt{\frac{\log(np)}{p}}\sigma_{n} + \frac{\|\Delta_{n}\|_{1}}{p}\rightarrow 0$:  
    \begin{theorem}\label{thm3}
    	Assume that Conditions (C1)-(C3) hold and  $\|\bH_1\|_{\infty}<\infty$. For any given initial estimate $\bH^{(0)}$ with zero diagonal entries and satisfies $\|\bH^{(0)}- \bH_1\|_1=O_P(D_n)$, by choosing $\lambda\asymp D_{n}$, the inequality below holds with probability $1-(np)^{-c}$ for all sufficiently large and some constant $c>0$:
    	\begin{eqnarray}\label{b}
    		\max\left\{\| \hat{\f}-\f \|_{\infty}, \| \hat{\bbeta}- \bbeta\|_{\infty},\frac{\|\bH_{1}-\hat{\bH}\|_{\mathrm{F}}}{p}\right\} 
    		&\lesssim &   D_{n} + O\left(\frac{1}{p}\right).
    	\end{eqnarray}
    \end{theorem} 
    \begin{remark}
      Note that $D_{n}=\frac{\log(np)e_{n}}{p}+\sqrt{\frac{\log(np)}{p}}\sigma_{n} + \frac{\|\Delta_{n}\|_{1}}{p} \rightarrow 0$ under Condition (C3), which quantify the quality of the initial estimator $\hat{\bA}$.  The terms  $ \frac{\|\Delta_{n}\|_{1}}{p}$ and $\sqrt{\frac{\log(p)}{p}}\sigma_{n}$ account for the effects of bias and variance on the estimation, respectively.
     We reiterate that the consistency of the two-stage estimator does not strictly require $\hat{\bA}$ to be consistent. Provided $\hat{A}_{i,j}\in [0,1]$, both $e_n$ and $\sigma_n$ are bounded; consequently,   $\frac{\log(np)e_{n}}{p}+\sqrt{\frac{\log(np)}{p}}\sigma_{n}\rightarrow 0$ as   $p\rightarrow \infty$. The consistency then depends only on the condition that $\frac{\|\Delta_{n}\|_{1}}{p}\rightarrow 0$.
  When $\hat{A}_{i,j}$ is an unbiased estimator of $A_{i,j}$, we have $\max\{e_n,\sigma_n \}=o(1)$, $\|\Delta_{n}\|_{\infty}=0$ and $D_{n}=O\left( \sqrt{\frac{\log(np)}{p}}\left(\sigma_n+\sqrt{\frac{\log(np)}{p}}e_n\right)\right)$. The term $\sqrt{\frac{\log(np)}{p}}$ is consistent with the classical rate in the static single-network setting, while $\sigma_n+\sqrt{\frac{\log (p)}{p}}e_n$ characterizes how this rate how this rate is refined by the estimation accuracy of $\hat{\bA}$. 
    \end{remark}
 \subsection{Initial point selection}\label{ini}
Now we discuss how to choose a good enough initial for the estimation consistency established in Section 4.1.  
 We consider two  scaling regimes: (i) $\left\| \bH_1\right\|_{\mathrm{F}}=O(pD_{n})$, and (ii) $\left\| \bH_1\right\|_{\mathrm{F}} \gtrsim pD_{n}$. Notably, as $ D_{n} \to 0 $, these regimes overlap asymptotically.
  
Under Regime (i), we can simply set $\bH^{(0)}= {\bf{0}}_{p,p}$. Under Regime (ii), we may initiate $(\f,\bbeta)$ through below steps.
        \begin{itemize}
            \item Step 1: $\hat{\bK}:=  \argmin_{\bY}\frac{1}{2p}\|\hat{\bA}  -  \bY\|_{\mathrm{F}}^2+ \lambda \|\bY\|_*$.
            \item Step 2: $\left( \hat{\f}^{(0)}, \hat{\bbeta}^{(0)}\right):= \argmin_{\f,\bbeta} \frac{1}{2p}\sum_{1\leq i\neq j\leq p}\left(\left( \hat{A}_{i,j} - \hat{K}_{i,j}\right)e^{-f_{i}-f_{j}}-e^{f_{i}+f_{j}}    \bZ_{i,j}^\top \beta_{k}\right)^2$.
            \item Step 3: Return $\bH^{(0)}=(H^{(0)}_{i,j})$ by setting $\bH^{(0)}_{i,i}=0$ for $i=1,\ldots, p$, and $\bH^{(0)}_{i,j}=e^{-2\hat{f}_i^{(0)}-2\hat{f}_j^{(0)}}\hat{K}_{i,j}$ for $1\le i\ne j \le p$. Here $\hat{K}_{i,j}$ is the $(i,j)$th element of $\hat{\bK}$ and $\hat{f}_i^{(0)}$ is the $i$th element of $\hat{\f}^{(0)}$.
        \end{itemize} 
 As shown in  Theorem \ref{thm1}, $\tilde{\f}$ is consistent under the normalized $\ell_2$ norm if the initial $\bH^{(0)}$ satisfies $\frac{\|\bH^{(0)}-\bH_1\|_F}{p}=o(1)$, while uniform consistency requires  $\frac{\|\bH^{(0)}-\bH_1\|_1}{p}=o(1)$. 
We next show that these requirements are satisfied under the following regularity conditions:

 \begin{itemize}
     	\item[(C4)] $\|\bZ_\beta\|_2=O_p(pD_n)$.

 \item[(C5)]  
      $\bH_1$ is $\mu$-incoherent in that 
    $\|U_{\bH_1}\|_{2,\infty}
    \le
    \sqrt{\frac{\mu d_1}{p}}$ for some constant $\mu$.  

\item[(C6)]
  The projections of $\bZ_\beta$ and the bias term \(\Delta_n\) on the singular vectors $U_{\bH_1}$
    satisfy the projected row-noise bounds
    \[
    \|\bZ_\beta U_{\bH_1}\|_{2,\infty}
    \lesssim
    \sqrt{\frac{\log(p)}{p}}\|\bZ_\beta\|_2, \quad \|\Delta_n  U_{\bH_1}\|_{2,\infty}
    \lesssim
    \sqrt{\frac{\log(p)}{p}}\|\Delta_n\|_2.
    \]
\end{itemize}
Condition (C4) ensures that the signal from the observed covariates does not spectrally dominate or distort the low-rank structure of the latent homophily matrix, and is naturally satisfied under standard centered and independent covariate structures. Conditions (C5) and (C6) are further required for establishing the uniform consistency. 
Condition (C5) is the so-called $\mu-$incoherence condition, which is widely used in the low-rank matrix recovery literature \citep{candes2012exact,candes2011robust,chen2015incoherence,chi2019nonconvex}. It requires the eigenspaces of \(\bH_1\) to be delocalized with respect to the standard basis, meaning that its singular vectors cannot be concentrated on a small number of coordinates. This rules out  pathological cases where \(\bH_1\) is spiky (i.e., coherent).
In such cases, even a small perturbation in the Frobenius norm could lead to a disproportionately large error in the induced 
1-norm, making the requirement $\frac{\|\bH^{(0)}-\bH_1\|_1}{p}=o(1)$ for uniform consistency impossible to satisfy.
 Condition (C6) is a projected row-noise condition ensuring that the covariate effects and estimation biases do not highly concentrate on any single node when projected onto the latent space. This delocalization property is essential for establishing the uniform consistency of the node-specific parameters.

Now, under Condition (C4) we can establish that $\frac{\|\bH^{(0)}-\bH_1\|_F}{p}\asymp\frac{\|\bH_1\|_F}{p}\asymp \frac{1}{\sqrt{p}}\| \left( \wt{\f},\sqrt{p}\wt{\bbeta}\right) -\left(\f, \sqrt{p}\bbeta\right) \|_{2} =O(D_n)$ holds for both regimes  (c.f. equation  \eqref{L2consistency}). 
Furthermore, under the $\mu$-incoherence Condition (C5), we have  $\frac{\|\bH^{(0)}-\bH_1\|_1}{p}= \frac{\| \bH_1\|_1}{p}\asymp \frac{\| \bH_1\|_F}{p}=O(D_n)$ for Regime (i). Next, The following proposition further shows that this key requirement for uniform consistency is also satisfied under Regime (ii):  
        \begin{proposition}\label{prop3}
         Let Conditions (C1)-(C6) hold. Under Regime (ii):  $\left\| \bH_1\right\|_{\mathrm{F}} \gtrsim pD_{n}$, by setting $\lambda \asymp D_{n}$, we have, probability $1-(np)^{-c}$ for all $(n,p)$ for some large enough constant $c>0$,  
          \begin{align*}
		  \frac{\|\bH^{(0)}-\bH_1\|_1}{p}= O\left(D_{n}\right)+O\left(\frac{1}{p}\right). 
		  \end{align*}
          Moreover, for the corresponding estimators $(\hat{\f}^{(0)},\hat{\bbeta}^{(0)})$ obtained in the initial Step 2, we have:
         \begin{align*}
		    	\| \left( \hat{\f}^{(0)},\hat{\bbeta}^{(0)}\right) -\left(\f, \bbeta\right) \|_{\infty} =O\left(D_{n}\right)+O\left(\frac{1}{p}\right).
		  \end{align*}
        \end{proposition} 
In practice, we can simply initialize $\bH^{(0)}$ as the zero matrix ${\bf 0}_{p\times p}$ or use the estimator obtained from Steps 1–3 under Regime (ii) described above, avoiding the need for multiple starting values. Although the theoretical conditions can be challenging to verify in practice, our results above provides a strong theoretical foundation for a highly efficient, two-initialization strategy, eliminating the need for heuristic multi-start methods.

\subsection{Extension to edge-dependent case }\label{Depen}
Due to the simplicity of the proposed normalized least squared loss, in addition to the efficient estimation of the node-heterogeneity parameters and homophily effects simpler, it also allows us to consider more complicated structures in the network formations. In this section, we extend the theoretical results to the case where the edges are allowed to be dependent. 

\begin{corollary}\label{cor2}
    Let Conditions (C1) and (C2) hold, $\|\bH_{1}\|_{\infty}<\infty$ and $D_n \rightarrow 0$ as $np\rightarrow \infty$. 
    Then the inequality below holds with probability $1-(np)^{-c}$ for all $(n,p)$ such that $np$ is sufficiently large.
		\begin{eqnarray*}
		\frac{1}{\sqrt{p}}\| \left( \hat{\f},\sqrt{p}\hat{\bbeta}\right) -\left(\f, \sqrt{p}\bbeta\right) \|_{2}
			&\leq&  C\left( \frac{\| \hat{\bA}- \bA\|_{\mathrm{F}}}{p} + \frac{\left\| \bH^{(0)}- \bH_{1}\right\|_{\mathrm{F}}}{p}\right),\\
            \left\|\f-\hat{\f} \right\|_{\infty}&\leq &C\left( \frac{\| \hat{\bA}- \bA\|_{1}}{p} + \frac{\left\| \bH^{(0)}- \bH_{1}\right\|_{1}}{p}\right),
		\end{eqnarray*}
    uniformly for all $\bH^{(0)}$, where   $C$ and $c$  are positive constants independent of $\bH^{(0)}$.
\end{corollary}
\begin{remark}
   Corollary \ref{cor2} demonstrates that the consistency of estimators $\{\hat{\f},\hat{\bbeta}\}$ can be achieved without imposing independence assumptions on $\hat{A}_{i,j}$'s. Instead, the core requirement lies in the overall consistency in the sense that $\|\hat{\bA}- \bA\|_{\mathrm{F}}=o(p)$. Notably, the requirement $\|\hat{\bA}- \bA\|_{\mathrm{F}}=o(p)$ allows the $\hat{A}_{i,j}$'s to be dependent. For example, this would hold when the $\hat{A}_{i,j}$'s are consistent and weakly dependent. 
\end{remark}
 Similar, based on Corollary\ref{cor2} we can also extend the error bounds of $\hat{\bH}$ in Theorem \ref{thm2} to the case where the independence assumption in (C3) not holds.
\begin{corollary}\label{cor3}
    Let Conditions (C1) and (C2) hold, $\|\bH_{1}\|_{\infty}<\infty$ and $D_n \rightarrow 0$ as $np\rightarrow \infty$.         Then the inequality below holds with probability $1-(np)^{-c}$ for all $(n,p)$ such that $np$ is sufficiently large, where $c>0$ is a constant.
    	\begin{eqnarray*}
    		\max\left\{\| \left( \hat{\f},\hat{\bbeta}\right) -\left(\f, \bbeta\right) \|_{\infty},\frac{\|\bH_{1}-\hat{\bH}\|_{\mathrm{F}}}{p}\right\} 
    		&= & O\left( \frac{\| \hat{\bA}- \bA\|_{\mathrm{1}}}{p} \right) .
    	\end{eqnarray*}
\end{corollary}

\section{Numerical studies}\label{sec:numerical_study}

In this section, we assess the finite-sample performance of the proposed estimators through simulations and two real network datasets, namely the Citation networks and the Co-authorship networks.

\subsection{Simulation study: parameter estimation}

We first investigate the estimation accuracy of the proposed method under different combinations of the temporal sample size $n$ and the network size $p$. Specifically, we consider $n \in \{100, 500\}$ and $p \in \{150, 300\}$. The model parameters are generated according to the following settings:
\begin{itemize}
    \item[$S_{1,\f}$:] The node heterogeneity parameters are set to $\btheta = \mathbf{1}_{p}$. The entries of $\bbeta \in \mathbb{R}^{3}$ are generated independently from the uniform distribution $U(-1,1)$ and then rescaled so that $\|\bbeta\|_{2}=1$.
    
    \item[$S_{2,\f}$:] The node heterogeneity parameters are set to $\btheta = 0.5\,\mathbf{1}_{p}$. The entries of $\bbeta \in \mathbb{R}^{3}$ are generated independently from the uniform distribution $U(-1,1)$ and then rescaled so that $\|\bbeta\|_{2}=1$.
    
    \item[$S_{1,\bH}$:] The latent homophily effect $\bH_{1}$ is generated from the expected adjacency matrix of a stochastic block model  with $k=3$ communities:
    \[
    \bH_{1} := \bC \bG \bC^\top,
    \]
    where $\bC \in \mathbb{R}^{p \times k}$ is the community membership matrix and $\bG \in \mathbb{R}^{k \times k}$ is the block probability matrix. In the experiment, the $p$ nodes are partitioned into three communities of equal size, and
    \[
    \bG=
    \begin{bmatrix}
        0.05 & 0.01 & 0.01\\
        0.01 & 0.05 & 0.01\\
        0.01 & 0.01 & 0.05
    \end{bmatrix}.
    \]
    
    \item[$S_{2,\bH}$:] The latent homophily effect $\bH_{1}$ is generated from the expected adjacency matrix of a stochastic block model with $k=3$ communities:
    \[
    \bH_{1} := \bC \bG \bC^\top,
    \]
    where $\bC \in \mathbb{R}^{p \times k}$ and $\bG \in \mathbb{R}^{k \times k}$. Again, the $p$ nodes are partitioned into three equal-sized communities, and
    \[
    \bG=
    \begin{bmatrix}
        0.25 & 0.10 & 0.10\\
        0.10 & 0.20 & 0.10\\
        0.10 & 0.10 & 0.15
    \end{bmatrix}.
    \]
\end{itemize}

The four simulation models are summarized in Table~\ref{setting}. For each model, the formation and dissolution networks are generated according to the corresponding parameter settings for $\bA$ and $\bB$. For example, under Model~1, the formation parameters are generated from $S_{1,\f}$ and the dissolution parameters are generated from $S_{2,\f}$, while both latent homophily matrices are generated according to $S_{1,\bH}$.
  
\begin{table}[]
\centering
\caption{Simulation settings for four models. The columns labeled ``$\mathbf{A}$" and ``$\mathbf{B}$" specify the settings for matrices $\mathbf{A}$ and $\mathbf{B}$, respectively. The ``$\mathbf{A}_{\rm stat}$" and ``$\mathbf{B}_{\rm stat}$" columns report the (min, mean, max) of the non-zero elements in matrices $\mathbf{A}$ and $\mathbf{B}$, respectively. Models~3--4 serve as oracle benchmarks, which isolate the estimation error attributable to the second-stage procedure from the error induced by estimating $\bA$ and $\bB$. A checkmark ($\checkmark$) in the ``Oracle" column indicates that $\hat{\bA}$ and $\hat{\bB}$ are set to be the  real matrices $\mathbf{A}$ and $\mathbf{B}$ in the estimation; otherwise, the estimators $\hat{\mathbf{A}}$ and $\hat{\mathbf{B}}$ defined in \eqref{MLE0} were used.}
\label{setting}
\begin{tabular}{|c|c|c|c|c|c|}
\hline
Model & $\bA$ & $\bB$ & $\bA_{\rm stat}$  & $\bB_{\rm stat}$ & Oracle \\ \hline
Model 1 & $S_{1,\f}$ + $S_{1,\bH}$ & $S_{2,\f}$ + $S_{1,\bH}$ & (0.01,0.10,0.27) & $(0.003,0.03,0.08)$ &  \\ \hline
Model 2 & $S_{1,\f}$ + $S_{2,\bH}$ & $S_{2,\f}$ + $S_{2,\bH}$ & (0.10,0.24,0.61) & (0.03,0.07,0.17) &  \\ \hline
Model 3 & $S_{1,\f}$ + $S_{1,\bH}$ & $S_{2,\f}$ + $S_{1,\bH}$ & (0.01,0.10,0.27) & $(0.003,0.03,0.08)$ & $\checkmark$ \\ \hline
Model 4 & $S_{1,\f}$ + $S_{2,\bH}$ & $S_{2,\f}$ + $S_{2,\bH}$ & (0.10,0.24,0.61) & (0.03,0.07,0.17) & $\checkmark$ \\ \hline
\end{tabular}
\end{table}

As indicated in the last column of Table~\ref{setting}, for Models~1--2 we initialize $\hat{\bA}$ and $\hat{\bB}$ using the estimators in \eqref{MLE0}, whereas for Models~3--4 we set $\hat{\bA}=\bA$ and $\hat{\bB}=\bB$. In all experiments, the tuning parameter is chosen as
\[
\lambda = \sqrt{\frac{\log(np)}{np}}.
\]
Based on this choice, we compute the two-stage estimators $\{\hat{\f}, \hat{\bbeta}, \hat{\bH}^{(0)}\}$ and define
\[
\hat{\btheta}:=\bigl(\exp(2\hat{f}_{i})\bigr)_{1\le i\le p}.
\]
All experiments are repeated 100 times.

Figure~\ref{fig:loss} displays the objective values along the local iterations. Across all models and ((n,p)) settings, the loss decreases substantially after the first update and then changes only marginally in later iterations. This suggests that the two-step estimator is already close to a stable local solution in these experiments. Accordingly, in the numerical implementation we terminate the algorithm after at most two local updates of Steps 1--2, unless the stopping criterion is reached earlier. This convention keeps the computation simple while producing estimates that are essentially indistinguishable from those obtained by further iterations.

\begin{figure}[ht]
    \centering
    \includegraphics[width=0.8\linewidth]{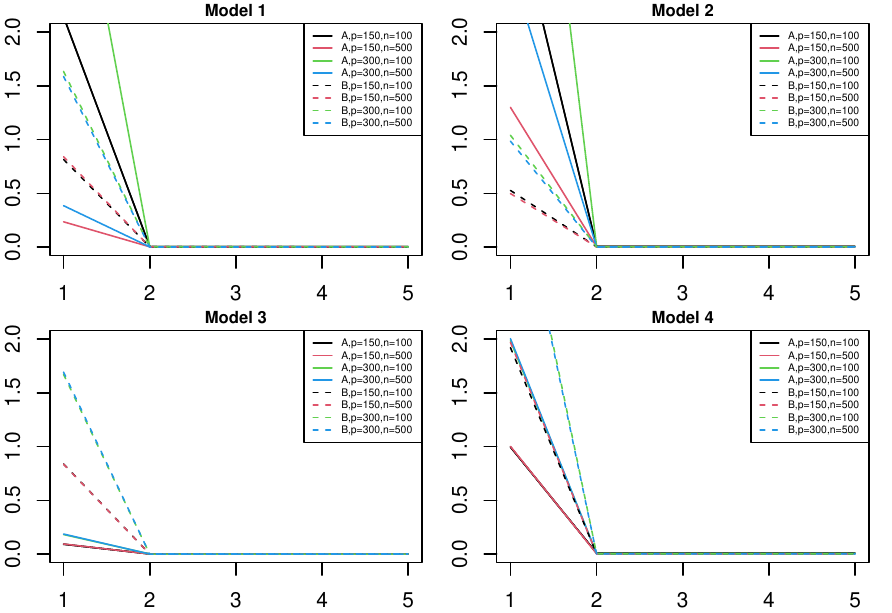}
    \caption{Values of the loss function across iterations under different models and $(n,p)$ settings.}
    \label{fig:loss}
\end{figure}

The estimation results for the estimators $(\hat{\btheta}, \hat{\bbeta}, \hat{\bH}^{(0)})$ and $(\hat{\bfeta}, \hat{\bnu}, \hat{\bH}^{(1)})$ are reported in Table~\ref{tab:acc}. The corresponding results for the two-stage estimators are very similar and are therefore omitted for brevity.
 
 Specifically, the table reports the mean and standard deviation of the following metrics: 
 \begin{itemize}
     \item ${\rm Error_{nh}}$: the $\ell_{\infty}$ norm error of the estimators for node heterogeneity. For the formation parameters in $\bA$,  ${\rm Error_{nh}}=\|\hat{\btheta}-\btheta\|_\infty$, and for the dissolution parameters in $\bB$,  ${\rm Error_{nh}}=\|\hat{\bfeta}-\bfeta\|_\infty$. 
     \item ${\rm Error_{co}}$: the $\ell_{\infty}$ norm errorof the covariate coefficient estimator. For the formation parameters in $\bA$,  ${\rm Error_{co}}= \|\hat{\bbeta}-\bbeta\|_\infty$, and for the dissolution parameters in $\bB$,  ${\rm Error_{co}}=\|\hat{\bnu}-\bnu\|_\infty$.
     \item ${\rm Error_{\lambda}}$: maximum error over all eigenvalues of the latent homophily estimator $\hat{\bH}$, defined as ${\rm Error_{\lambda}} := \max_{k}|\lambda_{k}(\bH_{1})-\lambda_{k}(\hat{\bH})|$.
     \item ${\rm Error_{H}}$: the normalized Frobenius norm error of $\hat{\bH}$ defined as $\|\mathbf{H}_{1} - \hat{\mathbf{H}} \|_{ \mathrm{F}}/p$ under different settings.
 \end{itemize}

Table~\ref{tab:acc} reports the estimation errors for the node heterogeneity parameters, the covariate coefficients, and the latent homophily matrices. Several patterns can be observed. First, increasing the temporal sample size from (n=100) to (n=500) generally improves the estimation of the node-specific parameters and the covariate effects, especially in Models 1--2 where $\widehat A$ and $\widehat B$ are obtained from empirical transition frequencies. This is consistent with the fact that the first-stage estimation error of the transition probabilities is an important source of uncertainty. Second, the normalized Frobenius error $|H_1-\widehat H|_F/p$ remains small across the simulation settings and is typically reduced when more temporal observations are available. Third, the oracle settings in Models 3--4 lead to substantially smaller errors for the node heterogeneity and covariate parameters, indicating that much of the finite-sample error in Models 1--2 is attributable to the estimation of (A) and (B), rather than to the second-stage optimization procedure itself.

The eigenvalue error $ \max_k |\lambda_k(H_1)-\lambda_k(\widehat H)|$ should be interpreted with some care. Since the leading eigenvalues of the block-structured latent homophily matrix scale with (p), the raw eigenvalue error is not directly comparable across different network sizes. For this reason, we mainly use it as a diagnostic for the recovery of the leading latent structure within each fixed-(p) setting, while the normalized Frobenius error provides a more comparable measure of the overall matrix estimation accuracy. Overall, the simulation results support the theoretical prediction that the proposed estimator is stable when the initial transition-probability estimates are sufficiently accurate, and they also illustrate the finite-sample impact of estimating these transition probabilities from a limited number of network snapshots.

\begin{table}[]
\scriptsize
\centering
\caption{Mean and standard deviation of the estimation errors under Models 1-4. ${\rm Error_{nh}}$: for part $\bA$,  $\|\hat{\btheta} -\btheta\|_\infty$, and for part $\bB$,  $\|\hat{\bfeta}-\bfeta\|_\infty$; ${\rm Error_{co}}$: for part $\bA$,  $\|\hat{\bbeta}-\bbeta\|_\infty$, and for part  $\bB$,  $\|\hat{\bnu}-\bnu\|_\infty$; ${\rm Error_{\lambda}}= \max_{k}|\lambda_{k}(\bH_{1})-\lambda_{k}(\hat{\bH})|$ and $\|\mathbf{H}_{1} - \hat{\mathbf{H}} \|_{ \mathrm{F}}/p$.} 
\setlength{\tabcolsep}{0.4pt}
\label{tab:acc}
\begin{tabular}{|cccccccccc|}
\hline
\multicolumn{10}{|c|}{Model 1} \\ \hline
p & \multicolumn{1}{c|}{part} & \multicolumn{4}{c|}{n=100} & \multicolumn{4}{c|}{n=500} \\ \cline{3-10} 
 & \multicolumn{1}{c|}{} & \multicolumn{1}{c}{${\rm Error_{nh}}$} & ${\rm Error_{co}}$ & ${\rm Error_{\lambda}}$ & \multicolumn{1}{c|}{${\rm Error_{H}}$} & {${\rm Error_{nh}}$} & ${\rm Error_{co}}$ & ${\rm Error_{\lambda}}$ & \multicolumn{1}{c|}{${\rm Error_{H}}$} \\ \hline
\multirow{2}{*}{150} & \multicolumn{1}{c|}{$\bA$} & 3.56(1.08) & 1.27(0.06) & 13.82(0.69) & \multicolumn{1}{c|}{0.10($<0.01$)} & 0.47(0.16) & 0.05(0.06) & 1.79(0.18) & 0.04($<0.01$) \\
 & \multicolumn{1}{c|}{$\bB$} & 1.30(0.74) & 1.32(0.07) & 14.61(0.87) & \multicolumn{1}{c|}{0.11(0.02)} & 0.79(1.77) & 0.15(0.06) & 17.18(1.68) & 0.12(0.01) \\
\multirow{2}{*}{300} & \multicolumn{1}{c|}{A} & 0.94(1.12) & 1.18(0.04) & 36.11(8.43) & \multicolumn{1}{c|}{0.13(0.03)} & 0.28(0.07) & 0.02($<0.01$) & 3.55(0.09) & 0.04($<0.01$) \\
 & \multicolumn{1}{c|}{$\bB$} & 3.67(0.67) & 1.29(0.07) & 28.76(0.65) & \multicolumn{1}{c|}{0.11($<0.01$)} & 0.24($<0.01$) & 0.14(0.10) & 18.60(0.84) & 0.08($<0.01$) \\ \hline
                     \multicolumn{10}{|c|}{Model 2} \\ \hline
p & \multicolumn{1}{c|}{part} & \multicolumn{4}{c|}{n=100} & \multicolumn{4}{c|}{n=500} \\ \cline{3-10} 
 & \multicolumn{1}{c|}{} & {${\rm Error_{nh}}$} & ${\rm Error_{co}}$ & ${\rm Error_{\lambda}}$ & \multicolumn{1}{c|}{${\rm Error_{H}}$} & {${\rm Error_{nh}}$} & ${\rm Error_{co}}$ & ${\rm Error_{\lambda}}$ & \multicolumn{1}{c|}{${\rm Error_{H}}$} \\ \hline
\multirow{2}{*}{150} & \multicolumn{1}{c|}{$\bA$} & 1.11(0.10) & 0.12(0.06) & 17.37(0.19) & \multicolumn{1}{c|}{0.14($<0.01$)} & 0.81(0.04) & 0.15(0.06) & 17.23(0.10) & 0.13($<0.01$) \\
 & \multicolumn{1}{c|}{$\bB$} & 4.61(2.87) & 1.31(0.08) & 5.77(0.08) & \multicolumn{1}{c|}{0.07($<0.01$)} & 0.47(0.02) & 0.10(0.05) & 13.25(0.40) & 0.10($<0.01$) \\
\multirow{2}{*}{300} & \multicolumn{1}{c|}{$\bA$} & 1.02(0.06) & 0.12(0.07) & 35.22(0.22) & \multicolumn{1}{c|}{0.15($<0.01$)} & 0.82(0.03) & 0.16(0.07) & 34.59(0.13) & 0.13($<0.01$) \\
 & \multicolumn{1}{c|}{$\bB$} & 1.97(1.68) & 1.00(0.46) & 5.77(0.08) & \multicolumn{1}{c|}{0.01(0.02)} & 0.49(0.02) & 0.10(0.04) & 27.19(0.77) & 0.10($<0.01$) \\ \hline
 \multicolumn{10}{|c|}{Model 3} \\ \hline
p & \multicolumn{1}{c|}{part} & \multicolumn{4}{c|}{n=100} & \multicolumn{4}{c|}{n=500} \\ \cline{3-10} 
 & \multicolumn{1}{c|}{} & {${\rm Error_{nh}}$} & ${\rm Error_{co}}$ & ${\rm Error_{\lambda}}$ & \multicolumn{1}{c|}{${\rm Error_{H}}$} & {${\rm Error_{nh}}$} & ${\rm Error_{co}}$ & ${\rm Error_{\lambda}}$ & \multicolumn{1}{c|}{${\rm Error_{H}}$} \\ \hline
\multirow{2}{*}{150} & \multicolumn{1}{c|}{$\bA$} & 0.04($<0.01$) & $<0.01(<0.01)$ & 0.66(0.09) & \multicolumn{1}{c|}{$<0.01(<0.01)$} & 0.04($<0.01$) & $<0.01(<0.01)$ & 0.63(0.08) & $<0.01(<0.01)$ \\
 & \multicolumn{1}{c|}{$\bB$} & 0.15(0.01) & 0.09(0.05) & 4.89(0.25) & \multicolumn{1}{c|}{0.03(0.01)} & 0.17($<0.01$) & 0.11(0.05) & 5.71(0.22) & 0.04(0.01) \\
\multirow{2}{*}{300} & \multicolumn{1}{c|}{$\bA$} & 0.04($<0.01$) & $<0.01(<0.01)$ & 1.30(0.16) & \multicolumn{1}{c|}{$<0.01(<0.01)$} & 0.04($<0.01$) & $<0.01(<0.01)$ & 1.29(0.18) & $<0.01(<0.01)$ \\
 & \multicolumn{1}{c|}{$\bB$} & 0.17($<0.01$) & 0.10(0.05) & 11.42(0.33) & \multicolumn{1}{c|}{0.04(0.01)} & 0.18($<0.01$) & 0.10(0.05) & 11.96(0.31) & 0.04(0.01) \\ \hline
 \multicolumn{10}{|c|}{Model 4} \\ \hline
p & \multicolumn{1}{c|}{part} & \multicolumn{4}{c|}{n=100} & \multicolumn{4}{c|}{n=500} \\ \cline{3-10} 
 & \multicolumn{1}{c|}{} & {${\rm Error_{nh}}$} & ${\rm Error_{co}}$ & ${\rm Error_{\lambda}}$ & \multicolumn{1}{c|}{${\rm Error_{H}}$} & {${\rm Error_{nh}}$} & ${\rm Error_{co}}$ & ${\rm Error_{\lambda}}$ & \multicolumn{1}{c|}{${\rm Error_{H}}$} \\ \hline
\multirow{2}{*}{150} & \multicolumn{1}{c|}{$\bA$} & 0.06($<0.01$) & 0.03($<0.01$) & 0.29(0.04) & \multicolumn{1}{c|}{0.02($<0.01$)} & 0.07($<0.01$) & 0.02(0.01) & 0.42(0.05) & 0.03($<0.01$) \\
 & \multicolumn{1}{c|}{$\bB$} & 0.14(0.02) & 0.07(0.03) & 1.16(0.24) & \multicolumn{1}{c|}{0.08(0.02)} & 0.14(0.03) & 0.07(0.03) & 1.20(0.28) & 0.08(0.02) \\
\multirow{2}{*}{300} & \multicolumn{1}{c|}{$\bA$} & 0.07($<0.01$) & 0.03($<0.01$) & 0.84(0.11) & \multicolumn{1}{c|}{0.03($<0.01$)} & 0.08($<0.01$) & 0.03($<0.01$) & 1.09(0.18) & 0.04($<0.01$) \\
 & \multicolumn{1}{c|}{$\bB$} & 0.15(0.02) & 0.07(0.02) & 2.60(0.46) & \multicolumn{1}{c|}{0.09(0.02)} & 0.15(0.02) & 0.06(0.03) & 2.85(0.37) & 0.10(0.01) \\ \hline
\end{tabular}
\end{table}

\subsection{Real data analysis}
\subsubsection{MADStat dataset}
We applied the proposed method to the MADStat dataset \citep{ji2022co}, which contains co-authorship and co-citation relationships constructed from 83,331 articles published over 41 years in 36 leading journals in statistics, probability, and machine learning. We consider the following two network sequences.

\begin{itemize}
    \item \textbf{Co-citation networks.} 
    Following \cite{ji2022co}, we consider
    the 21 overlapping aggregated networks constructed from rolling time windows between 1991 and 2015. 
 An edge between two authors indicates that they were co-cited at least once by someone else in the time window concerned. Starting from a total of 2,183 authors, we removed those with fewer than five edges over the whole period, resulting in $p=576$ nodes in the networks.
    
    \item \textbf{Co-authorship networks.}  
    This subset consists of 41 annual networks from 1974 to 2015. An edge between two authors indicates that they co-authored at least one paper in a given year. Again starting from the same set of 2,183 authors, we removed those with fewer than five edges over the study period, yielding $p=466$ nodes.
\end{itemize}

We applied the proposed model (2.1) without $\bZ_{i,j}$
to both the co-authorship and co-citation network data, to explore their dynamic dependence structure. In particular, With the plug-in estimators $\hat{\bA}$ and $\hat{\bB}$ defined in \eqref{MLE0}, we  estimated the model parameters using the procedure described in Section~\ref{latentG}. Our goal here is not only to fit the dynamic network models, but also to assess whether the estimated latent homophily components yield scientifically interpretable structure in real collaboration and co-citation networks.

Figure~\ref{fig-yao-c} displays the first two coordinates of the latent homophily embedding for Prof.~Jianqing Fan in the co-citation networks, together with the 19 nearest neighbors in the estimated latent space. Similarly, Figure~\ref{fig-fan} shows the corresponding embedding for Prof.~Jianqing Fan in the co-authorship networks also with the 19 nearest neighbors.
In each figure, the left and right panels show the two-dimensional spectral embeddings of $\hat{\bH}_1$ and $\hat{\bH}_2$, respectively, using the first two eigenvectors of each estimated latent homophily matrix as author
coordinates. Here $\hat{\bH}_1$, the estimator for $\bH_1$, is associated with forming new edges, while $\hat{\bH}_2$, the estimator for $\bH_2:=(\bU_{i,2}^\top \bLambda_2 \bU_{j,2})_{1\le i, j\le p}$, is associated with dissolving existing edges.  The authors close with each other in the left panel tend to have similar latent traits related to the formation of new links. In the co-citation networks, such proximity may reflect similarity in research topics or citation communities. In the co-authorship networks, it may indicate similar attributes leading to future co-authorship. By contrast, proximity in the right panels
suggests the similar traits contributing towards dissolving existing links.

To provide a more concrete interpretation of the dissolution embedding in the co-authorship networks, we further examine the authors highlighted in red in the right panel of Figure~\ref{fig-fan}. Table~\ref{tablefan} reports the year of their most recent collaboration with Prof.~Jianqing Fan, together with the total number of co-authored papers recorded in the MADStat dataset. These authors all collaborated with Prof.~Fan in the past, but no longer maintain an active co-authorship link in the observed network process.

\begin{figure}[ht!]
	\begin{minipage}[b]{0.45\textwidth} 
		\centering 
		\includegraphics[width=0.9\textwidth]{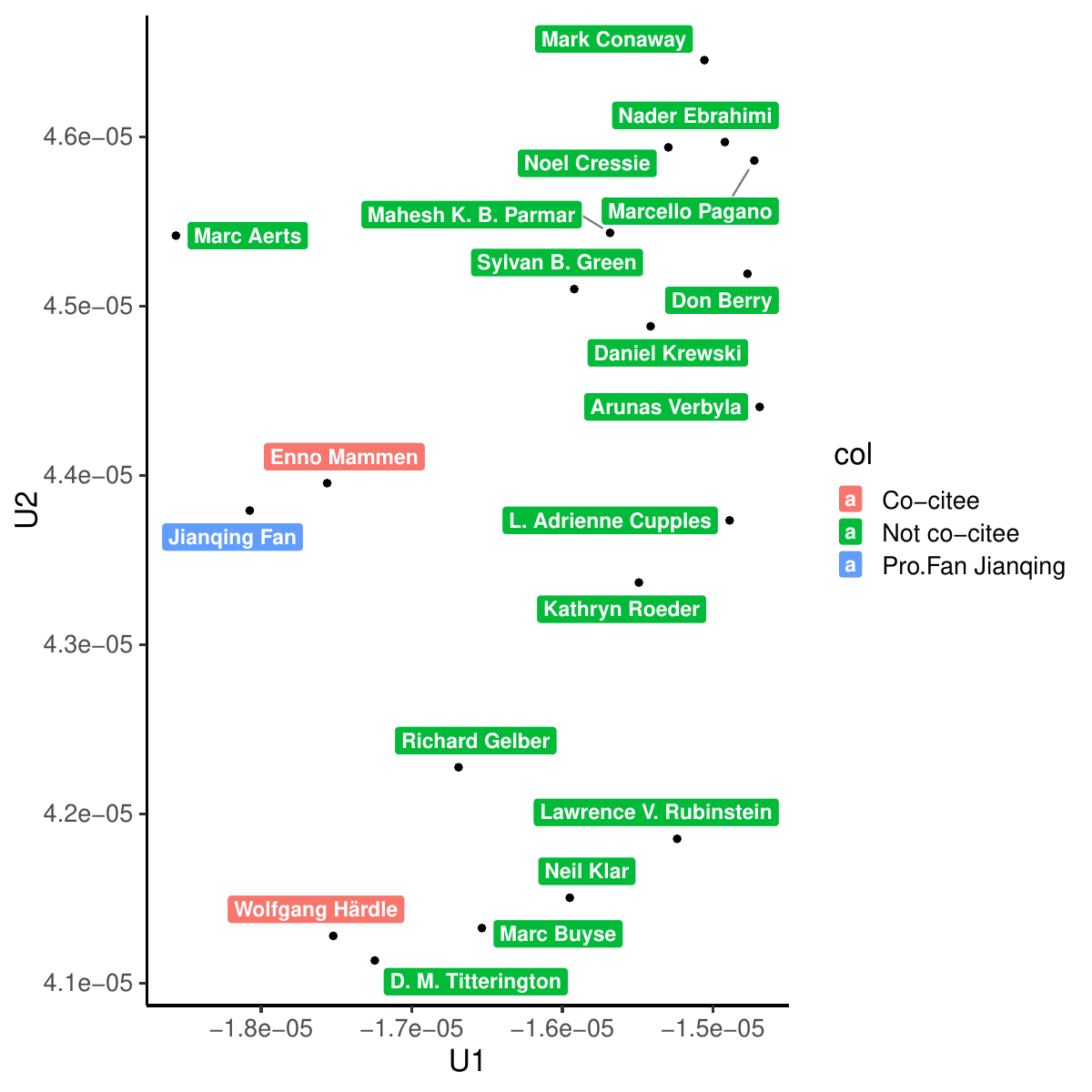}
	\end{minipage}
	\begin{minipage}[b]{0.45\textwidth} 
		\centering 
		\includegraphics[width=0.9\textwidth]{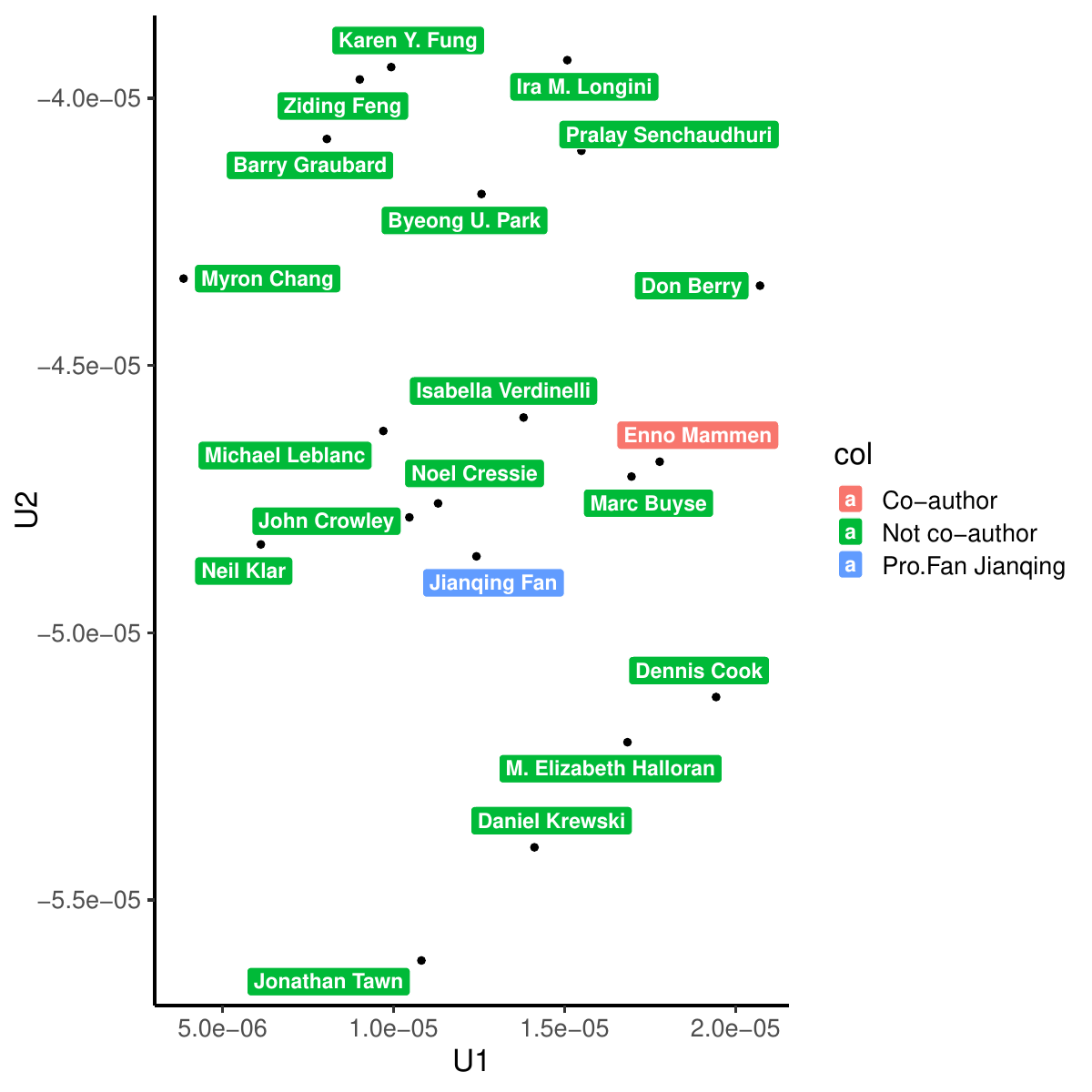}
	\end{minipage}
	\caption{The first two coordinates of the estimated latent homophily embedding vectors in the co-citation networks for Prof.~Jianqing Fan and his 19 nearest neighbors. Left panel: $\hat{\bH}_1$, the estimated latent homophily effect associated with edge formation. Right panel: $\hat{\bH}_2$, the estimated latent homophily effect associated with edge dissolution.}
	\label{fig-yao-c}
\end{figure}

\begin{figure}[ht]
	\begin{minipage}[b]{0.5\textwidth} 
		\centering 
		\includegraphics[width=0.9\textwidth]{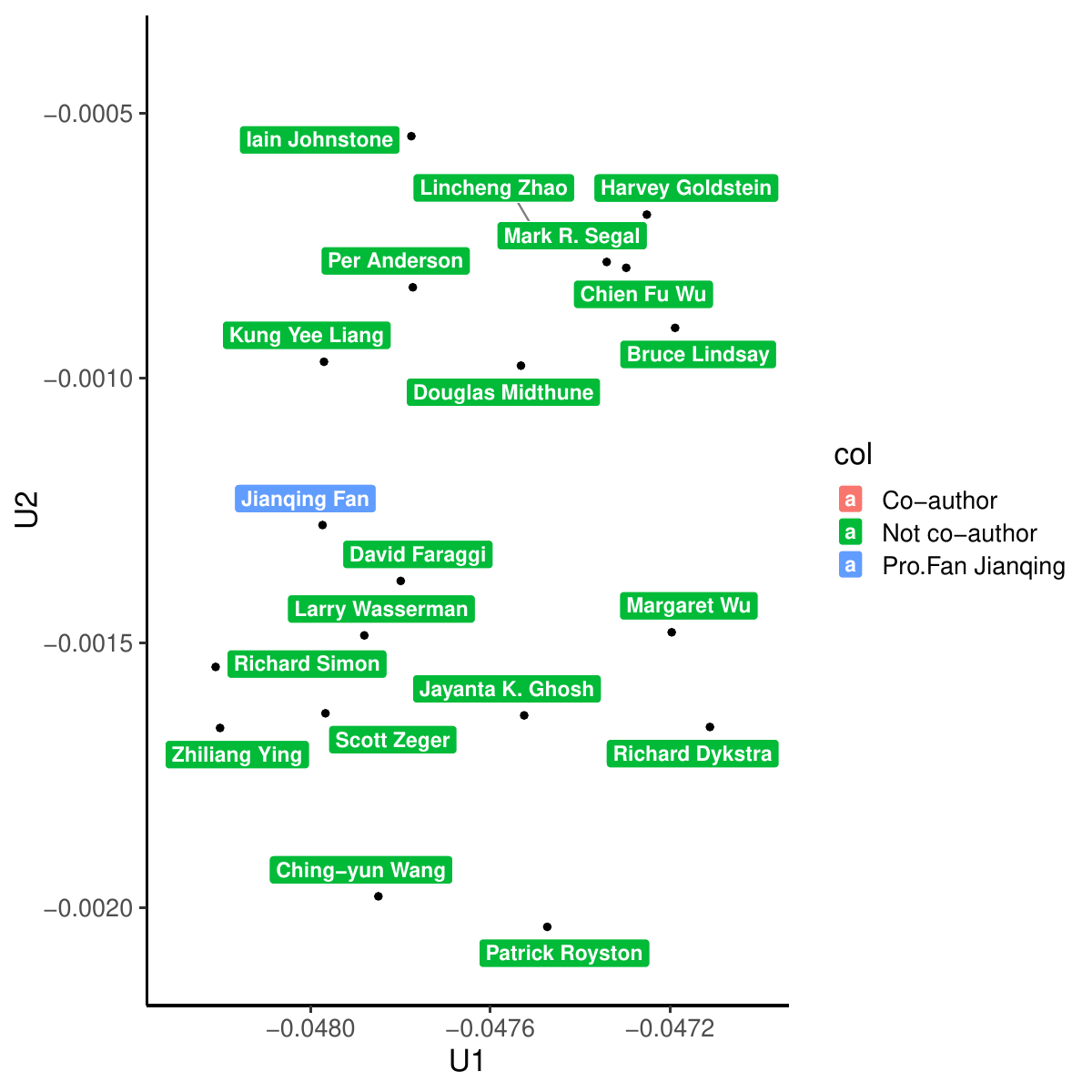}
	\end{minipage}
	\begin{minipage}[b]{0.5\textwidth} 
		\centering 
		\includegraphics[width=0.9\textwidth]{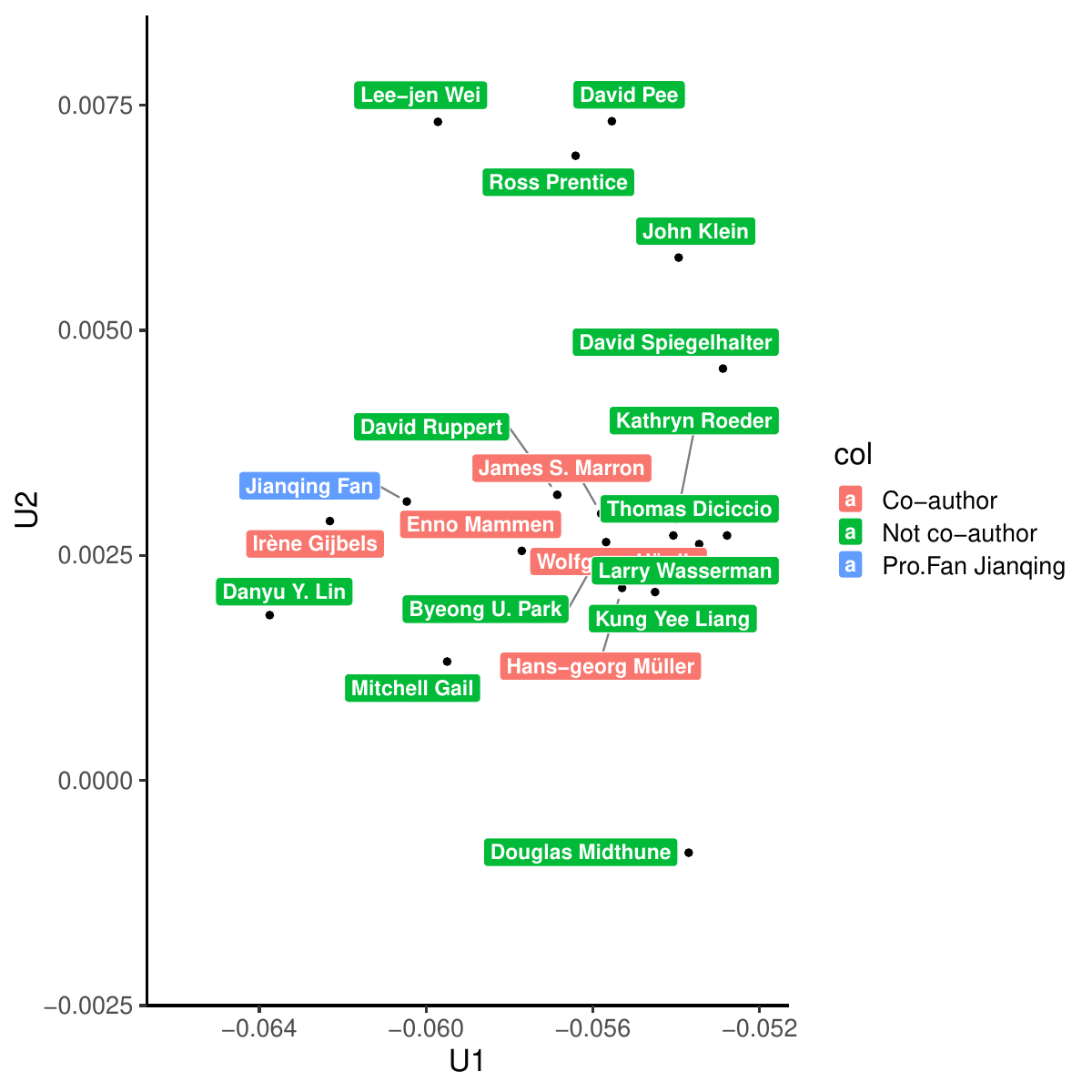}
	\end{minipage}
	\caption{The first two coordinates of the estimated latent homophily embedding vectors in the co-authorship networks for Prof.~Jianqing Fan and his 19 nearest neighbors. Left panel: $\hat{\bH}_1$, the estimated latent homophily effect associated with edge formation. Right panel: $\hat{\bH}_2$, the estimated latent homophily effect associated with edge dissolution. The authors highlighted in red in the right panel are further examined in Table~\ref{tablefan}.}
	\label{fig-fan}
\end{figure}

\begin{table}[ht] 
	\centering
	\begin{tabular}{|c|c|c|}
		\hline
		Name & Last co-authorship year & $\#$ co-authored papers in MADStat \\ \hline
		Irène Gijbels & 2001 & 10\\ \hline
		Enno Mammen & 1998 & 1\\ \hline
		James S. Marron & 1997 & 2\\ \hline
		Wolfgang Härdle & 1998 & 1\\ \hline
		Hans-Georg Müller & 2010 & 1\\ \hline
	\end{tabular}
	\caption{Most recent year of collaboration with Prof.~Jianqing Fan and the total number of co-authored papers in the MADStat dataset for the authors highlighted in red in Figure~\ref{fig-fan}.}
	\label{tablefan}
\end{table}

\begin{figure}[ht!]
	\includegraphics[width=0.9\textwidth]{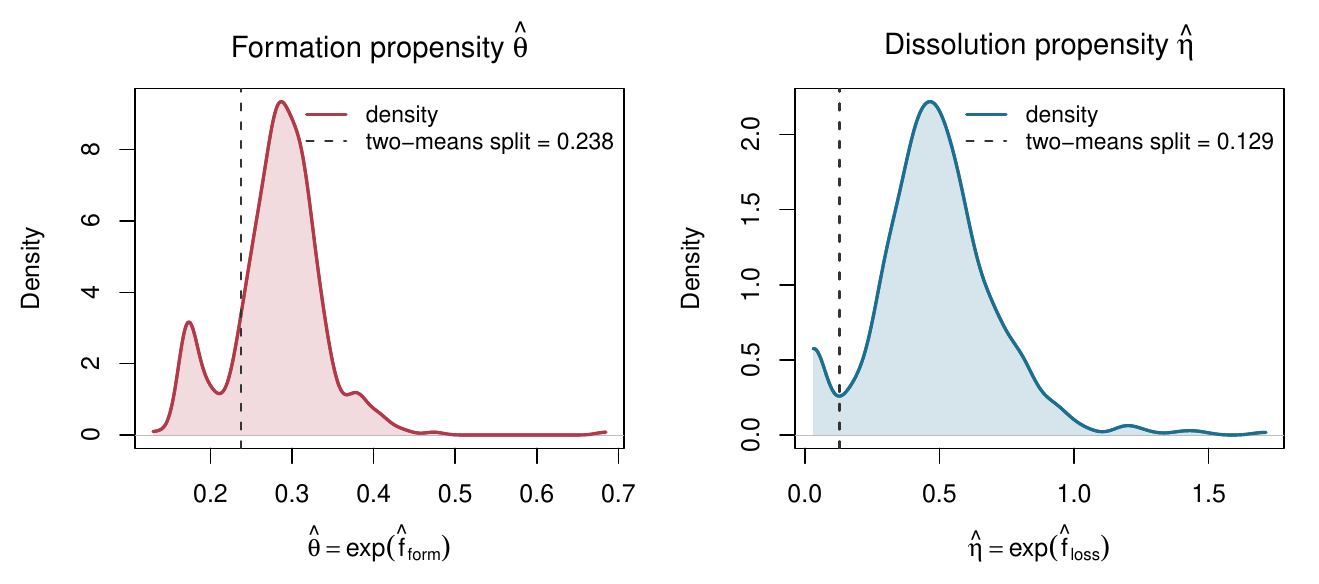}
	\caption{Densities of the estimated edge-formation propensity parameters $\hat{\theta}_i$'s and the estimated edge-formation propensity parameters $\hat{\eta}_i$'s in the co-authorship network. The dashed vertical
  lines give the high/low splits obtained from two-means clustering on the
  log-scale estimates.}
	\label{fig:density}
\end{figure}

Figure~\ref{fig:density} shows the estimated probability densities of the elements in
$\hat{\btheta}$ and $\hat{\bfeta}$ for the co-authorship networks. Under our model, a larger
$\theta_i$ indicates a higher probability of forming a new edge for node $i$, and a larger $\eta_i$ indicates a higher probability of losing an existing edge. 
Both density curves show clear heterogeneity among the authors. In particular, the estimated edge-formation propensity parameters $\hat{\theta}_i$'s have a lower  mode and a
higher mode, and the distribution of  $\hat{\eta}_i$'s separates a small lower-valued group from the majority individuals.
By two-means clustering, we divide authors into high and low groups separately based on  $\hat{\theta}_i$ and
$\hat{\theta}_i$. Crossing the two labels leads to the  four subgroups in the Table \ref{tab:theta_eta}.
\begin{table}[htbp]
\centering
\caption{Number of authors classified by high and low values of $\hat{\theta}_i$ and $\hat{\eta}_i$.}
\label{tab:theta_eta}
\begin{tabular}{c|cc}
 & High $\hat{\eta}_i$ & Low $\hat{\eta}_i$ \\
\hline
High $\hat{\theta}_i$ & 368 authors & 1 author \\
Low $\hat{\theta}_i$  & 46 authors  & 31 authors
\end{tabular}
\end{table}
The high-$\hat{\theta}_i$/high-$\hat{\eta}_i$ group contains those who are keen to form new co-authorship, but are unable to maintain sustained collaboration. It contains a majority of 368 individuals.
The high-$\hat{\theta}_i$/low-
$\hat{\eta}_i$ group consists of those who are keen to form new co-authorship but also manage to maintain the existing collaborations. There is only 1 individual in this group.
The low-$\hat{\theta}_i$/high-
$\hat{\eta}_i$ group contains 46 individuals who may prefer to work on their own. The low-$\hat{\theta}_i$/low-$\hat{\eta}_i$ group contains 31 individuals who are keen to keep the existing relationships but less likely to form new ones.

This finding is broadly consistent with previous studies of scientific
collaboration networks, which emphasize strong heterogeneity, clustering, and
network growth through cumulative advantage or preferential attachment
\citep{newman2001structure,barabasi2002evolution}. Our analysis adds a
dynamic interpretation, i.e. the heterogeneity is not only in the difference of the numbers of collaborators, but is also reflected by the propensities of forming new  collaborations or dissolving existing ones.

\subsubsection{High school contact dataset.}
We further applied the proposed method to the 
contact data collected in a
high school in Marseilles, France \citep{mastrandrea2015contact}.
The data are the recorded face-to-face contacts among the students from 9 classes during five
days in December 2013. We label the 3 classes majored in mathematics and physics as MP1, MP2 and MP3, the 3 classes majored
in biology as BIO1, BIO2 and BIO3, the 2 classes majored in physics and chemistry as PC1
and PC2, and the class majored in engineering as EGI.
Since the data do not contain overnight contact records, we
aggregated the observations into one-hour networks and counted transitions only
between consecutive hourly networks within the same day, i.e. we treat the hourly networks on each day as a separate time series. 
After
removing two students who appear in the metadata but never appear in the
proximity records, the retained sample contains \(327\) students and \(41\)
hourly networks. These \(41\) networks give \(36\) within-day one-hour
transitions. We set $X_{i,j}^t =1$ if individuals $i$ and $j$ have at least one recorded face-to-face contact during hour $t$, and 0 otherwise.

The dyadic covariate vector \(\bZ_{i,j}\) is constructed from the student metadata
and the three additional survey-based relation files, containing the information on demographic/school-organization  and  social-ties. More specifically,  \(\bZ_{i,j}=\bZ_{j,i}\) for $i\neq j$, and
\begin{align*}
\bZ_{i,j}
 =& \;
\bigl(
1,\ 
\mathbf 1\{c_i=c_j\},\
\mathbf 1\{r_i=r_j\},\
\mathbf 1\{g_i \text{ and } g_j \text{ are both observed}\},\
\mathbf 1\{g_i=g_j\},\\ 
& \;\mathbf 1\{ i \text{ and } j 
\text{ are friends}\},\
\mathbf 1\{ i \text{ and } j 
\text{ are connected on Facebook}\},\,\\
& \; \mathbf 1\{ (i, j) 
\text{  appears in the diary file}\}
\bigr)^{\top},
\end{align*}
where \(c_i\), \(r_i\), and \(g_i\) denote, respectively, the observed class label, the academic track (in total 4 different tracks: biology, mathematics and physics, physics and chemistry, and engineer), gender of student \(i\). We take the convention $\mathbf 1\{g_i=g_j\}=0$ if at least one of $g_i$ andd $g_j$ is not observed.

\begin{figure}[!htbp]
\centering
\includegraphics[width=0.9\textwidth]{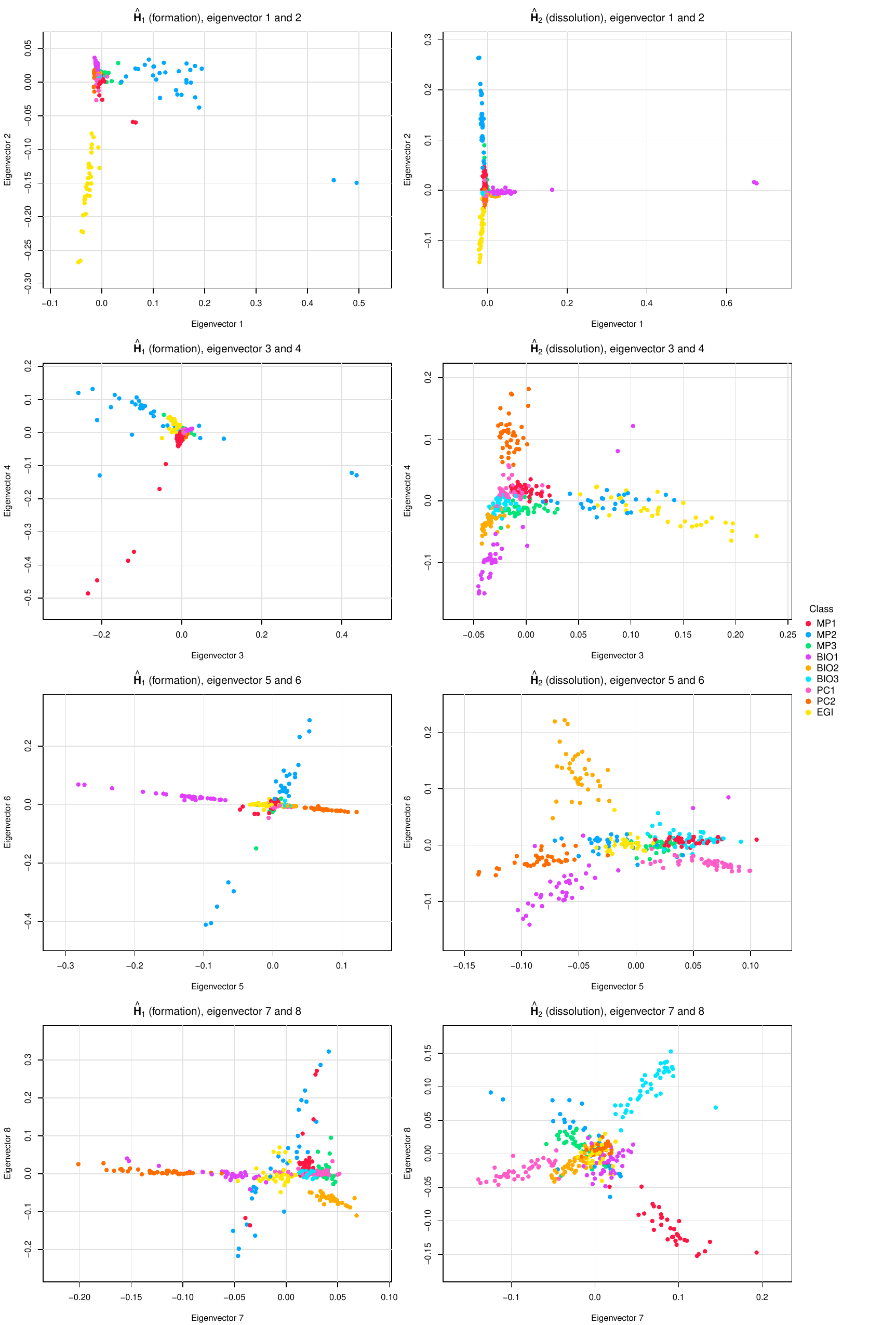}
\caption{Spectral embeddings of \(\widehat{\bH}_1\) and
\(\widehat{\bH}_2\) for the high school hourly contact networks.}
\label{fig:hs2013_embedding_all}
\end{figure}

Figure~\ref{fig:hs2013_embedding_all} visualizes the estimated latent homophily matrices through several pairs of eigenvectors. In each row, the left panel corresponds to \(\widehat{\bH}_1\), which is associated with \(0\to1\) edge formation, and the right panel corresponds to \(\widehat{\bH}_2\), which is associated with \(1\to0\) edge dissolution. The colors indicate the 9 observed class labels, which are used only for visual comparison.

The embeddings show that the estimated latent homophily structure is closely related to the observed school classes. The first two eigenvectors already separate some classes, such as EGI, BIO1, and MP2, from the main concentration of students. The higher eigenvector pairs reveal additional class-related patterns that are less visible in the first two coordinates. In particular, several classes form distinct branches or local concentrations in \(\widehat{\bH}_2\), suggesting that the dissolution-related latent homophily matrix contains especially clear class-structured variation. At the same time, some overlap remains across classes in each two-dimensional projection, so these plots should be interpreted as low-dimensional views of a richer latent space rather than as a complete clustering result.

To examine this structure more systematically, we next use eigen-ratio selection and \(k\)-means clustering on a higher-dimensional spectral representation. 
The eigen-ratio criterion selects the first eight eigenvectors of \(\bD_{\hat{f}_{1}}\widehat{\bH}_1\bD_{\hat{f}_{1}}\) and the first eight eigenvectors of \(\bD_{\hat{f}_{2}}\widehat{\bH}_2\bD_{\hat{f}_{2}}\). We concatenate these two row-normalized spectral representations and apply \(k\)-means clustering over candidate values of \(K\). The clustering criteria select \(K=9\), which matches the true
number of classes in the School. The resulting clusters are the 9 true classes, with $\mathrm{ARI}=\mathrm{NMI}=\mathrm{Purity}=1$. 

This finding should be interpreted together with the construction of \(\bZ_{i,j}\). The same-class and same-track indicators are already included in \(\bZ_{i,j}\), so part of the school-organization effect is explicitly modeled as observed homophily. The strong class agreement in the latent representation therefore suggests that these observed indicators do not fully absorb the class-structured variation in the temporal contact dynamics. In addition, the estimated coefficients of the observed dyadic covariates are very similar across covariates and between the formation and dissolution transition matrices. This indicates that the observed covariates mainly encode stable pair-level background information, rather than producing a strong mechanism-specific separation between \(0\to1\) and \(1\to0\) transitions. The remaining class-structured variation is instead captured by the latent homophily matrices.

From the perspective of covariate-assisted community detection, the latent
homophily matrices therefore define a conditional notion of community. They do not replace the observed covariates in \(\bZ_{i,j}\), but capture structured similarities in edge dynamics that remain after accounting for the measured dyadic homophily effects. From the perspective of dynamic network analysis, these communities are mechanism-specific: they reflect similarity in the formation and dissolution of ties, rather than only similarity in static edge density. The  analysis of this data set therefore supports the use of both observed homophily and latent homophily in modeling dynamic networks.

\section{Summary and Discussion}\label{sec: conclusion}
We proposed a dynamic network model that jointly captures node heterogeneity and link homophily, where homophily may arise from both observed covariates and unobserved latent traits. By distinguishing edge formation from edge dissolution, the model provides a more flexible description of network evolution than those based only on a static edge probability. Methodologically, the normalized squared loss and the associated reparameterization lead to a tractable estimation framework for high-dimensional node-specific effects, and also make it possible to establish uniform consistency when the number of nodes diverges.
An important feature of the proposed normalized squared estimation framework is that it can also be easily adapted to other scenarios such as directed networks, weighted networks and hypergraphs. 

Several directions deserve further investigation. First, it would be important to extend the framework beyond first-order dynamics to allow for richer temporal dependence and possible nonstationarity. Second, while our loss-based approach can be extended to more general settings, a fuller treatment of edge dependence, directed networks, and multiplex networks remains to be developed. Third, future work should address statistical inference, including uncertainty quantification and testing procedures, as well as data-driven selection of tuning parameters and latent dimensions.

\renewcommand{\baselinestretch}{1}
\normalsize	
\bibliographystyle{apalike}
\bibliography{reference}

@article{mastrandrea2015contact,
  title   = {Contact Patterns in a High School: A Comparison between Data
             Collected Using Wearable Sensors, Contact Diaries and
             Friendship Surveys},
  author  = {Mastrandrea, Rossana and Fournet, Julie and Barrat, Alain},
  journal = {PLoS ONE},
  volume  = {10},
  number  = {9},
  pages   = {e0136497},
  year    = {2015},
  doi     = {10.1371/journal.pone.0136497}
}

@article{newman2001structure,
  title={The structure of scientific collaboration networks},
  author={Newman, Mark EJ},
  journal={Proceedings of the national academy of sciences},
  volume={98},
  number={2},
  pages={404--409},
  year={2001},
  publisher={The National Academy of Sciences}
}

@article{barabasi2002evolution,
  title={Evolution of the social network of scientific collaborations},
  author={Barab{\^a}si, Albert-Laszlo and Jeong, Hawoong and N{\'e}da, Zoltan and Ravasz, Erzsebet and Schubert, Andras and Vicsek, Tamas},
  journal={Physica A: Statistical mechanics and its applications},
  volume={311},
  number={3-4},
  pages={590--614},
  year={2002},
  publisher={Elsevier}
}

@article{athreya2018statistical,
  title={Statistical inference on random dot product graphs: a survey},
  author={Athreya, Avanti and Fishkind, Donniell E and Tang, Minh and Priebe, Carey E and Park, Youngser and Vogelstein, Joshua T and Levin, Keith and Lyzinski, Vince and Qin, Yichen and Sussman, Daniel L},
  journal={Journal of Machine Learning Research},
  volume={18},
  number={226},
  pages={1--92},
  year={2018}
}

@article{jiang2025two,
  title={A two-way heterogeneity model for dynamic networks},
  author={Jiang, Binyan and Leng, Chenlei and Yan, Ting and Yao, Qiwei and Yu, Xinyang},
  journal={The Annals of Statistics},
  volume={53},
  number={6},
  pages={2617--2641},
  year={2025},
  publisher={Institute of Mathematical Statistics}
}

@article{xie2022eigenvector,
  title={Eigenvector-Assisted Statistical Inference for Signal-Plus-Noise Matrix Models},
  author={Xie, Fangzheng and Wu, Dingbo},
  journal={arXiv preprint arXiv:2203.16688},
  year={2022}
}

@article{xie2024eigenvector,
  title={An eigenvector-assisted estimation framework for signal-plus-noise matrix models},
  author={Xie, Fangzheng and Wu, Dingbo},
  journal={Biometrika},
  volume={111},
  number={2},
  pages={661--676},
  year={2024},
  publisher={Oxford University Press}
}

@article{xie2023efficient,
  title={Efficient estimation for random dot product graphs via a one-step procedure},
  author={Xie, Fangzheng and Xu, Yanxun},
  journal={Journal of the American Statistical Association},
  volume={118},
  number={541},
  pages={651--664},
  year={2023},
  publisher={Taylor \& Francis}
}

@article{chi2019nonconvex,
  title={Nonconvex optimization meets low-rank matrix factorization: An overview},
  author={Chi, Yuejie and Lu, Yue M and Chen, Yuxin},
  journal={IEEE Transactions on Signal Processing},
  volume={67},
  number={20},
  pages={5239--5269},
  year={2019},
  publisher={IEEE}
}

@article{chen2015incoherence,
  title={Incoherence-optimal matrix completion},
  author={Chen, Yudong},
  journal={IEEE Transactions on Information Theory},
  volume={61},
  number={5},
  pages={2909--2923},
  year={2015},
  publisher={IEEE}
}

@article{candes2011robust,
  title={Robust principal component analysis?},
  author={Cand{\`e}s, Emmanuel J and Li, Xiaodong and Ma, Yi and Wright, John},
  journal={Journal of the ACM (JACM)},
  volume={58},
  number={3},
  pages={1--37},
  year={2011},
  publisher={ACM New York, NY, USA}
}

@article{candes2012exact,
  title={Exact matrix completion via convex optimization},
  author={Candes, Emmanuel and Recht, Benjamin},
  journal={Communications of the ACM},
  volume={55},
  number={6},
  pages={111--119},
  year={2012},
  publisher={ACM New York, NY, USA}
}

@article{yan2021covariate,
  title={Covariate regularized community detection in sparse graphs},
  author={Yan, Bowei and Sarkar, Purnamrita},
  journal={Journal of the American Statistical Association},
  volume={116},
  number={534},
  pages={734--745},
  year={2021},
  publisher={Taylor \& Francis}
}

@article{chang2024autoregressive,
  title={Autoregressive Networks with Dependent Edges},
  author={Chang, Jinyuan and Fang, Qin and Kolaczyk, Eric D and MacDonald, Peter W and Yao, Qiwei},
  journal={Journal of the Royal Statistical Society Series B, in press},
  year={2026}
}

@article{hu2024network,
  title={Network-adjusted covariates for community detection},
  author={Hu, Yaofang and Wang, Wanjie},
  journal={Biometrika},
  pages={asae011},
  year={2024},
  publisher={Oxford University Press}
}

@article{ji2022co,
  title={Co-citation and Co-authorship Networks of Statisticians},
  author={Ji, Pengsheng and Jin, Jiashun and Ke, Zheng Tracy and Li, Wanshan},
  journal={Journal of Business \& Economic Statistics},
  volume={40},
  number={2},
  pages={469--485},
  year={2022},
  publisher={Taylor \& Francis}
}

@article{li2023statistical,
	title={Statistical Inference on Latent Space Models for Network Data},
	author={Li, Jinming and Xu, Gongjun and Zhu, Ji},
	journal={arXiv preprint arXiv:2312.06605},
	year={2023}
}

@article{yu2015useful,
	title={A useful variant of the Davis--Kahan theorem for statisticians},
	author={Yu, Yi and Wang, Tengyao and Samworth, Richard J},
	journal={Biometrika},
	volume={102},
	number={2},
	pages={315--323},
	year={2015},
	publisher={Oxford University Press}
}

@article{mei2012encoding,
	title={Encoding low-rank and sparse structures simultaneously in multi-task learning},
	author={Mei, Shike and Cao, Bin and Sun, Jiantao},
	journal={Advances in Neural Information Processing Systems (NIPS)},
	pages={1--16},
	year={2012},
	publisher={Curran Associates, Inc. Red Hook, NY, USA}
}

@article{hoff2008,
	author = {Hoff, P.},
	title ={Modeling homophily and stochastic equivalence in symmetric relational data},
	journal ={Advances in Neural Information Processing Systems},
	year = {2008},
	volume={20},
	publisher={MIT Press},
	pages={657-664},
}

@article{MLC2001,
	title={BIRDS OF A FEATHER: Homophily in Social Networks},
	author={Miller McPherson and Lynn Smith-Lovin and James M Cook},
	journal={Annu. Rev. Sociol.
	},
	year={2001},
	volume={27},
	pages={415-444},
}

@article{graham2016,
	author = {Graham, B.S.},
	title ={Homophily and Transitivity in Dynamic Network Formation},
	journal ={NBER Working Paper 22186},
	year ={2016}
}

@article{Bickel2009,
	title={A nonparametric view of network models and Newman--Girvan and other modularities},
	author={Bickel, Peter J and Chen, Aiyou},
	journal={Proceedings of the National Academy of Sciences},
	volume={106},
	number={50},
	pages={21068--21073},
	year={2009},
	publisher={National Acad Sciences}
}

@article{Borsboom21,
	title={Network analysis of multivariate data in psychological science},
	author={Borsboom, Denny and Deserno, Marie K and Rhemtulla, Mijke and Epskamp, Sacha and Fried, Eiko I and McNally, Richard J and Robinaugh, Donald J and Perugini, Marco and Dalege, Jonas and Costantini, Giulio and others},
	journal={Nature Reviews Methods Primers},
	volume={1},
	number={1},
	pages={1--18},
	year={2021},
	publisher={Nature Publishing Group}
}

@article{wiuf2006likelihood,
	title={A likelihood approach to analysis of network data},
	author={Wiuf, Carsten and Brameier, Markus and Hagberg, Oskar and Stumpf, Michael PH},
	journal={Proceedings of the National Academy of Sciences},
	volume={103},
	number={20},
	pages={7566--7570},
	year={2006},
	publisher={National Acad Sciences}
}

@article{chen2021analysis,
	title={Analysis of networks via the sparse $\beta$-model},
	author={Chen, Mingli and Kato, Kengo and Leng, Chenlei},
	journal={Journal of the Royal Statistical Society: Series B (Statistical Methodology)},
	year={2021},
	publisher={Wiley Online Library}
}

@article{fienberg2012brief,
	title={A brief history of statistical models for network analysis and open challenges},
	author={Fienberg, Stephen E},
	journal={Journal of Computational and Graphical Statistics},
	volume={21},
	number={4},
	pages={825--839},
	year={2012},
	publisher={Taylor \& Francis}
}

@article{goldenberg2010survey,
	title={A survey of statistical network models},
	author={Goldenberg, Anna and Zheng, Alice X and Fienberg, Stephen E and Airoldi, Edoardo M and others},
	journal={Foundations and Trends{\textregistered} in Machine Learning},
	volume={2},
	number={2},
	pages={129--233},
	year={2010},
	publisher={Now Publishers, Inc.}
}

@book{Graham20,
	title={The Econometric Analysis of Network Data},
	author={Graham, Bryan and De Paula, Aureo},
	year={2020},
	publisher={Academic Press}
}

@article{chatterjee2011,
	title={Random graphs with a given degree sequence},
	author={Chatterjee, Sourav and Diaconis, Persi and Sly, Allan and others},
	journal={The Annals of Applied Probability},
	volume={21},
	number={4},
	pages={1400--1435},
	year={2011},
	publisher={Institute of Mathematical Statistics}
}

@book{ek17,
	title={Topics at the Frontier of Statistics and Network
	Analysis},
	author={Kolaczyk, Eric D},
	year={2017},
	publisher={Cambridge University Press}
}

@article{karrer2011stochastic,
	title={Stochastic blockmodels and community structure in networks},
	author={Karrer, Brian and Newman, Mark EJ},
	journal={Physical review E},
	volume={83},
	number={1},
	pages={016107},
	year={2011},
	publisher={APS}
}

@article{yan2019statistical,
	title={Statistical inference in a directed network model with covariates},
	author={Yan, Ting and Jiang, Binyan and Fienberg, Stephen E and Leng, Chenlei},
	journal={Journal of the American Statistical Association},
	volume={114},
	number={526},
	pages={857--868},
	year={2019},
	publisher={Taylor \& Francis}
}

@article{ma2020universal,
	title={Universal latent space model fitting for large networks with edge covariates},
	author={Ma, Zhuang and Ma, Zongming and Yuan, Hongsong},
	journal={Journal of Machine Learning Research},
	volume={21},
	number={4},
	pages={1--67},
	year={2020}
}

@article{graham2017econometric,
	title={An econometric model of network formation with degree heterogeneity},
	author={Graham, Bryan S},
	journal={Econometrica},
	volume={85},
	number={4},
	pages={1033--1063},
	year={2017},
	publisher={Wiley Online Library}
}

@article{ANJiang,
	title={Autoregressive networks},
	author={Jiang, Binyan and Li, Jailing and Yao, Qiwei},
	journal={Journal of Machine Learning Research},
	year={2023},
	volume={24(227)},
	pages={1-69}
}

@article{stein2021sparse,
	title={A Sparse Random Graph Model for Sparse Directed Networks},
	author={Stein, Stefan and Leng, Chenlei},
	journal={arXiv preprint arXiv:2108.09504},
	year={2021}
}
\newpage
\begin{appendix}
\section{Appendix A}
Throughout the appendix, we use $\f^{*}$, $\bbeta^{*}$, and $\bH^{*}$ to represent the ground-truth parameters ($\f$, $\bbeta$, and $\bH_{1}$ in Section \ref{sec: parameter estimation}), while $\f$, $\bbeta$, and $\bH$ represent arbitrary parameters.
\subsection{Proof of Proposition \ref{Hessian_inf} and Proposition \ref{Hessian_inf_H}(i)} 
\begin{proof}
The proof for Proposition \ref{Hessian_inf} is the same as that of Proposition \ref{Hessian_inf_H}(i) with the latent homophily component set to be $1$. We only provide the proof for Proposition \ref{Hessian_inf_H}(i) below. 
	Note that 
	\begin{eqnarray*}
		l_{1}\left(\f,\bbeta|\bA,\bH\right)&=&\frac{1}{4p}\sum_{1\leq i\neq j\leq p} \left( A_{i,j}e^{-f_{i}-f_{j}}-e^{f_{i}+f_{j}} \left( \bZ_{i,j}^{\top} \bbeta  +H_{i,j}\right)\right) ^2.
	\end{eqnarray*}
	The  first and second order partial derivatives of $l_1\left(\f,\bbeta|\bA,\bH\right)$ are given as:
	\begin{eqnarray*}
		\frac{\partial l_1\left(\f,\bbeta|\bA,\bH\right)}{\partial f_{i}}&=&\frac{1}{p}\sum_{j\neq i;j=1}^{p} \left[ e^{2(f_{i}+f_{j})} \left( \bZ_{i,j}^{\top} \bbeta +H_{i,j} \right)^2-  A_{i,j}^2e^{-2(f_{i}+f_{j})} \right], \\
		\frac{\partial^2 l_1\left(\f,\bbeta|\bA,\bH\right)}{\partial f_{i}f_{j}}&=& \frac{2}{p}\left( e^{2(f_{i}+f_{j})} \left( \bZ_{i,j}^{\top} \bbeta +H_{i,j} \right)^2+  A_{i,j}^2e^{-2(f_{i}+f_{j})}\right) ,\quad i\neq j;\\
		\frac{\partial^2 l_1\left(\f,\bbeta|\bA,\bH\right)}{\partial f_{i}^2}&=&\sum_{j\neq i;j=1}^{p}\frac{\partial^2 l}{\partial f_{i}f_{j}},\\
		\frac{\partial l_1\left(\f,\bbeta|\bA,\bH\right)}{\partial\bbeta}&=&\frac{1}{p} \sum_{1\leq i\neq j\leq p}  e^{2(f_{i}+f_{j})} \bZ_{i,j}\bZ_{i,j}^{\top} \bbeta-\left(A_{i,j}-e^{2(f_{i}+f_{j})} H_{i,j} \right)  \bZ_{i,j},\\ 
		\frac{\partial^2 l_1\left(\f,\bbeta|\bA,\bH\right)}{\partial\bbeta\bbeta^{\top}}&=& \frac{1}{p}\sum_{1\leq i\neq j\leq p}  e^{2(f_{i}+f_{j})} \bZ_{i,j}\bZ_{i,j}^{\top},\\
		\frac{\partial^2 l_1\left(\f,\bbeta|\bA,\bH\right)}{\partial\bbeta\partial f_{i}}&=&\frac{2}{p}\sum_{j\neq i;j=1}^{p}e^{2(f_{i}+f_{j})} \left( \bZ_{i,j}^{\top}\bbeta +H_{i,j}\right)\bZ_{i,j}  .
	\end{eqnarray*}		
	Denote the Hessian matrix of the loss function as 
	\[ \bV(\f,\bbeta|\bA,\bH)=
	\begin{bmatrix} 
		\frac{\partial^2{l_{1}(\f,\bbeta|\bA,\bH)}}{\partial{\f}\partial{\f^\top}} & \frac{\partial^2{l_{1}(\f,\bbeta|\bA,\bH)}}{\partial{\f}\partial{\bbeta^\top}} \\
		\frac{\partial^2{l_{1}(\f,\bbeta|\bA,\bH)}}{\partial{\bbeta}\partial{\f^\top}} &\frac{\partial^2{l_{1}(\f,\bbeta|\bA,\bH)}}{\partial{\bbeta}\partial{\bbeta^\top}} 
	\end{bmatrix} := \begin{bmatrix} 
		\bV_{1}(\f,\bbeta|\bA,\bH) & \bV_{2}(\f,\bbeta|\bA,\bH)^{\top}  \\
		\bV_{2}(\f,\bbeta|\bA,\bH) & \bV_{3} (\f,\bbeta|\bA,\bH)
	\end{bmatrix}.
	\]
	Consider $\inf_{\f,\bbeta}(\ba,\bb)^{\top}\bV(\f,\bbeta|\bA,\bH)(\ba,\bb)$ such that $\ba\in \mathbb{R}^{p}$, $\bb\in \mathbb{R}^{q}$.
	There exist constants $C_{1}>0$, $C_{2}>0$ such that, 
	\begin{eqnarray}\label{Hessian_low}
		&& (\ba,\bb)^{\top}\bV(\f,\bbeta|\bA,\bH)(\ba,\bb)\nonumber\\
		&=&\ba^{\top}\bV_{1}(\f,\bbeta|\bA,\bH) \ba +\bb^{\top}\bV_{3}(\f,\bbeta|\bA,\bH) \bb + 2\ba^{\top}\bV_{2}(\f,\bbeta|\bA,\bH)^{\top} \bb \nonumber\\
		&=&\frac{1}{p} \Bigg\{ \sum_{1\leq i\neq j\leq p} \left( e^{2(f_{i}+f_{j})} \left( \bZ_{i,j}^{\top} \bbeta+H_{i,j} \right)^2+  A_{i,j}^2e^{-2(f_{i}+f_{j})}\right)\left( a_{i}+a_{j}\right)^2+ \sum_{1\leq i\neq j\leq p}e^{2(f_{i}+f_{j})} \left(  \bZ_{i,j}^{\top} \bb \right) ^2\nonumber\\
		&&+\sum_{1\leq i\neq j\leq p}2\left(  a_{i}+a_{j} \right) e^{2(f_{i}+f_{j})} \left( \bZ_{i,j}^{\top}\bbeta+H_{i,j} \right)\bZ_{i,j}^{\top}\bb \Bigg\}\nonumber\\
		&=&\frac{1}{p} \Bigg\{\sum_{1\leq i\neq j\leq p}e^{2(f_{i}+f_{j})}\left[  \bZ_{i,j}^{\top} \bb +  \left( \bZ_{i,j}^{\top}\bbeta +H_{i,j}\right)\left( a_{i}+a_{j}\right) \right]^2\nonumber +A_{i,j}^2e^{-2(f_{i}+f_{j})}\left( a_{i}+a_{j}\right)^2 \Bigg\}\nonumber\\
		&=&\frac{1}{p} \Bigg\{\sum_{1\leq i\neq j\leq p}\left( A_{i,j}^2e^{-2(f_{i}+f_{j})}-D_{i,j}^2\right) \left( a_{i}+a_{j}\right)^2+\frac{e^{2(f_{i}+f_{j})} D_{i,j}^2}{e^{2(f_{i}+f_{j})} \left( \bZ_{i,j}^{\top} \bbeta +H_{i,j} \right)^2+D_{i,j}^2}\left(  \bZ_{i,j}^{\top} \bb \right) ^2\nonumber\\
		&&+\Bigg[ \frac{e^{2(f_{i}+f_{j})} \left( \bZ_{i,j}^{\top} \bbeta +H_{i,j} \right)}{ \sqrt{e^{2(f_{i}+f_{j})} \left( \bZ_{i,j}^{\top} \bbeta +H_{i,j} \right)^2+D_{i,j}^2}} \bZ_{i,j}^{\top} \bb +  \sqrt{e^{2(f_{i}+f_{j})} \left( \bZ_{i,j}^{\top}\bbeta+H_{i,j} \right)^2+D_{i,j}^2} \left( a_{i}+a_{j}\right) \Bigg]^2 \Bigg\}\nonumber\\
		&\geq&\frac{1}{p}\max_{D_{i,j}} \Bigg\{\sum_{1\leq i\neq j\leq p}\left( A_{i,j}^2e^{-2(f_{i}+f_{j})}-D_{i,j}^2\right) \left( a_{i}+a_{j}\right)^2+\frac{ e^{2(f_{i}+f_{j})}D_{i,j}^2}{e^{2(f_{i}+f_{j})} \left( \bZ_{i,j}^{\top} \bbeta +H_{i,j} \right)^2+D_{i,j}^2}\left(  \bZ_{i,j}^{\top} \bb \right) ^2 \Bigg\}\nonumber\\
		&\geq&\frac{1}{p}\max_{\{D_{i,j}: D^2_{i,j}< A_{i,j}^2e^{-2(f_{i}+f_{j})}\}} \Bigg\{2(p-2)\|\ba\|_{2}^2\inf_{1\leq i \neq j \leq p}\left( A_{i,j}^2e^{-2(f_{i}+f_{j})}-D_{i,j}^2\right)\nonumber\\
		&&+\|\bb\|_{2}^2\min_{\|\bx\|_{2}=1}\sum_{1\leq i\neq j\leq p}\frac{e^{2(f_{i}+f_{j})} D_{i,j}^2}{e^{2(f_{i}+f_{j})} \left( \bZ_{i,j}^{\top} \bbeta+H_{i,j} \right)^2+D_{i,j}^2} \left(  \bZ_{i,j}^{\top}\bx\right) ^2  \Bigg\}\nonumber\\
		&\geq& C \min_{1\le i\ne j\le p}A_{i,j}^2\left( \|\ba\|_{2}^2+p\| \bb\|_{2}^2\right) e^{-2\|\f\|_\infty},
	\end{eqnarray}
	for some small enough constant $C\ge 0$. Here in the last two inequality, we have used the fact that
	$\sum_{1\le i\ne j\le p}(a_i+a_j)^2= 2{\ba}^\top ((p-2)I_p+{\bf 1}_{p,1}{\bf 1}_{p,1}^\top){\ba}\ge 2(p-1)\|\ba\|_2^2$, and the last two inequalities are valid  by setting $D_{i,j}^2= \frac{1}{2}A_{i,j}^2\min\{ e^{-2(f_i+f_j)}, e^{2(f_i+f_j)}\}$ and using the fact that $\frac{e^{2(f_{i}+f_{j})} D_{i,j}^2}{e^{2(f_{i}+f_{j})} ( \bZ_{i,j}^{\top} \bbeta+H_{i,j} )^2+D_{i,j}^2}\ge \frac{A_{i,j}^2e^{-2 \|\f\|_\infty}} {2( \bZ_{i,j}^{\top} \bbeta+H_{i,j} )^2+ A_{i,j}^2 }$ and Condition (C1).
\end{proof}
 
\subsection{Proof of Theorem \ref{thm1}}\label{thm1_p}
\begin{proof}
	Denote
	\begin{align*}
		l_{p}\left(\f,\bbeta\right)&:=l\left(\f,\bbeta| \hat{\bA},\bH^{(0)} \right)=\frac{1}{2p}\sum_{1\leq i\neq j\leq p} \left( \hat{A}_{i,j}e^{-f_{i}-f_{j}}-e^{f_{i}+f_{j}}  \left( \bZ_{i,j}^{\top} \bbeta+H^{(0)}_{i,j}  \right)\right) ^2 ,\\
		l_{E}\left(\f,\bbeta\right) &:=l\left(\f,\bbeta|\bA,\bH\right)=\frac{1}{2p}\sum_{1\leq i\neq j\leq p} \left( A_{i,j}e^{-f_{i}-f_{j}}-e^{f_{i}+f_{j}}  \left( \bZ_{i,j}^{\top} \bbeta  +H_{i,j}^* \right)\right) ^2,
	\end{align*}
	and
	\begin{eqnarray}\label{l_d}
		l_{\Delta}\left(\f,\bbeta\right)
		:=& l_{ p}\left(\f^*, \bbeta^* \right)- l_{p}\left(\f,\bbeta\right)   -\left(l_{E}\left( \f^*, \bbeta^* \right)- l_{E}\left(\f, \bbeta \right)\right).
	\end{eqnarray}
	Define the $\ell_\infty$ ball centered at $\bx$ with radius $r$ as $\bB_\infty(\bx,r) = \{\by : \|\by-\bx\|_\infty \le r\}.$ With some abuse of notations, let  $\hat{\bf f}=(\hat{f}_1,\ldots, \hat{f}_p)^\top$ and $\hat{\bbeta}$ be the minimizer of $l\left(\cdot|\hat{\bA},\bH^{(0)}\right)$ with $(\f,\beta)$ constrained on the set $\{(\f,\beta): \f\in \bB_{\infty}\left( \f^*,\alpha_{0}\right), \beta\in \bB_{\infty}\left( \beta^*,\alpha_{0}\right) \}$, where  the radius $\alpha_{0}$ is a big enoiugh positive constant to be defined later.  By Proposition \ref{Hessian_inf_H}, we have that $\left(\f^*,\bbeta^*\right)$ is the minimizer of $l_{E}$ with $l_{E}\left(\f^*, \bbeta^* \right)=0$. Moreover, from  Proposition 2,  there exists a constant $C_{1}>0$ such that  
	\begin{eqnarray}\label{low}
		l_{\Delta}\left(\hat{\f},\hat{\bbeta}\right)
		&=&  l_{ p}\left(\f^*, \bbeta^* \right)- l_{p}\left( \hat{\f},\hat{\bbeta}\right)   -\left(l_{E}\left( \f^*, \bbeta^* \right)- l_{E}\left( \hat{\f},\hat{\bbeta}\right)\right)\nonumber\\
		&\geq&l_{E}\left( \hat{\f},\hat{\bbeta}\right)-l_{E}\left( \f^*,\bbeta^* \right)       \nonumber \\  
		&\geq&  C_{1}e^{-2\|\f^*\|_{\infty}} \left( \|\hat{\f}-\f^*\|_{2}^{2}+\|\sqrt{p}\left( \hat{\bbeta}-\bbeta^*\right) \|_{2}^{2}\right).
	\end{eqnarray}

	On the other hand, we have 
	\begin{eqnarray}\label{L1234}
		&&l_{\Delta}\left(\f,\bbeta\right)\nonumber\\
		&\leq&\frac{1}{2p}\Bigg| \sum_{1\leq i\neq j\leq p} \left( e^{-2f_{i}-2f_{j}}- e^{-2f_{i}^*-2f_{j}^*} \right)  \left(\hat{A}_{i,j}^2-\left( A_{i,j}\right) ^2 \right)  - 2\left(\hat{A}_{i,j}-A_{i,j} \right) \bZ_{i,j}^{\top}\left( \bbeta -\bbeta^* \right)\nonumber\\
		&&+2\left( e^{2f_{i}+2f_{j}}\bZ_{i,j}^{\top}\bbeta -e^{2f_{i}^*+2f_{j}^*}\bZ_{i,j}^{\top} \bbeta^*\right) \left(H^{(0)}_{i,j}-H_{i,j}^* \right)\nonumber\\
		&&+\left( e^{2f_{i}+2f_{j}}- e^{2f_{i}^*+2f_{j}^*} \right)  \left( \left(H^{(0)}_{i,j}\right)^2-\left( H_{i,j}^*\right) ^2 \right)\Bigg|\nonumber \\
		&\leq&\frac{1}{2p} \left|\sum_{1\leq i\neq j\leq p} \left( e^{-2f_{i}-2f_{j}}- e^{-2f_{i}^*-2f_{j}^*} \right)  \left(\hat{A}_{i,j}+A_{i,j} \right) \left(\hat{A}_{i,j}-A_{i,j} \right) \right| \nonumber\\
		&&+\frac{1}{p}\left|\sum_{1\leq i\neq j\leq p}  \left(\hat{A}_{i,j}-A_{i,j} \right) \bZ_{i,j}^{\top}\left(\bbeta -\bbeta^* \right) \right|\nonumber\\
		&&+\frac{1}{p} \left|\sum_{1\leq i\neq j\leq p} \left( e^{2f_{i}+2f_{j}}\bZ_{i,j}^{\top}\bbeta -e^{2f_{i}^*+2f_{j}^*}\bZ_{i,j}^{\top} \bbeta^*\right) \left(H^{(0)}_{i,j}-H_{i,j}^* \right)\right|\nonumber\\
		&&+\frac{1}{2p} \left|\sum_{1\leq i\neq j\leq p} \left( e^{2f_{i}+2f_{j}}- e^{2f_{i}^*+2f_{j}^*} \right)  \left((H^{(0)}_{i,j})^2-\left( H_{i,j}^*\right) ^2 \right)\right| \nonumber\\
		&=:&L_{1}+L_{2}+L_{3}+L_{4}.
	\end{eqnarray}

	Firstly, we consider $L_{1}$.   Set
	\begin{align*}
		l_{i,j}\left(\f \right)  := e^{-2f_{i}-2f_{j}} . 
	\end{align*}
	Note that for any given constant $\alpha$, there exists a big enough constant $C_{\alpha}>0$ such that
	\begin{align*}
		\frac{\partial^{k}l_{i,j}\left(\f \right)}{\partial f_{i}^{s}\partial f^{k-s}_{j}}=(-2)^ke^{-2f_{i}-2f_{j}}\leq C_{\alpha}\frac{(k-1)!}{\alpha^{k}},
	\end{align*}
	for all $k$.   
 Moreover, we also have that $  \hat{A}_{i,j}^2\leq 1$ and $\var\left(\hat{A}_{i,j}^2 \right)\leq1$.  
	By Theorem 4.1 in \cite{jiang2025two} and Condition (C2),  for any big enough constant $\alpha$ and  choosing $\alpha_{0}\in (0,\alpha/2)$,   there exist large enough constants $C_{2}>0, C'_{2}>0$ such that,  as $p\to \infty$, with probability greater than $1-(np)^{-C'_{2}}$,  
	\begin{eqnarray}\label{up1}
		L_{1}&=&\frac{1}{2p}\left| \sum_{1\leq i\neq j\leq p} \left( e^{-2f_{i}-2f_{j}}- e^{-2f_{i}^*-2f_{j}^*} \right)  \left(\hat{A}^2_{i,j}-\left( A_{i,j}\right) ^2 \right) \right|\nonumber\\
		&\leq&\frac{1}{2p}\left| \sum_{1\leq i\neq j\leq p} \left( e^{-2f_{i}-2f_{j}}- e^{-2f_{i}^*-2f_{j}^*} \right)  \left(\hat{A}^2_{i,j}-\E\hat{A}^2_{i,j} \right) \right|\nonumber\\
		&&+\frac{1}{2p}\left| \sum_{1\leq i\neq j\leq p} \left( e^{-2f_{i}-2f_{j}}- e^{-2f_{i}^*-2f_{j}^*} \right)  \left(\E\hat{A}^2_{i,j} -\left( A_{i,j}\right) ^2\right) \right|\nonumber\\
		&\leq& C_{2}\left( \frac{\log(np)e_{n}}{p}+\sqrt{\frac{\log(np)}{p}}\sigma_{n}\right)  \left\|\f-\f^*\right\|_{1}\nonumber\\
        &&+ C_2\min\left\{\left\|\f-\f^*\right\|_{2}\frac{\|\Delta_{n}\|_{\mathrm{F}}}{\sqrt{p}},\left\|\f-\f^*\right\|_{1}\frac{\|\Delta_{n}\|_{1}}{p} \right\},  
	\end{eqnarray}
	holds uniformly for all $\f\in \bB_{\infty}\left( \f^*,\alpha_{0}\right)$. Next, we derive the upper bound of $L_{2}$.  Note that   
    \begin{eqnarray}\label{up2}
		L_{2}&=&\frac{1}{p}\left|\sum_{1\leq i\neq j\leq p}  \left(\hat{A}_{i,j}-A_{i,j} \right) \bZ_{i,j}^{\top}\left( \bbeta -\bbeta^* \right) \right|\nonumber\\
            &=&\frac{1}{p}\left|\sum_{k=1}^{}\left( \beta_{k} -\beta^*_{k} \right) \left( \sum_{1\leq i\neq j\leq p}  \left(\hat{A}_{i,j}-A_{i,j} \right)   Z_{i,j,k} \right) \right|\nonumber\\
		&\leq &\frac{\| \sqrt{p}\left( \bbeta-\bbeta^*\right)\|_{1} }{p}\max_{k}\left| \frac{1}{\sqrt{p}}\sum_{1\leq i\neq j\leq p} \left(\hat{A}_{i,j} -A_{i,j}   \right)Z_{i,j,k} \right|\nonumber\\
		&\leq &\frac{\| \sqrt{p}\left( \bbeta-\bbeta^*\right)\|_{1} }{p}\max_{k}\left| \frac{1}{\sqrt{p}}\sum_{1\leq i\neq j\leq p} \left(\hat{A}_{i,j} -\E\hat{A}_{i,j} \right)Z_{i,j,k} \right|\nonumber\\
		&&+\frac{\| \sqrt{p}\left( \bbeta-\bbeta^*\right)\|_{1} }{p}\max_{k}\left|\frac{1}{\sqrt{p}}\sum_{1\leq i\neq j\leq p} \left(\E\hat{A}_{i,j}-A_{i,j} \right)Z_{i,j,k} \right|.
	\end{eqnarray}
        By Bernstein's inequality and Condition (C3), we have that there exist large enough constants $C_{3}>0$, $C_3'>0$ such that, given $\bZ_{i,j}$ and as $p\to \infty$, with probability greater than $1-(np)^{-C_{3}'}$,
        \begin{eqnarray*}
        &&\frac{\| \sqrt{p}\left( \bbeta-\bbeta^*\right)\|_{1} }{p}\max_{k}\left| \frac{1}{\sqrt{p}}\sum_{1\leq i\neq j\leq p} \left(\hat{A}_{i,j} -\E\hat{A}_{i,j} \right) Z_{i,j,k}   \right|\\
           &\leq& C_3\left( \frac{\log(np)e_{n}}{p\sqrt{p}}+\sqrt{\frac{\log(np)}{p}} \sigma_{n}\right) \| \sqrt{p}\left( \bbeta-\bbeta^*\right)\|_{1}.
        \end{eqnarray*}
        On the other hand, with the assumption
\[
c_{1} \le \bx^\top\Big(\frac{1}{p^2}\sum_{1\le i\neq j\le p}\bZ_{i,j}\bZ_{i,j}^\top\Big)\bx \le c_{2}
\quad\text{for all }\|\bx\|_{2}=1,
\]
taking $\bx=e_k$ yields $\sum_{i\neq j}  Z_{i,j,k}^2\le c_2 p^2$ for all $k$ where $e_k$ refers to the  $k$-th column of the $p\times p$ identity matrix. We thus have,
\begin{eqnarray*}
            &&\frac{\| \sqrt{p}\left( \bbeta-\bbeta^*\right)\|_{1} }{p\sqrt{p}}\max_{k}\left|\sum_{1\leq i\neq j\leq p} \left(\E\hat{A}_{i,j}-A_{i,j} \right)Z_{i,j,k} \right|\\
            &\leq&\frac{\| \sqrt{p}\left( \bbeta-\bbeta^*\right)\|_{1} }{p\sqrt{p}}\min\left\{\|\Delta_n\|_{\mathrm{F}}\max_k\Big(\sum_{1\le i\neq j\le p}Z_{i,j,k}^2\Big)^{1/2}, p\|\Delta_n\|_{1}\max_{i,j,k}\Big|Z_{i,j,k}\Big|\right\}\\
            &\lesssim&\frac{\| \sqrt{p}\left( \bbeta-\bbeta^*\right)\|_{1} }{p\sqrt{p}} \min\left\{p\| \Delta_n\|_{\mathrm{F}} ,p\|\Delta_n\|_{1}\right\}\\
            &\leq&\min\left\{\frac{\| \Delta_n\|_{\mathrm{F}} }{\sqrt{p}},\frac{\| \Delta_n\|_{1} }{\sqrt{p}} \right\} \| \sqrt{p}\left( \bbeta-\bbeta^*\right)\|_{1}.
\end{eqnarray*}

\noindent Therefore we have that there exist large enough constants $C_{4}>0$, $C_4'>0$ such that, as $p\to \infty$, with probability greater than $1-(np)^{-C_{4}'}$,
\begin{eqnarray*}
    L_{2}&\leq&C_4\left( \frac{\log(np)e_{n}}{p\sqrt{p}}+\sqrt{\frac{\log(np)}{p}} \sigma_{n}+\frac{\| \Delta_n\|_{\mathrm{F}} }{\sqrt{p}}\right) \| \sqrt{p}\left( \bbeta-\bbeta^*\right)\|_{1} .
\end{eqnarray*}
     {\color{black} Notice that Conditions (C1) and (C2) imply that the elements in $(\bH_{1},\f^*)$ are bounded as well as the initial $\bH^{(0)}$.} Denote $b_{u}:=\max_{j}\{f_{j}^*\}+\alpha_{0}$, $b_{l}:=\min_{j}{(f_{j}^*)}$, $b_{H}:=\|\bH_{1}\|_{\infty}+ \|\bH^{(0)}\|_{\infty}$  and $b_{Z}:=\max_{i\neq j}\sum_{k}|Z_{i,j,k}|$ where $Z_{i,j,k}$ is the $k$th element of $\bZ_{i,j}$. 
Then we have, under  Conditions (C1)  and (C2), there exists a constant $C_{2}>0$ such that, 
	\begin{eqnarray*}
	 L_{3}
        &=&\frac{1}{p} \left|\sum_{1\leq i\neq j\leq p} \left( e^{2f_{i}+2f_{j}}\bZ_{i,j}^{\top}\bbeta -e^{2f_{i}^*+2f_{j}^*}\bZ_{i,j}^{\top} \bbeta^*\right) \left(H^{(0)}_{i,j}- H_{i,j}^*\right)\right|\\
        &\leq&\frac{1}{p} \left|\sum_{1\leq i\neq j\leq p} \left( e^{2f_{i}+2f_{j}}\bZ_{i,j}^{\top}\bbeta -e^{2f_{i}^*+2f_{j}^*} \bZ_{i,j}^{\top}\bbeta \right) \left(H^{(0)}_{i,j}- H_{i,j}^*\right)\right|\\
        &&+ \frac{1}{p} \left|\sum_{1\leq i\neq j\leq p} \left( e^{2f_{i}^*+2f_{j}^*}\bZ_{i,j}^{\top}\bbeta -e^{2f_{i}^*+2f_{j}^*}\bZ_{i,j}^{\top} \bbeta^*\right) \left(H^{(0)}_{i,j}- H_{i,j}^*\right)\right|\\
         &\leq&\frac{1}{p} \left|\sum_{1\leq i\neq j\leq p} \left( e^{2f_{i}+2f_{j}} -e^{2f_{i}^*+2f_{j}^*}\right) \bZ_{i,j}^{\top} \bbeta^*\left(H^{(0)}_{i,j}- H_{i,j}^*\right)\right|\\
        &&+ \frac{1}{p} \left|\sum_{1\leq i\neq j\leq p} \left( \bZ_{i,j}^{\top}\bbeta -\bZ_{i,j}^{\top} \bbeta^*\right) e^{2f_{i}^*+2f_{j}^*}\left(H^{(0)}_{i,j}- H_{i,j}^*\right)\right|\\
		&\leq & 4b_{Z}\|\bbeta^*\|_{\infty}e^{4b_{u}} \min\left\{ \left\|\f-\f^*\right\|_{1}\frac{\left\| \bH^{(0)}- \bH_{1} \right\|_{1}}{p}, \left\|\f-\f^*\right\|_{2}\frac{\left\| \bH^{(0)}- \bH_{1} \right\|_{\mathrm{F}}}{\sqrt{p}}\right\}\\
        &&+b_{Z}e^{4b_{u}}\left\| \sqrt{p}\left( \bbeta^*- \bbeta\right)\right\|_{\infty} \min\left\{ \frac{\left\| \bH^{(0)}- \bH_{1} \right\|_{1}}{p},\frac{\left\| \bH^{(0)}- \bH_{1} \right\|_{\mathrm{F}}}{\sqrt{p}}\right\},
	\end{eqnarray*}
and
	\begin{eqnarray*}
		L_{4}&=&\frac{1}{2p} \left|\sum_{1\leq i\neq j\leq p} \left( e^{2f_{i}+2f_{j}}- e^{2f_{i}^*+2f_{j}^*} \right)  \left( \left(H^{(0)}_{i,j}\right)^2-\left( H_{i,j}^*\right) ^2 \right)\right|\\
        &=&\frac{1}{2p} \left|\sum_{1\leq i\neq j\leq p} \left( e^{2f_{i}+2f_{j}}- e^{2f_{i}^*+2f_{j}^*} \right)  \left(H^{(0)}_{i,j}+ H_{i,j}^* \right) \left(H^{(0)}_{i,j}-H_{i,j}^*  \right)\right|\\
		&\leq & 2b_{H}e^{4b_{u}} \left\|\f-\f^*\right\|_{2}\min\left\{\frac{\left\| \bH^{(0)}- \bH_{1} \right\|_{1}}{\sqrt{p}},\frac{\left\| \bH^{(0)}- \bH_{1} \right\|_{\mathrm{F}}}{\sqrt{p}}\right\}.
	\end{eqnarray*}
 
	Combining inequalities \eqref{up1} and \eqref{up2}, we conclude that   there exist large enough constants $C_{5}>0, C_{5}'>0$, such that, as $p\to \infty$, with probability greater than $1-(np)^{-C_{5}'}$,
	\begin{eqnarray}\label{upbound}
		&&l_{\Delta}\left(\f,\bbeta\right)\nonumber\\
		&\leq&  C_{5}\left( \frac{\log(np)e_{n}}{p}+\sqrt{\frac{\log(np)}{p}}\sigma_{n}\right)  \left\|\f-\f^*\right\|_{1}+ 	 C_5\min\left\{\left\|\f-\f^*\right\|_{2}\frac{\|\Delta_{n}\|_{\mathrm{F}}}{\sqrt{p}},\left\|\f-\f^*\right\|_{1}\frac{\|\Delta_{n}\|_{1}}{p} \right\}  \nonumber\\
        &&+ 	  C_{5}\left( \frac{\log(np)e_{n}}{p\sqrt{p}}+\sqrt{\frac{\log(np)}{p}} \sigma_{n}+\min\left\{\frac{\| \Delta_n\|_{\mathrm{F}} }{\sqrt{p}},\frac{\| \Delta_n\|_{1} }{\sqrt{p}} \right\}\right) \| \sqrt{p}\left( \bbeta-\bbeta^*\right)\|_{1} 
		  \nonumber\\
		&&+  ( 2b_{H}e^{4b_{u}}+4b_{Z}\|\bbeta^*\|_{\infty}e^{4b_{u}})\left\|\f-\f^*\right\|_{2}\min\left\{\frac{\left\| \bH^{(0)}- \bH_{1} \right\|_{1}}{\sqrt{p}},\frac{\left\| \bH^{(0)}- \bH_{1} \right\|_{\mathrm{F}}}{\sqrt{p}}\right\}\nonumber\\
        &&+ b_{Z}e^{4b_{u}}\| \sqrt{p}\left( \bbeta^*- \bbeta\right) \|_{\infty} \min\left\{\frac{\left\| \bH^{(0)}- \bH_{1} \right\|_{1}}{\sqrt{p}},\frac{\left\| \bH^{(0)}- \bH_{1} \right\|_{\mathrm{F}}}{\sqrt{p}}\right\},
	\end{eqnarray}
	holds uniformly for all $\f\in\bB_{\infty}\left( \f^*,\alpha_{0}\right) $, $\bbeta$ and $\bH^{(0)}$. Recall that
	 \begin{eqnarray*}
	 l_{\Delta}\left(\hat{\f},\hat{\bbeta}\right)
	 &\geq&  C_{1} \left( \|\hat{\f}-\f^*\|_{2}^{2}+\|\sqrt{p}\left( \hat{\bbeta}-\bbeta^*\right) \|_{2}^{2}\right).
	 \end{eqnarray*}
	We thus conclude that, as $n,p\to \infty$,  there exists a constant $C_{6}>0$ such that with probability tending to 1,
	\begin{eqnarray}\label{thm1_u}
		&&\frac{1}{\sqrt{p}}\| \left( \hat{\f},\sqrt{p}\hat{\bbeta}\right) -\left(\f^*, \sqrt{p}\bbeta^*\right) \|_{2}\nonumber\\
		&\leq&  C_{6}\left( \frac{\log(np)e_{n}}{p}+\sqrt{\frac{\log(np)}{p}}\sigma_{n}  +\min\left\{\frac{\| \Delta_n\|_{\mathrm{F}} }{p},\frac{\| \Delta_n\|_{1} }{p} \right\} \right)\nonumber \\
        &&+  \left(2b_{H}e^{4b_{u}}+4b_{Z}\|\bbeta^*\|_{\infty}e^{4b_{u}}\right)\min\left\{\frac{\left\| \bH^{(0)}- \bH_{1} \right\|_{1}}{p},\frac{\left\| \bH^{(0)}- \bH_{1} \right\|_{\mathrm{F}}}{p}\right\}  , 
	\end{eqnarray}
	holds uniformly for all $\bH^{(0)}$ such that $\|\bH^{(0)}\|_{\infty} <\infty$.     
    To consider the marginal effect, for any given $i$, define  
	\begin{align*}
		l_{p}\left(f_{i}, \f_{-i},\bbeta\right) &:=l\left(f_{i}, \f_{-i},\bbeta|\hat{\bA},\bH^{(0)}\right)=\frac{1}{2p}\sum_{1\leq j\leq p, j\ne i} \left( \hat{A}_{i,j}e^{-f_{i}-f_{j}}-e^{f_{i}+f_{j}}  \left( \bZ_{i,j}^{\top} \bbeta  +H^{(0)}_{i,j}\right)\right) ^2,\\
		l_{E}\left(f_{i}, \f_{-i},\bbeta\right)  &:=l\left(f_{i}, \f_{-i},\bbeta|\bA,\bH_{1}\right)=\frac{1}{2p}\sum_{1\leq  j\leq p, j\ne i} \left( A_{i,j}e^{-f_{i}-f_{j}}-e^{f_{i}+f_{j}}  \left( \bZ_{i,j}^{\top} \bbeta  +H_{i,j}^*\right)\right) ^2.
	\end{align*}
	By inequality \eqref{low}, we have that, for any $i$,
	\begin{eqnarray}\label{in_low}
		&& l_{p}\left(f_{i}^*, \hat{\f}_{-i},\hat{\bbeta}\right) - l_{p}\left(\hat{f}_{i},\hat{\f}_{-i},\hat{\bbeta} \right)   -\left[ l_{E}\left(f_{i}^*,\f_{-i}^*,\bbeta^*\right)  - l_{E}\left(\hat{f}_{i},\f_{-i}^*,\bbeta^*\right)  \right]\nonumber\\
		&\geq&  l_{E}\left(\hat{f}_{i},\f_{-i}^*,\bbeta^*\right) - l_{E}\left(f_{i}^*,\f_{-i}^*,\bbeta^*\right) \nonumber\\
		&\geq& C_{1}\left|f_{i}^*-\hat{f}_{i} \right|^2.
	\end{eqnarray}
	On the other hand, notice that 
	\begin{eqnarray*}
		&&l_{p}\left(f_{i}^*, \hat{\f}_{-i},\hat{\bbeta}\right) -l_{p}\left(\hat{f}_{i},\hat{\f}_{-i},\hat{\bbeta} \right)   -\left[ l_{E}\left(f_{i}^*,\f_{-i}^*,\bbeta^*\right) - l_{E}\left(\hat{f}_{i},\f_{-i}^*,\bbeta^*\right)   \right]\\
		&\leq&\left|l_{p}\left(f_{i}^*, \hat{\f}_{-i},\hat{\bbeta}\right) -l_{p}\left(\hat{f}_{i},\hat{\f}_{-i},\hat{\bbeta} \right)  - \left[ l_{p}\left(f_{i}^*,\f_{-i}^*,\bbeta^*\right)  - l_{p}\left(\hat{f}_{i},\f_{-i}^*,\bbeta^*\right)  \right]\right|\\
		&&+ \left|l_{p}\left(f_{i}^*,\f_{-i}^*,\bbeta^* \right) -l_{p}\left(\hat{f}_{i}, \f_{-i}^*,\bbeta^*\right) -\left[ l_{E}\left(f_{i}^*,\f_{-i}^*,\bbeta^*\right) - l_{E}\left(\hat{f}_{i},\f_{-i}^*,\bbeta^*\right)   \right]\right|.
	\end{eqnarray*}
   By the definition of $l_{\Delta}(\cdot)$ in \eqref{l_d}, the first term of above inequality equals to $l_{\Delta}\left(\left( \hat{f}_{i}, \f_{-i}^*\right) ,\bbeta^* \right)$. 
	{\color{black} Following \eqref{L1234}, we can decompose  $l_{\Delta}\left(\left( \hat{f}_{i}, \f_{-i}^*\right) ,\bbeta^* \right)$ as: 
    	\begin{eqnarray*} 
		&&l_{\Delta}\left(\left( \hat{f}_{i}, \f_{-i}^*\right) ,\bbeta^* \right) \nonumber\\
		&\leq&\frac{1}{2p}\Bigg| \sum_{1\leq j\leq p, j\ne i} \left( e^{-2\hat{f}_{i}-2f^*_{j}}- e^{-2f_{i}^*-2f_{j}^*} \right)  \left(\hat{A}_{i,j}^2-\left( A_{i,j}\right) ^2 \right)   \nonumber\\
		&&+2\left( e^{2\hat{f}_{i}+2f^*_{j}}\bZ_{i,j}^{\top}\bbeta^* -e^{2f_{i}^*+2f_{j}^*}\bZ_{i,j}^{\top} \bbeta^*\right) \left(H^{(0)}_{i,j}-H_{i,j}^* \right)\nonumber\\
		&&+\left( e^{2\hat{f}_{i}+2f^*_{j}}- e^{2f_{i}^*+2f_{j}^*} \right)  \left( \left(H^{(0)}_{i,j}\right)^2-\left( H_{i,j}^*\right) ^2 \right)\Bigg|\nonumber \\
		&\leq&\frac{1}{2p} \left|\sum_{1\leq j\leq p, j\ne i} \left( e^{-2\hat{f}_{i}-2f^*_{j}}- e^{-2f_{i}^*-2f_{j}^*} \right)  \left(\hat{A}_{i,j}+A_{i,j} \right) \left(\hat{A}_{i,j}-A_{i,j} \right) \right| \nonumber\\
		&&+\frac{1}{p} \left|\sum_{1\leq j\leq p, j\ne i} \left( e^{2\hat{f}_{i}+2f^*_{j}}\bZ_{i,j}^{\top}\bbeta^* -e^{2f_{i}^*+2f_{j}^*}\bZ_{i,j}^{\top} \bbeta^*\right) \left(H^{(0)}_{i,j}-H_{i,j}^* \right)\right|\nonumber\\
		&&+\frac{1}{2p} \left|\sum_{1\leq j\leq p, j\ne i} \left( e^{2\hat{f}_{i}+2f^*_{j}}- e^{2f_{i}^*+2f_{j}^*} \right)  \left((H^{(0)}_{i,j})^2-\left( H_{i,j}^*\right) ^2 \right)\right| \nonumber\\
		&=:&L_{i1}+L_{i2}+L_{i3}.
	\end{eqnarray*}
    Similar to the proof of Theorem \ref{thm1}, we have that there exist large enough positive constants $C_7$ and $C_7'$  such that, with probability greater than $1-(np)^{-C_7'}$, for $L_{i1}$,
   \begin{eqnarray*}
		L_{i1}&=&\frac{1}{2p} \left|\sum_{1\leq j\leq p, j\ne i}\left( e^{-2\hat{f}_{i}-2f^*_{j}}- e^{-2f_{i}^*-2f_{j}^*} \right)  \left(\hat{A}_{i,j}+A_{i,j} \right) \left(\hat{A}_{i,j}-A_{i,j} \right) \right| \nonumber\\
        &\leq&\frac{1}{2p} \left|\sum_{1\leq j\leq p, j\ne i}\left( e^{-2\hat{f}_{i}-2f^*_{j}}- e^{-2f_{i}^*-2f_{j}^*} \right)  \left(\hat{A}_{i,j}+A_{i,j} \right) \left(\E\hat{A}_{i,j}-A_{i,j} \right) \right| \nonumber\\
		&&+\frac{1}{2p} \left|\sum_{1\leq j\leq p, j\ne i}\left( e^{-2\hat{f}_{i}-2f^*_{j}}- e^{-2f_{i}^*-2f_{j}^*} \right)  \left(\hat{A}_{i,j}+A_{i,j} \right) \left(\hat{A}_{i,j}-\E\hat{A}_{i,j} \right) \right|\nonumber\\
        &\leq &C_{7}\left( \frac{\log(np)e_{n}}{p}+\sqrt{\frac{\log(np)}{p}}\sigma_{n}+ \min\left\{\frac{\left\| \bH^{(0)}- \bH_{1} \right\|_{1}}{p},\frac{\left\| \bH^{(0)}- \bH_{1} \right\|_{\mathrm{F}}}{\sqrt{p}}\right\}  \right)  \left|\hat{f}_{i}-f_{i}^*\right|,
	\end{eqnarray*}
    for $L_{i2}$,
    \begin{eqnarray*}
		L_{i2}&=&\frac{1}{p} \left|\sum_{1\neq j\leq p, j\ne i} \left( e^{2\hat{f}_{i}+2f_{j}^*}\bZ_{i,j}^{\top}\bbeta^* -e^{2f_{i}^*+2f_{j}^*}\bZ_{i,j}^{\top} \bbeta^*\right) \left(H^{(0)}_{i,j}- H_{i,j}^*\right)\right|\\
         &=&\frac{1}{p} \left|\sum_{1\neq j\leq p, j\ne i}  \left( e^{2\hat{f}_i+2f_{j}^*} -e^{2{f}_{i}^*+2f_{j}^*}\right) \bZ_{i,j}^{\top} \bbeta^*\left(H^{(0)}_{i,j}- H_{i,j}^*\right)\right|\\
		&\leq & C_{7}|\hat{f}_i-f_i^*|\min\left\{\frac{\left\| \bH^{(0)}- \bH_{1} \right\|_{1}}{p},\frac{\left\| \bH^{(0)}- \bH_{1} \right\|_{\mathrm{F}}}{\sqrt{p}}\right\},
	\end{eqnarray*}
    and for $L_{i3}$,
    \begin{eqnarray*}
		L_{i3}&=&\frac{1}{2p} \left|\sum_{1\leq j\leq p, j\ne i} \left( e^{2\hat{f}_{i}+2f^*_{j}}- e^{2f_{i}^*+2f_{j}^*} \right)  \left((H^{(0)}_{i,j})^2-\left( H_{i,j}^*\right) ^2 \right)\right|\\
        &=&\frac{1}{2p} \left|\sum_{1\leq j\leq p, j\ne i} \left( e^{2\hat{f}_{i}+2f^*_{j}}- e^{2f_{i}^*+2f_{j}^*} \right)  \left(H^{(0)}_{i,j} + H_{i,j}^* \right)\left(H^{(0)}_{i,j} - H_{i,j}^* \right)\right|\\
        &\leq&  C_{7}|\hat{f}_i-f_i^*|\min\left\{\frac{\left\| \bH^{(0)}- \bH_{1} \right\|_{1}}{p},\frac{\left\| \bH^{(0)}- \bH_{1} \right\|_{\mathrm{F}}}{\sqrt{p}}\right\},
	\end{eqnarray*}
    hold uniformly for any $i$ and $\bH^{(0)}$. 
    }
    Consequently,we have that  there exist large enough positive constants $C_8$ and $C_8'$  such that, with probability greater than $1-(np)^{-C_8'}$, 
	\begin{eqnarray*}
		&&\left|l_{p}\left(f_{i}^*,\f_{-i}^*,\bbeta^* \right) -l_{p}\left(\hat{f}_{i}, \f_{-i}^*,\bbeta^*\right) -\left[ l_{E}\left(f_{i}^*,\f_{-i}^*,\bbeta^*\right) - l_{E}\left(\hat{f}_{i},\f_{-i}^*,\bbeta^*\right)   \right]\right|\\
		&\leq&    C_{8}\left( \frac{\log(np)e_{n}}{p}+\sqrt{\frac{\log(np)}{p}}\sigma_{n} +\frac{\|\Delta_{n}\|_{1}}{p} + \frac{\left\| \bH^{(0)}- \bH_{1} \right\|_{1}}{p}  \right)  \left|\hat{f}_{i}-f_{i}^*\right| 
	\end{eqnarray*}
	holds uniformly for any $i$ and $\bH^{(0)}$. 
	On the other hand, there exists a big enough constant $C_9>0$, such that for all $i$, $\bH^{(0)}$,
	\begin{eqnarray}
		&&\left|l_{p}\left(f_{i}^*, \hat{\f}_{-i},\hat{\bbeta}\right) -l_{p}\left(\hat{f}_{i},\hat{\f}_{-i},\hat{\bbeta} \right)  - \left[ l_{p}\left(f_{i}^*,\f_{-i}^*,\bbeta^*\right)  - l_{p}\left(\hat{f}_{i},\f_{-i}^*,\bbeta^*\right)  \right]\right|\nonumber\\
		&\leq&\frac{1}{2p} \sum_{j\neq i;j=1}^{p}\Bigg| \hat{A}_{i,j}^2e^{-2\hat{f}_{i}-2\hat{f}_{j}}\left( e^{2\hat{f}_{i} -2f_{i}^* }-1 \right)+ \left( \bZ_{i,j}^{\top}\hat{\bbeta} +H^{(0)}_{i,j}\right)^2e^{2\hat{f}_{i}+2\hat{f}_{j}}\left( e^{2f_{i}^* -2\hat{f}_{i} }-1 \right)\nonumber\\
		&&-\hat{A}_{i,j}^2e^{-2\hat{f}_{i}-2f_{j}^*}\left( e^{2\hat{f}_{i} -2f_{i}^* }-1 \right)- \left( \bZ_{i,j}^{\top} \bbeta^*  +H^{(0)}_{i,j}\right)^2e^{2\hat{f}_{i}+2f_{j}^*}\left( e^{2f_{i}^* -2\hat{f}_{i} }-1 \right)\Bigg| \nonumber\\
		&=&\frac{1}{2p} \sum_{j\neq i;j=1}^{p}\Bigg| \hat{A}_{i,j}^2e^{-2\hat{f}_{i}-2f_{j}^*} \left(e^{2f_{j}^*-2\hat{f}_{j}}-1\right)\left( e^{2\hat{f}_{i} -2f_{i}^* }-1 \right)\nonumber \\
		&& +\left( \bZ_{i,j}^{\top}\hat{\bbeta} +H^{(0)}_{i,j}\right)^2e^{2\hat{f}_{i}+2f_{j}^*} \left(e^{2\hat{f}_{j}-2f_{j}^*}-1\right)\left( e^{2f_{i}^* -2\hat{f}_{i} }-1 \right)\nonumber \\
		&&+\left( \left(\bZ_{i,j}^{\top}\hat{\bbeta} +H^{(0)}_{i,j}\right)^2-\left( \bZ_{i,j}^{\top} \bbeta^*  +H^{(0)}_{i,j}\right)^2\right)e^{2\hat{f}_{i}+2f_{j}^*}\left( e^{2f_{i}^* -2\hat{f}_{i} }-1 \right)\Bigg|\nonumber\\
		&\leq&\frac{C_{9}}{p} \sum_{j\neq i;j=1}^{p}\left( \left|f_{j}^*-\hat{f}_{j}\right| \left|\hat{f}_{i} -f_{i}^*\right|+ \left\|\bbeta^*-\hat{\bbeta}\right\|_{\infty} \left|\hat{f}_{i} -f_{i}^*\right|\right) \nonumber \\
		&\leq&\frac{C_{9}}{p} \sum_{j\neq i;j=1}^{p}\left(  z_1\left|f_{j}^*-\hat{f}_{j}\right|^2 + z_2 \left|\hat{f}_{i} -f_{i}^*\right|^2+\left\|\bbeta^*-\hat{\bbeta}\right\|_{\infty} \left|\hat{f}_{i} -f_{i}^*\right|\right)\nonumber   \\
		&\leq&\frac{C_{9}}{p}z_{1}\|\f_{-i}^*-\hat{\f}_{-i}\|_{2}^2 + C_{3}z_{2}  \left|\hat{f}_{i} -f_{i}^*\right|^2 +C_{9}\left\|\bbeta^*-\hat{\bbeta}\right\|_{\infty} \left|\hat{f}_{i} -f_{i}^*\right|\nonumber.
	\end{eqnarray}   
	holds for any positive constants $z_1,z_2$ s.t. $z_1z_2\geq 1/4$. 
	Consequently, we have that, with probability greater than $1-(np)^{-C_2'}$,
	\begin{eqnarray}\label{in_up}
		&&l_{p}\left(f_{i}^*, \hat{\f}_{-i},\hat{\bbeta}\right) -l_{p}\left(f_{i},\hat{\f}_{-i},\hat{\bbeta} \right)   -\left[ l_{E}\left(f_{i}^*,\f_{-i}^*,\bbeta^*\right) - l_{E}\left(f_{i},\f_{-i}^*,\bbeta^*\right)   \right]\nonumber\\
		&\leq&(C_{8}+C_{9})\left( \frac{\log(np)e_{n}}{p}+\sqrt{\frac{\log(np)}{p}}\sigma_{n} +\frac{\|\Delta_{n}\|_{1}}{p}  \right) \left|\hat{f}_{i}-f_{i}^*\right|\nonumber\\
        &&+(C_{8}+C_{9}) \frac{\left\| \bH^{(0)}- \bH_{1} \right\|_{1}}{p} \left|\hat{f}_{i}-f_{i}^*\right|\nonumber\\
		&&+\frac{C_{9}}{p}z_{1}\|\f_{-i}^*-\hat{\f}_{-i}\|_{2}^2 + C_{9}z_{2}  \left|\hat{f}_{i} -f_{i}^*\right|^2
	\end{eqnarray}
	holds uniformly for all $i$, $\bH^{(0)}$ and $f_{i}\in\bB_{\infty}\left( f_{i}^*,\alpha_{0}\right)$. 
    Define $d_{n,p}:= \left\|\f^*-\hat{\f} \right\|_{\infty} B^{-1}_{n,p}\left( \bH^{(0)}\right)$
    where
	\begin{align}\label{Bnp}
		B_{n,p}\left( \bH^{(0)}\right):= \frac{\log(np)e_{n}}{p}+\sqrt{\frac{\log(np)}{p}}\sigma_{n} +\frac{\|\Delta_{n}\|_{1}}{p} + \frac{\left\| \bH^{(0)}- \bH_1 \right\|_{1}}{p}. 
	\end{align} 
    {\color{black} To establish the error bound in the \( \ell_{\infty} \) norm, it suffices to show that \( d_{n,p}\) is bounded (in probability). 
    With some abuse of notations, let $j:=\argmax_{k} |f^*_{k}-\hat{f}_{k}|$.} Following inequality \eqref{in_low}, we have  
	\begin{eqnarray}\label{in_low2}
		&& l_{p}\left(f_{j}^*, \hat{\f}_{-j},\hat{\bbeta}\right) - l_{p}\left(\hat{f}_{j},\hat{\f}_{-j},\hat{\bbeta} \right)   -\left[ l_{E}\left(f_{j}^*,\f_{-j}^*,\bbeta^*\right)  - l_{E}\left(\hat{f}_{j},\f_{-j}^*,\bbeta^*\right)  \right]\nonumber\\
		&\geq& C_{1}e^{-2\|\f^*\|_{\infty}}\left|f_{j}^*-\hat{f}_{j} \right|^2\nonumber\\
		&=& C_{1}e^{-2\|\f^*\|_{\infty}}d_{n,p}^2B_{n,p}^2\left( \bH^{(0)}\right).
	\end{eqnarray}
	On the other hand, by inequality \eqref{in_up}, there exist large enough positive constants $C_10$ and $C_{10}'$ which are independent of $\bH^{(0)}$,  such that, with probability greater than $1-(np)^{-C_{10}'}$, 
	\begin{eqnarray*}
		&&l_{p}\left(f_{j}^*, \hat{\f}_{-j},\hat{\bbeta}\right) -l_{p}\left(\hat{f}_{j},\hat{\f}_{-j},\hat{\bbeta} \right)   -\left[ l_{E}\left(f_{j}^*,\f_{-j}^*,\bbeta^*\right) - l_{E}\left(\hat{f}_{j},\f_{-j}^*,\bbeta^*\right)   \right]\\
		&\leq&(C_{8}+C_{9})B_{n,p}\left( \bH^{(0)}\right)  \left|\hat{f}_{j}-f_{j}^*\right| +\frac{C_{9}}{p}z_{1}\|\f_{-j}^*-\hat{\f}_{-j}\|_{2}^2 + C_{9}z_{2}  \left|\hat{f}_{j} -f_{j}^*\right|^2\nonumber\\
		&\leq&C_{10}\left[ d_{n,p}B_{n,p}^2\left( \bH^{(0)}\right) + z_{1}B_{n,p}^2\left( \bH^{(0)}\right) +z_{2} d_{n,p}^2B_{n,p}^2 \left( \bH^{(0)}\right)\right],
	\end{eqnarray*}
    holds for any positive $z_1,z_2$ s.t. $z_1z_2\geq 1/4$. Choosing $z_1=0.5d_{n,p}$ and  $z_2=0.5d_{n,p}^{-1}$, we have that, with probability greater than $1-(np)^{-C'_2}$, 
	\begin{eqnarray}\label{in_up_i}
		&&l_{p}\left(f_{j}^*, \hat{\f}_{-j},\hat{\bbeta}\right) -l_{p}\left(\hat{f}_{j},\hat{\f}_{-j},\hat{\bbeta} \right)   -\left[ l_{E}\left(f_{j}^*,\f_{-j}^*,\bbeta^*\right) - l_{E}\left(\hat{f}_{j},\f_{-j}^*,\bbeta^*\right)   \right]\nonumber\\
		&\leq&C_{10}\left[ d_{n,p}B_{n,p}^2\left( \bH^{(0)}\right) + z_{1}B_{n,p}^2\left( \bH^{(0)}\right) +z_{2} d_{n,p}^2B_{n,p}^2 \left( \bH^{(0)}\right)\right]\nonumber\\
		&=&2C_{10} d_{n,p}B_{n,p}^2\left( \bH^{(0)}\right).
	\end{eqnarray}
	{Combining the inequalities \eqref{in_low2} and \eqref{in_up_i}, we conclude that there exists a constant $C_{11}>0$ which is independent of $n,p$, $\bH^{(0)}$ and the random index $j$, such that $d_{n,p}\leq C_{11}$ hold with probability greater than $1-(np)^{-c_2}$ for some large enough constant $c_2$.}
	Consequently, there exists a constant $C_{12}>0$, we have that as $np \to \infty$, with probability greater than $1-(np)^{-c_2}$,
	\begin{align*}
		\left\|\f^*-\hat{\f} \right\|_{\infty}\leq C_{12}\left( \frac{\log(np)e_{n}}{p}+\sqrt{\frac{\log(np)}{p}}\sigma_{n} +\frac{\|\Delta_{n}\|_{1}}{p} + \frac{\left\| \bH^{(0)}- \bH_1 \right\|_{1}}{p}   \right) .
	\end{align*}
    Note that $\max\{e_{n},\sigma_{n}, \|\Delta_{n}\|_{\infty}, \left\| \bH^{(0)}- \bH_1 \right\|_{\infty}\}  < \alpha_{0}$, we have $(\hat{\f},\hat{\bbeta})=(\wt{\f},\wt{\bbeta})$.
\end{proof}	
\subsection{Proof of Theorem \ref{thm2}}
Before presenting the proof, we first introduce some technical lemmas.
\begin{lemma}\label{low_rank}
        Suppose $\bD^*$ is a low rank symmetric matrix with bounded elements such that rank$(\bD^*)=d'=O(1)$ and $\|\bD^*\|_{\infty}=O(1)$. Let $\lambda_s$ be the $s$-th leading eigenvalue of $\bL^*:=\bD^*$. Define
	\begin{eqnarray*}
		\hat{\bL}&:=&\argmin_{\bL}\frac{1}{2p}\left\|\bD -\bL \right\|_{\mathrm{F}}^{2} + \lambda \|\bL\|_*.
	\end{eqnarray*}
         There exist positive constants $C_1$ and $C_2$ such that by setting {\color{black}$\lambda \asymp \left\| \bD^*- \bD\right\|_{2}/p$} we have: 
        \begin{itemize}
            \item[(1)]  
            \begin{align*}
     	     \|\bL^*-\hat{\bL}\|_{\mathrm{F}} &\leq 
             4\sqrt{2d'} \left\| \bD^*- \bD\right\|_{2};
            \end{align*}
            \item[(2)]  For a given $1\leq r\le s\leq p$, and assume that $\min\{\lambda_{r-1}-\lambda_{r}, \lambda_s-\lambda_{s+1}\} >0$. Let $k=s-r+1$ and $V_\bL \in \mathbb{R}^{p \times d}$ be   the matrix containing the  $r+1$-th to $s$-th  eigenvectors of a given matrix $\bL$. We have
            \begin{align*}
                \|\sin \Theta(V_{\hat{\bL}},  V_{\bL^*})\|_{\mathrm{F}} \leq C_2\frac{k^{1 / 2}\|\bD^*-\bD\|_{2} }{ \min\{\lambda_{r-1}-\lambda_{r}, \lambda_s-\lambda_{s+1}\}};
            \end{align*}
            \item[(3)]  There exists an  orthogonal matrix ${\bf O} \in \mathbb{R}^{d\times d}$ such that
             \begin{align*}
      	\|V_{\hat{\bL}} {\bf O}-V_{\bL^*} \|_{\mathrm{F}} \leq C_2\frac{k^{1 / 2}\|\bD^*-\bD\|_{2} }{ \min\{\lambda_{r-1}-\lambda_{r}, \lambda_s-\lambda_{s+1}\}}.
      \end{align*}  
            \item[(4)] 
            Suppose that $\|V_{\bL^*}\|_{2,\infty} \lesssim            \sqrt{\frac{1}{p}}$,  and $\lambda_{d_{1}}\gtrsim \|\bD-\bD^*\|_2$.
            If the projected row noise satisfies
            \[
            \|(\bD-\bD^*)V_{\bL^*}\|_{2,\infty}
            \lesssim
            \alpha_{n,p}\frac{\|\bD-\bD^*\|_2}{\sqrt p},
            \]
            with some scalar $\alpha_{n,p}$ then
            \[
            \|\hat{\bL}-\bL^*\|_{1}
            \lesssim
            \alpha_{n,p} \|\bD-\bD^*\|_2 .
            \]
        \end{itemize}
\end{lemma}
Next, we proceed to the proof of Theorem \ref{thm2}.
\begin{proof}
	Consider the loss function of $\bL$,
	\begin{eqnarray*}
		l_{L}\left(\bL|\f,\bbeta,\lambda\right)&:=&\frac{1}{p}\sum_{1\leq i,j\leq p} \left( \hat{A}_{i,j}e^{-f_{i}-f_{j}}-e^{f_{i}+f_{j}}  \bZ_{i,j}^{\top} \bbeta  - L_{i,j}\right) ^2 + \lambda \|\bL\|_*\\
		&=&\frac{1}{p}\left\|\bD\left(\f,\bbeta \right) -\bL \right\|_{\mathrm{F}}^{2} + \lambda \|\bL\|_*,
	\end{eqnarray*}
	where $\hat{A}_{i,i}=(\bZ_{i,i})_{k}=0$ for all $(i,j,k)$ and  $\bD\left(\f, \bbeta,\hat{\bA} \right):=\left( \hat{A}_{i,j}e^{-f_{i}-f_{j}}-e^{f_{i}+f_{j}}  \bZ_{i,j}^{\top} \bbeta\right)_{1\leq i,j\leq p}$. For brevity, $l_{L}\left(  \bL|\wt{\f},\wt{\bbeta},\lambda \right)$ is denoted by $l_{L}\left(  \bL\right)$ and $\bD\left(\wt{\f},\wt{\bbeta},\hat{\bA} \right)$ is denoted by $\bD$. We denote the minimizer of  $l_{L}\left(  \bL\right)$ as $\tilde{\bL}$. 
    Moreover, as all the conclusions hold uniformly for all restricted $(\wt{\f},\wt{\bbeta})$, the argument ``uniformly for all restricted $(\wt{\f},\wt{\bbeta})$" are also omitted in what follows. Notate that $\bL^*=\bD\left(\f^*,\bbeta^*,\bA \right):=\bD^*$,  $\bF^{(0)}:= \text{diag}\{e^{\wt{f}_{1}},\cdots, e^{\wt{f}_{p}}\}$,  $\bF^*:= \text{diag}\{e^{f_{1}^*},\cdots, e^{f_{p}^*}\}$, $\bB^{(0)}:=(\bZ_{i,j}^{\top} \wt{\bbeta})_{1\leq i,j \leq p}$ and  $\bB^*:=(\bZ_{i,j}^{\top} \bbeta^*)_{1\leq i,j \leq p}$. Note that for a diagonal matrix $\bY:=\text{diag}\{\by\}=\text{diag}\{y_1,\cdots,y_{p}\}$, we have
    \begin{eqnarray*}
        \|\bX\bY\|_{\mathrm{F}}^2 = \sum_{1\leq i,j\leq p} |X_{i,j}y_{j}|^2\leq \sum_{j=1}^{p} p \|\bX\|^2_{\infty}|y_{j}|^2 = p \|\bX\|_{\infty}^2 \|\by\|_{2}^2=p \|\bX\|_{\infty}^2\|\bY\|_{\mathrm{F}}^2.
    \end{eqnarray*}
    Denote $b_{u}:=\max_{j}{(f_{j}^*)}$, $b_{l}:=\min_{j}{(f_{j}^*)}$  and $b_{Z}:=\max_{i\neq j}\sum_{k}|Z_{i,j,k}|$. We have that there exists a positive constant $C_1$ such that
      \begin{eqnarray}\label{bound_2}
      	&&\|\bD^*-\bD\|_{2}\nonumber\\
        &=& \|(\bF^*)^{-1}\bA(\bF^*)^{-1}- \bF^*\bB^*\bF^*-((\bF^{(0)})^{-1}\hat{\bA}(\bF^{(0)})^{-1}- \bF^{(0)}\bB^{(0)}\bF^{(0)})\|_{2} \nonumber\\
        &\leq& \|(\bF^*)^{-1}\bA(\bF^*)^{-1}-(\bF^{(0)})^{-1} \hat{\bA}(\bF^{(0)})^{-1}\|_{2} +  \| \bF^*\bB^*\bF^*-\bF^{(0)}\bB^{(0)}\bF^{(0)}\|_{2} \nonumber\\
        &\leq& \|(\bF^*)^{-1}\bA(\bF^*)^{-1}-(\bF^*)^{-1} \bA(\bF^{(0)})^{-1}\|_{2} + \|(\bF^*)^{-1} \bA(\bF^{(0)})^{-1}-(\bF^{(0)})^{-1} \bA(\bF^{(0)})^{-1}\|_{2}\nonumber\\
        &&+\|(\bF^{(0)})^{-1} \bA(\bF^{(0)})^{-1}-(\bF^{(0)})^{-1} \hat{\bA}(\bF^{(0)})^{-1}\|_{2} +\| \bF^*\bB^*\bF^*-\bF^*\bB^*\bF^{(0)}\|_{2}\nonumber\\
        &&+\| \bF^*\bB^*\bF^{(0)}-\bF^{(0)}\bB^*\bF^{(0)}\|_{2}+  \| \bF^{(0)}\bB^*\bF^{(0)}-\bF^{(0)}\bB^{(0)}\bF^{(0)}\|_{2} \nonumber\\
        &\leq& \|(\bF^*)^{-1}\bA(\bF^*)^{-1}-(\bF^*)^{-1} \bA(\bF^{(0)})^{-1}\|_{\mathrm{F}} + \|(\bF^*)^{-1} \bA(\bF^{(0)})^{-1}-(\bF^{(0)})^{-1} \bA(\bF^{(0)})^{-1}\|_{\mathrm{F}}\nonumber\\
        &&+\| \bF^*\bB^*\bF^{(0)}-\bF^{(0)}\bB^*\bF^{(0)}\|_{\mathrm{F}} +\| \bF^*\bB^*\bF^*-\bF^*\bB^*\bF^{(0)}\|_{\mathrm{F}}\nonumber\\
        &&+\|(\bF^{(0)})^{-1} \bA(\bF^{(0)})^{-1}-(\bF^{(0)})^{-1} \hat{\bA}(\bF^{(0)})^{-1}\|_{2}+  \| \bF^{(0)}\bB^*\bF^{(0)}-\bF^{(0)}\bB^{(0)}\bF^{(0)}\|_{2} \nonumber\\
        &\leq& \sqrt{p}\|(\bF^*)^{-1}(\bF^*)^2(\bB^*+\bH_{1})\|_{\infty}\|(\bF^*)^2((\bF^*)^{-1}-(\bF^{(0)})^{-1})\|_{\mathrm{F}}\nonumber\\ 
        &&+ \sqrt{p}\|(\bF^*)^2((\bF^*)^{-1}-(\bF^{(0)})^{-1})\|_{\mathrm{F}} \|
        (\bB^*+\bH_{1})(\bF^*)^2(\bF^{(0)})^{-1}\|_{\infty}\nonumber\\
        &&+\sqrt{p}\| \bF^*-\bF^{(0)}\|_{\mathrm{F}}\|\bB^*\bF^{(0)}\|_{\infty} +\sqrt{p}\| \bF^*\bB^*\|_{\infty}\|\bF^*-\bF^{(0)}\|_{\mathrm{F}}\nonumber\\
        &&+\|(\bF^{(0)})^{-1} \bA(\bF^{(0)})^{-1}-(\bF^{(0)})^{-1} \hat{\bA}(\bF^{(0)})^{-1}\|_{2}+  \| \bF^{(0)}\bB^*\bF^{(0)}-\bF^{(0)}\bB^{(0)}\bF^{(0)}\|_{2} \nonumber\\
        &\leq& \sqrt{p}\|(\bF^*)^2((\bF^*)^{-1}-(\bF^{(0)})^{-1})\|_{\mathrm{F}}\|\bF^*(\bB^*+\bH_{1})\|_{\infty}\left(\|(\bF^*)(\bF^{(0)})^{-1}\|_{\infty}+1\right)\nonumber\\
        &&+\sqrt{p}\| \bF^*-\bF^{(0)}\|_{\mathrm{F}}\|\bB^*\bF^{(0)}\|_{\infty} +\sqrt{p}\| \bF^*\bB^*\|_{\infty}\|\bF^*-\bF^{(0)}\|_{\mathrm{F}}\nonumber\\
        &&+\|(\bF^{(0)})^{-1} \bA(\bF^{(0)})^{-1}-(\bF^{(0)})^{-1} \hat{\bA}(\bF^{(0)})^{-1}\|_{2}+  \| \bF^{(0)}\bB^*\bF^{(0)}-\bF^{(0)}\bB^{(0)}\bF^{(0)}\|_{2} \nonumber\\
        &\leq&\sqrt{p}e^{2b_u }\Big(b_Z\|\bbeta^*\|_\infty + \|\bH_{1}\|_{\infty}\Big)\Big(1 + e^{\max_{i}(f_i^*-f_i^{(0)})}\Big)^2\|\wt{\f} - \f^*\|_{2}  \nonumber\\
      	&& + 2\sqrt{p}\,e^{2 b_u}\, b_Z \|\bbeta^*\|_\infty \|\wt{\f} - \f^*\|_{2}+ p\,e^{2 b_u} b_Z \|\wt{\bbeta} - \bbeta^*\|_{2} \nonumber\\
      	&&+  C_{1}\Big( \log(np)\,e_{n} + \sqrt{p\log(np)}\,\sigma_{n} + \|\Delta_n\|_{2} + \max_{i}A_{i,i} +\max_{i}B^*_{i,i} \Big),
      \end{eqnarray}
      Consequently, by Lemma \ref{low_rank}, we have:
      \begin{itemize}
      	\item[(1)] 
      	\begin{eqnarray*}
      		\|\bL^*-\tilde{\bL}\|_{\mathrm{F}} &\lesssim&  \sqrt{p}\left\|\wt{\f}-\f^* \right\|_{2}  + p \left\|\wt{\bbeta}-\bbeta^* \right\|_{\infty}\\
            &&+ \log(np)e_{n} + \sqrt{p\log(np)}\sigma_{n} + \|\Delta_n\|_{2}+O(1);
      	\end{eqnarray*}
        \item[(2)]
        \begin{eqnarray*}
        	\max_{k}\left|\lambda_{k}\left(\bL^* \right)- \lambda_{k}\left(\tilde{\bL}  \right) \right| &\lesssim & \sqrt{p}\left\|\wt{\f}-\f^* \right\|_{2}  + p \left\|\wt{\bbeta}-\bbeta^* \right\|_{\infty}\\
            &&+ \log(np)e_{n} + \sqrt{p\log(np)}\sigma_{n} + \|\Delta_n\|_{2}+O(1);
        \end{eqnarray*} 
        \item [(3)] 
        there exists an orthogonal matrix ${\bf O} \in \mathbb{R}^{d\times d}$ such that
        \begin{eqnarray*}
        	&&\|V_{\tilde{\bL}}  {\bf O}-V_{\bL^*}\|_{\mathrm{F}} \\
            &\lesssim& d^{1 / 2}\frac{\sqrt{p}\left\|\wt{\f}-\f^* \right\|_{2}  + p \left\|\wt{\bbeta}-\bbeta^* \right\|_{\infty}+ \log(np)e_{n} + \sqrt{p\log(np)}\sigma_{n} + \|\Delta_n\|_{2} + O(1)}{ \min\{\lambda_{r-1}\left( \bL^*\right)-\lambda_{r}\left( \bL^*\right), \lambda_s\left( \bL^*\right)-\lambda_{s+1}\left( \bL^*\right)\}}.
        \end{eqnarray*} 
      \end{itemize}
      Here $V_{\bL}\in \mathbb{R}^{p \times s}$ and $ V_{\tilde{\bL}}\in \mathbb{R}^{p \times s}$ are the matrices containing the first $s$  leading eigenvectors of $\bL$ and $\tilde{\bL}$ respectively.
      Next, we derive the error bound of above distance between $\bH_{1}$ and $\tilde{\bH}$. Let $b_{l}^{(0)}:= \min_{i} f_{i}^{(0)}$. Similar to \eqref{bound_2}, we have that
      \begin{eqnarray}\label{bound_h2}
          \|\tilde{\bH} - \bH_{1}\|_{\mathrm{F}} &= & \|(\bF^{(0)})^{-1} \tilde{\bL}(\bF^{(0)})^{-1} - (\bF^*)^{-1}\bL^*(\bF^*)^{-1} \|_{\mathrm{F}}\nonumber\\
          &\leq&\|(\bF^{(0)})^{-1} (\tilde{\bL}-\bL^*)(\bF^{(0)})^{-1}\|_{\mathrm{F}} +\|(\bF^{(0)})^{-1} \bL^*(\bF^{(0)})^{-1}-(\bF^{(0)})^{-1}\bL^*(\bF^*)^{-1} \|_{\mathrm{F}}\nonumber\\
          &&+ \|(\bF^{(0)})^{-1} \bL^*(\bF^{*})^{-1}-(\bF^*)^{-1}\bL^*(\bF^*)^{-1} \|_{\mathrm{F}}\nonumber\\
          &\leq&\|(\bF^{(0)})^{-1} (\tilde{\bL}-\bL^*)(\bF^{(0)})^{-1}\|_{\mathrm{F}} +\sqrt{p}\|(\bF^{(0)})^{-1} \bL^*\|_{\infty}\|(\bF^{(0)})^{-1}-(\bF^*)^{-1} \|_{\mathrm{F}}\nonumber\\
          &&+ \sqrt{p}\|(\bF^{*})^{-1} \bL^*\|_{\infty}\|(\bF^{(0)})^{-1}-(\bF^*)^{-1} \|_{\mathrm{F}}\nonumber\\
          &\leq&\|(\bF^{(0)})^{-1} (\tilde{\bL}-\bL^*)(\bF^{(0)})^{-1}\|_{\mathrm{F}} +\sqrt{p}\|(\bF^{(0)})^{-1}\bF^{*} \bH_{1}\|_{\infty}\|\bF^{*}(\bF^{(0)})^{-1}-\bI \|_{\mathrm{F}}\nonumber\\
          &&+ \sqrt{p}\|\bH_{1}\|_{\infty}\|\bF^{*}(\bF^{(0)})^{-1}-\bI \|_{\mathrm{F}}\nonumber\\
          &\leq& e^{-2b_{l}^{(0)}}\|\tilde{\bL}-\bL^*\| _{\mathrm{F}} + \Big(2 + e^{\max_{i}(f_i^*-f_i^{(0)})}\Big)^2\sqrt{p}\left\|\wt{\f}-\f^* \right\|_{2}.
      \end{eqnarray}
      By Weyl's inequality, we can bound the eigen gap as:
      \begin{eqnarray*}
        	\max_{k}\left|\lambda_{k}\left(\bH_{1} \right)- \lambda_{k}\left(\tilde{\bH}  \right) \right| &\lesssim&  \|\tilde{\bH}-\bH_{1}\|_{2}\lesssim \|\tilde{\bH}-\bH_{1}\|_{\mathrm{F}}\lesssim\|\tilde{\bL}-\bL^*\| _{\mathrm{F}}.
        \end{eqnarray*}
      By inequality \eqref{bound_h2} and the Davis-Kahan Theorem \citep{yu2015useful}, we have that there exists an orthogonal matrix ${\bf O} \in \mathbb{R}^{d\times d}$ such that
      \begin{align*}
      	\|U_{\tilde{\bH}}  {\bf O}-U_{\bH_{1}}\|_{\mathrm{F}} \lesssim \frac{\|\tilde{\bL}-\bL^*\| _{\mathrm{F}}}{\min\{\lambda_{r-1}\left(\bH_{1}\right)-\lambda_{r}\left( \bH_{1}\right), \lambda_s\left( \bH_{1}\right)-\lambda_{s+1}\left( \bH_{1}\right)\}},
      \end{align*}   
      for all $(r,s)$  where $U_{\bH_{1}}$ and $U_{\tilde{\bH}}$ are defined similarly as above.

      Let $\bH_{1}$ has $d_{1,+}$ positive eigenvalue(s) and $d_{1,-}$ negative eigenvalue(s) with $d_{1,+}+d_{1,+}= d_{1}$. Let $U_{\bH_{1}}^{+}$ $\left(U_{\bH_{1}}^{-}\right)$ and $U_{\tilde{\bH}}^{+}$ $\left(U_{\tilde{\bH}}^{-}\right)$ refer to the matrix containing the first $d_{1,+}$ ( last $d_{1,-}$) eigenvectors of  $\bH$ and  $\tilde{\bH}$, respectively. Consider pairs $(r,s)= (1,d_{1,+})$ and $(r,s)= (p-d_{1,-}+1,p)$ and there exist orthogonal matrices ${\bf O}_{+} \in \mathbb{R}^{d_{1,+}\times d_{1,+}}$ and ${\bf O}_{-} \in \mathbb{R}^{d_{1,-}\times d_{1,-}}$ such that
        \begin{eqnarray*}
        	&&      	\max\left\{\|U_{\tilde{\bH}}^{+}  {\bf O}_{+}-U_{\bH_{1}}^{+} \|_{\mathrm{F}},\|U_{\tilde{\bH}}^{-}  {\bf O}_{-}-U_{\bH_{1}}^{-} \|_{\mathrm{F}}\right\} \\
            &\lesssim& \frac{\sqrt{p}\left\|\wt{\f}-\f^* \right\|_{2}  + p \left\|\wt{\bbeta}-\bbeta^* \right\|_{\infty}+ \log(np)e_{n} + \sqrt{p\log(np)}\sigma_{n} + \|\Delta_n\|_{2}}{ \min\{\lambda_{d_{1,+}}\left( \bH_{1}\right),-\lambda_{d_{1,-}} \left( \bH_{1}\right)\}}.
        \end{eqnarray*} 
        Moreover, we have that there exists an orthogonal matrix ${\bf O} \in \mathbb{R}^{d_{1}\times d_{1}}$ ($=\text{diag}({\bf O}_{+},{\bf O}_{-})$) such that
        \begin{eqnarray*}
        	&&      	\|U_{\tilde{\bH}}  {\bf O}-U_{\bH_{1}}\|_{\mathrm{F}} \\
            &\lesssim& \frac{\sqrt{p}\left\|\wt{\f}-\f^* \right\|_{2}  + p \left\|\wt{\bbeta}-\bbeta^* \right\|_{\infty}+ \log(np)e_{n} + \sqrt{p\log(np)}\sigma_{n} + \|\Delta_n\|_{2}}{\lambda_{d_{1}}}.
        \end{eqnarray*} 
        where  $U_{\bH_{1}}$ and $U_{\tilde{\bH}}$ denote the matrices whose columns are the right singular vectors corresponding to the $d_1$ largest singular values of $\bH_{1}$ and $\tilde{\bH}$ respectively.              
      Consequently, we have that there exists positive constant $C_2$ such that 
      \begin{eqnarray}\label{thm2_u}
      		&&\|\bH_{1}-\tilde{\bH}\|_{\mathrm{F}} \nonumber\\
        &\leq&\sqrt{p}e^{2b_u-2b_{l}^{(0)} }\Big(b_Z\|\bbeta^*\|_\infty + \|\bH_{1}\|_{\infty}\Big)\Big(1 + e^{\max_{i}(f_i^*-f_i^{(0)})}\Big)\,\|\wt{\f} - \f^*\|_{2}  \nonumber\\
      	&& + \left(2\sqrt{p}e^{2 b_u-2b_{l}^{(0)}} b_Z \|\bbeta^*\|_\infty+\Big(2 + e^{\max_{i}(f_i^*-f_i^{(0)})}\Big)^2\sqrt{p}\right)\|\wt{\f} - \f^*\|_{2}\nonumber\\
      	&&+ pe^{2 b_u-2b_{l}^{(0)}}b_Z \|\wt{\bbeta} - \bbeta^*\|_{2}+  C_{2}\Big( \log(np)\,e_{n} + \sqrt{p\log(np)}\,\sigma_{n} + \|\Delta_n\|_{2}\Big).
      	\end{eqnarray}
      \begin{itemize}
        \item[(2)]
        \begin{eqnarray*}
        	&&\max_{k}\left|\lambda_{k}\left(\bH_{1} \right)- \lambda_{k}\left(\tilde{\bH}  \right) \right|\\
            &\lesssim& \sqrt{p}\left\|\wt{\f}-\f^* \right\|_{2}  + p \left\|\wt{\bbeta}-\bbeta^* \right\|_{\infty}+ \log(np)e_{n} + \sqrt{p\log(np)}\sigma_{n} + \|\Delta_n\|_{2};
        \end{eqnarray*} 
        \item [(3)]  There exists an orthogonal matrix ${\bf O} \in \mathbb{R}^{d_{1}\times d_{1}}$  such that
        \begin{eqnarray*}
        	&&      	\|U_{\tilde{\bH}}  {\bf O}-U_{\bH_{1}}\|_{\mathrm{F}} \\
            &\lesssim& \frac{\sqrt{p}\left\|\wt{\f}-\f^* \right\|_{2}  + p \left\|\wt{\bbeta}-\bbeta^* \right\|_{\infty}+ \log(np)e_{n} + \sqrt{p\log(np)}\sigma_{n} + \|\Delta_n\|_{2}}{\lambda_{d_{1}}}.
        \end{eqnarray*} 
              where  $U_{\bH_{1}}$ and $U_{\tilde{\bH}}$ denote the matrices whose columns are the right singular vectors corresponding to the $d_1$ largest singular values of $\bH_{1}$ and $\tilde{\bH}$ respectively.

      \end{itemize}       
    \end{proof}
    \subsection{Proof of Lemma \ref{low_rank}}
    \begin{proof}
    Denote $l_{L}\left(\bL\right):=\frac{1}{2p}\left\|\bD -\bL \right\|_{\mathrm{F}}^{2} + \lambda \|\bL\|_*$.  Since $\hat{\bL}$ is the minimizer of 	$l_{L}\left(\bL\right)$, we have that $l_{L}\left(\hat{\bL}\right)\leq l_{L}\left(\bL^*\right)$, and it follows that
	\begin{eqnarray}\label{nu1}
		\lambda\|\hat{\bL}\|_*- \lambda \|\bL^*\|_* &\leq& \frac{1}{2p}\left( \sum_{1\leq i,j\leq p} \left(D_{i,j} - L_{i,j}^*\right) ^2 -\left(D_{i,j} - \hat{L}_{i,j}\right) ^2\right) \nonumber \\
		&=&\frac{1}{2p}\left( \sum_{1\leq i,j\leq p} \left(L_{i,j}^*+\hat{L}_{i,j} - 2D_{i,j}\right) \left(L_{i,j}^*-\hat{L}_{i,j} \right)\right) \nonumber \\
		&=&\frac{1}{2p}\left( \sum_{1\leq i,j\leq p} \left( 2\left( D_{i,j}^*- D_{i,j}\right) -\left( L_{i,j}^*-\hat{L}_{i,j}\right)  \right) \left(L_{i,j}^*-\hat{L}_{i,j} \right)\right) \nonumber\\
		&=&\frac{1}{2p}\left( \sum_{1\leq i,j\leq p} 2\left( D_{i,j}^*- D_{i,j}\right) \left(L_{i,j}^*-\hat{L}_{i,j} \right)- \left(L_{i,j}^*-\hat{L}_{i,j} \right)^2\right) \nonumber\\
		&\leq&\frac{1}{p} \left\| \bD^*- \bD\right\|_{2} \|\bL^*-\hat{\bL}\|_{*} - 	\frac{1}{2p}\|\bL^*-\hat{\bL}\|_{\mathrm{F}}^2.
	\end{eqnarray}
     Set $\lambda \asymp \left\| \bD^*- \bD\right\|_{2}/p$. By inequality \eqref{nu1} and Lemma 6 of \cite{mei2012encoding}, there exists a matrix decomposition $ \bL^*-\hat{\bL} = \bB_L' +\bB_L'' $ such that 
     \begin{eqnarray*}
     	\frac{\|\bL^*-\hat{\bL}\|_{\mathrm{F}}^2 }{2} 	&\leq&\left\| \bD^*- \bD\right\|_{2} \|\bL^*-\hat{\bL}\|_{*} +p\lambda\left( \|\bL^*\|_*-\|\hat{\bL}\|_*\right) \\
     	&\leq&\left\| \bD^*- \bD\right\|_{2} \left(\left\|\bB_L'\right\|_*+ \left\|\bB_L''\right\|_*\right) +p\lambda\left(\left\|\bB_L'\right\|_*-\left\|\bB_L''\right\|_*\right) \\
     	&\leq&2 p\left\| \bD^*- \bD\right\|_{2}\left\|\bB_L'\right\|_*\\
     	&\leq&2\sqrt{2\text{rank}\left( \bL^*\right)}\left\| \bD^*- \bD\right\|_{2}\|\bL^*-\hat{\bL}\|_{\mathrm{F}}.
     \end{eqnarray*}
     In the last step we have used the fact that 
     \begin{align*}
     	\left\|\bB_L'\right\|_*\leq\sqrt{\text{rank}\left( \bB_L'\right) }\left\|\bB_L'\right\|_{\mathrm{F}} \leq \sqrt{2\text{rank}\left( \bL^*\right)}\left\|\bB_L'\right\|_{\mathrm{F}} \leq \sqrt{2\text{rank}\left( \bL^*\right)}\left\|\bL^*-\hat{\bL}\right\|_{\mathrm{F}}.
     \end{align*}
     Then we conclude that 
     \begin{align*}
     	\|\bL^*-\hat{\bL}\|_{\mathrm{F}} &\leq4\sqrt{2\text{rank}\left( \bL^*\right)}\left\| \bD^*- \bD\right\|_{2}.
     \end{align*}      
      By inequality \eqref{bound_2}, and the Davis-Kahan theorem \citep{yu2015useful},
      we have that 
      \begin{align*}
                \|\sin \Theta(U_{\hat{\bL}},  U_{\bL})\|_{\mathrm{F}} \leq C_2\frac{k^{1 / 2}\|\bD^*-\bD\|_{2} }{ \min\{\lambda_{r-1}-\lambda_{r}, \lambda_s-\lambda_{s+1}\}} ,
            \end{align*}
      and there exists an orthogonal matrix ${\bf O} \in \mathbb{R}^{d\times d}$ such that
      \begin{align*}
      	\|U_{\hat{\bL}} {\bf O}-U_{\bL} \|_{\mathrm{F}} \leq \frac{2^{3 / 2}k^{1 / 2}\|\bD^*-\bD\|_{2} }{ \min\{\lambda_{r-1}-\lambda_{r}, \lambda_s-\lambda_{s+1}\}}.
      \end{align*}
      We next prove part (4). Since
\[
\hat{\bL}
=
\argmin_{\bL}
\frac{1}{2p}\|\bD-\bL\|_{\mathrm F}^{2}
+
\lambda\|\bL\|_*,
\]
equivalently,
\[
\hat{\bL}
=
\argmin_{\bL}
\frac{1}{2}\|\bD-\bL\|_{\mathrm F}^{2}
+
p\lambda\|\bL\|_* .
\]
Thus \(\hat{\bL}\) is the spectral soft-thresholding estimator with
threshold \(p\lambda\). Let
\[
\bL^*
=
U_{\bL^*}\Lambda_{\bL^*}U_{\bL^*}^{\top},
\qquad
P_{\bL^*}:=U_{\bL^*}U_{\bL^*}^{\top}.
\]
By the standard first-order expansion of spectral shrinkage around
\(\bL^*\), we have
\[
\hat{\bL}-\bL^*
=
\mathcal P_T(\bD-\bD^*)
-
p\lambda
U_{\bL^*}\operatorname{sgn}(\Lambda_{\bL^*})U_{\bL^*}^{\top}
+
\bR,
\]
where
\[
\mathcal P_T(\bM)
=
P_{\bL^*}\bM+\bM P_{\bL^*}-P_{\bL^*}\bM P_{\bL^*},
\]
and
\[
\|\bR\|_{2,\infty}
\lesssim
\frac{\bigl(\|\bD-\bD^*\|_2+p\lambda\bigr)^2}
{\lambda_{d_1}\sqrt p}.
\]
Since \(p\lambda\asymp \|\bD-\bD^*\|_2\) and
\(\lambda_{d_1}\gtrsim \|\bD-\bD^*\|_2\), it follows that
\[
\|\bR\|_{2,\infty}
\lesssim
\frac{\|\bD-\bD^*\|_2}{\sqrt p}.
\]
We now bound the leading terms. First, by the projected row-noise
condition,
\[
\|(\bD-\bD^*)P_{\bL^*}\|_{2,\infty}
=
\|(\bD-\bD^*)U_{\bL^*}\|_{2,\infty}
\lesssim
\alpha_{n,p}\frac{\|\bD-\bD^*\|_2}{\sqrt p}.
\]
Second, using
\[
\|U_{\bL^*}\|_{2,\infty}\lesssim p^{-1/2},
\]
we have
\[
\|P_{\bL^*}(\bD-\bD^*)\|_{2,\infty}
\le
\|U_{\bL^*}\|_{2,\infty}
\|\bD-\bD^*\|_2
\lesssim
\frac{\|\bD-\bD^*\|_2}{\sqrt p}.
\]
Similarly,
\[
\|P_{\bL^*}(\bD-\bD^*)P_{\bL^*}\|_{2,\infty}
\lesssim
\frac{\|\bD-\bD^*\|_2}{\sqrt p}.
\]
Moreover,
\[
\left\|
p\lambda
U_{\bL^*}\operatorname{sgn}(\Lambda_{\bL^*})U_{\bL^*}^{\top}
\right\|_{2,\infty}
\le
p\lambda\|U_{\bL^*}\|_{2,\infty}
\lesssim
\frac{\|\bD-\bD^*\|_2}{\sqrt p}.
\]
Combining the preceding bounds gives
\[
\|\hat{\bL}-\bL^*\|_{2,\infty}
\lesssim
\alpha_{n,p}\frac{\|\bD-\bD^*\|_2}{\sqrt p}.
\]
Since \(\hat{\bL}-\bL^*\) is symmetric,
\[
\|\hat{\bL}-\bL^*\|_1
=
\max_{1\le i\le p}
\sum_{j=1}^p |\hat L_{i,j}-L^*_{i,j}|.
\]
Therefore,
\[
\|\hat{\bL}-\bL^*\|_1
\le
\sqrt p\|\hat{\bL}-\bL^*\|_{2,\infty}
\lesssim
\alpha_{n,p}\|\bD-\bD^*\|_2.
\]
\end{proof}
    \subsection{Proof of Theorem \ref{thm3}}
    \begin{proof}
        The proof is similar to the proof of Theorems \ref{thm1} and \ref{thm2}. The only difference is the choice of the $\lambda$.
        Recall that $ \bD\left(\f,\bbeta,\bA \right):=\left( A_{i,j}e^{-f_{i}-f_{j}}-e^{f_{i}+f_{j}}  \bZ_{i,j}^{\top} \bbeta\right)_{1\leq i,j\leq p}$ and similarly let $\lambda 
        =\left\| \bD\left(\f^*,\bbeta^*,\bA \right)-\bD\left(\hat{\f},\hat{\bbeta},\hat{\bA} \right)\right\|_{2}/p$. Here the $(\hat{\f},\hat{\bbeta})$ are not given as $\wt{\f},\wt{\bbeta}$ in Theorem \ref{thm2}. Thus, we choose $\lambda \gtrsim \left\| \bD\left(\f^*,\bbeta^*,\bA \right)-\bD\left(\hat{\f},\hat{\bbeta},\hat{\bA} \right)\right\|_{2}/p$ such that by Theorem \ref{thm1}, $\lambda= C D_{n}$ with big enough constant $C>0$.
    \end{proof}
    \subsection{Proof of Proposition \ref{prop3}}
        \begin{proof}
        We may estimate the $\bK^*:=\left(e^{2f_i^*+2f_j^*}\left(H^{(1)}_{i,j}+(\bZ_{i,j})^{\top}\bbeta^* \right) \right)_{1\le i, j\le p}$ through
        \begin{align*}
            \hat{\bK}:=  \argmin_{\bY}\frac{1}{2p}\|\hat{\bA}  -  \bY\|_{\mathrm{F}}^2+ \lambda \|\bY\|_*.
		\end{align*}
        By Lemma \ref{low_rank}, take $\lambda \asymp D_{n}$ and we have that there exist positive constant $c_{1}$ such that  as $p\to \infty$,   with probability greater than $1-(p)^{-c_{1}}$,
            \begin{eqnarray}\label{L2consistency}
                \|\bK^*-\hat{\bK}\|_{\mathrm{F}}&\lesssim&\|\hat{\bA}-\bK^*\|_{2}\\
                &\lesssim&\|\hat{\bA}-\bA\|_{2} +\|\bZ_{\beta}^{(k)}\|_{2} \nonumber \\
                &\lesssim&pD_{n}. \nonumber 
            \end{eqnarray}
        Note that $\hat{\bA}-\hat{\bK}$ is an estimator of $\tilde{\bA}:=\left( e^{2f_i+2f_j}\bZ_{i,j}^{\top}\bbeta \right)_{1\leq i,j\leq p}$. Let
        \begin{eqnarray*}
            \left( \hat{\f}^{(0)},\hat{\bbeta}^{(0)}\right):= \argmin_{\f,\bbeta}\frac{1}{2p}\sum_{1\leq i\neq j\leq p}\left(\left( \hat{A}_{i,j} - \hat{K}^{(0)}_{i,j}\right)e^{-f_{i}-f_{j}}-e^{f_{i}+f_{j}}  \bZ_{i,j}^{\top} \bbeta\right)  ^2 .
        \end{eqnarray*}
        With the results in Corollaries \ref{cor2} and \ref{cor3} in Section \ref{Depen}, we have that there exist big enough positive constant  $c_2$,  for all $(n,p)$ such that  as $np$ is sufficiently large,   with probability greater than $1-(np)^{-c_2}$,
        \begin{align*}
		    	\frac{1}{\sqrt{p}}\| \left( \hat{\f}^{(0)},\sqrt{p}\hat{\bbeta}^{(0)}\right) -\left(\f, \sqrt{p}\bbeta\right) \|_{2}\lesssim \frac{\|\hat{\bA}-\hat{\bK}^{(0)} -\tilde{\bA} \|_{F}}{p} =O\left(D_{n}\right)+O\left(\frac{1}{p}\right).
		  \end{align*}
        Moreover, when Conditions (C5) and (C6) hold, we have that there exist positive constant $c_{3}$ such that  as $p\to \infty$,   with probability greater than $1-(p)^{-c_{3}}$,
        \begin{eqnarray*}
            \|\bK^*-\hat{\bK}\|_{1}&\lesssim&\|\hat{\bA}-\bK^*\|_{2}\\
                &\lesssim&\|\hat{\bA}-\bA\|_{2} +\|\bZ_{\beta}\|_{2}\\
                &\lesssim&pD_{n}.
        \end{eqnarray*}
        Similarly, with the results in Corollaries \ref{cor2} and \ref{cor3} in Section \ref{Depen}, we have that there exist big enough positive constant $c_{4}$,  for all $(n,p)$ such that  as $np$ is sufficiently large,   with probability greater than $1-(np)^{-c_{4}}$,
        \begin{align*}
		    	\| \left( \hat{\f}^{(0)},\hat{\bbeta}^{(0)}\right) -\left(\f, \bbeta\right) \|_{\infty}\leq \frac{\|\hat{\bA}-\hat{\bK}^{(0)} -\tilde{\bA} \|_{F}}{p} =O\left(D_{n}\right)+O\left(\frac{1}{p}\right).
		  \end{align*}
        Moreover, we have that 
        \begin{eqnarray*}
            &&\frac{\|\bH^{(0)}-\bH_1\|_1}{p}\\
            &=& \frac{1}{p} \max_{i} \sum_{j;j\neq i}^{p}\left|H^{(0)}_{i,j}-H_{i,j}^{(1)}\right| \\
            &=&\frac{1}{p} \max_{i} \sum_{j;j\neq i}^{p}\left|e^{-2\hat{f}_i^{(0)}-2\hat{f}_j^{(0)}}\hat{K}_{i,j}-e^{-2f_i^{*}-2 f_j^{*}}K_{i,j}^{*}\right|\\
            &\leq & \frac{1}{p} \max_{i} \sum_{j;j\neq i}^{p}\left|e^{-2\hat{f}_i^{(0)}-2\hat{f}_j^{(0)}}\hat{K}_{i,j}-e^{-2f_i^{*}-2f_j^{*}}\hat{K}_{i,j}\right|+ \frac{1}{p} \max_{i} \sum_{j;j\neq i}^{p}\left|e^{-2f_i^{*}-2f_j^{*}}\hat{K}_{i,j}-e^{-2f_i^{*}-2f_j^{*}}K_{i,j}^{*}\right|\\
            &\lesssim& O\left(D_{n}\right)+O\left(\frac{1}{p}\right).
        \end{eqnarray*} 
        \end{proof}
    \subsection{Proof of Corollary \ref{cor3}}
    \begin{proof}
        We use the results and similar arguments as those in the proof of Theorems \ref{thm1} and \ref{thm1}. 
        First derive the error bound in $\ell_{2}$ norm.
        Replace the inequalities \eqref{up1}, \eqref{up2} and \eqref{Bnp} by
        \begin{eqnarray*}
		L_{1}&=&\frac{1}{2p}\left| \sum_{1\leq i\neq j\leq p} \left( e^{-2f_{i}-2f_{j}}- e^{-2f_{i}^*-2f_{j}^*} \right)  \left(\hat{A}^2_{i,j}-\left( A_{i,j}\right) ^2 \right) \right|\nonumber\\
		&\leq& C_{2}  \frac{\left\|\f-\f^*\right\|_{1}\|\hat{\bA}-\bA\|_{1}}{p},\\
        L_{2}&=&\frac{1}{p}\left|\sum_{1\leq i\neq j\leq p}  \left(\hat{A}_{i,j}-A_{i,j} \right) \bZ_{i,j}^{\top}\left( \bbeta -\bbeta^* \right) \right|\nonumber\\
		&\leq &C_{3}\frac{\| \sqrt{p}\left( \bbeta-\bbeta^*\right)\|_{\infty}\|\hat{\bA}-\bA\|_{1} }{\sqrt{p}},\nonumber
    \end{eqnarray*}
    and remain the other conditions and arguments in the proof of Theorems \ref{thm1}.  Follow the same steps and we can prove
    \begin{eqnarray*}
		\frac{1}{\sqrt{p}}\| \left( \hat{\f},\sqrt{p}\hat{\bbeta}\right) -\left(\f^*, \sqrt{p}\bbeta^*\right) \|_{2}
			&\leq&  C\left( \frac{\| \hat{\bA}- \bA\|_{\mathrm{F}}}{p} + \frac{\left\| \bH^{(0)}- \bH_{1}\right\|_{\mathrm{F}}}{p}\right).
		\end{eqnarray*}
    Replace the inequalities \eqref{in_up} and \eqref{Bnp} by
    \begin{eqnarray*}
		&&l_{p}\left(f_{i}^*, \hat{\f}_{-i},\hat{\bbeta}\right) -l_{p}\left(f_{i},\hat{\f}_{-i},\hat{\bbeta} \right)   -\left[ l_{E}\left(f_{i}^*,\f_{-i}^*,\bbeta^*\right) - l_{E}\left(f_{i},\f_{-i}^*,\bbeta^*\right)   \right]\nonumber\\
		&\leq&(C_{2}+C_{3})\left( \frac{\| \hat{\bA}- \bA\|_{1}}{p} + \frac{\left\| \bH^{(0)}- \bH_{1}\right\|_{\mathrm{F}}}{p}   \right) \left|f_{i}-f_{i}^*\right|\nonumber\\
		&&+\frac{C_{3}}{p}z_{1}\|\f_{-i}^*-\hat{\f}_{-i}\|_{2}^2 + C_{3}z_{2}  \left|f_{i} -f_{i}^*\right|^2\\
        B_{n,p}\left( \bH^{(0)}\right)&:= &\frac{\| \hat{\bA}- \bA\|_{1}}{p} + \frac{\left\| \bH^{(0)}- \bH_{1}\right\|_{\mathrm{F}}}{p}.
	\end{eqnarray*}
	and remain the other conditions and arguments in the proof of Theorems \ref{thm1}.  Follow the same steps and we can prove
    \begin{eqnarray*}
            \left\|\f^*-\hat{\f} \right\|_{\infty}&\leq &C\left( \frac{\| \hat{\bA}- \bA\|_{1}}{p} + \frac{\left\| \bH^{(0)}- \bH_{1}\right\|_{\mathrm{F}}}{p}\right).
		\end{eqnarray*}
    \end{proof}
\end{appendix}
\end{document}